\documentclass[11pt]{amsart}

\usepackage{times}
\usepackage{amssymb}
\usepackage{graphicx,xspace}
\usepackage{epsfig}
\usepackage{enumitem}
\usepackage{xcolor}
\usepackage{tikz}
\usepackage{gensymb}
\usepackage[T1]{fontenc}
\usepackage[utf8]{inputenc}
\usepackage{bm}

\usepackage{hyperref}
\usepackage{todonotes}

\theoremstyle{plain}

\newtheorem{lemma}{Lemma}[section]
\newtheorem{theorem}[lemma]{Theorem}
\newtheorem{corollary}[lemma]{Corollary}
\newtheorem{proposition}[lemma]{Proposition}

\theoremstyle{remark}
\newtheorem*{remark}{Remark}

\newtheorem{definition}[lemma]{Definition}

\newcommand{\RR}{\mathbb{R}}
\newcommand{\C}{\mathbb{C}}
\newcommand{\N}{\mathbb{N}}

\newcommand{\F}{\mathcal{F}}

\newcommand{\Sym}[1]{\mathfrak{S}_{#1}}

\newcommand{\CCC}{\mathcal{C}}
\newcommand{\HHH}{\mathcal{H}}
\newcommand{\Z}{\mathbb{Z}}

\newcommand{\II}{\mathbb{I}}
\newcommand{\OO}{\mathbb{O}}
\newcommand{\QQ}{\mathbb{Q}}
\newcommand{\PP}{\mathbb{P}}

\newcommand{\Pola}{\Lambda_\star^{(\alpha)}}
\newcommand{\Polun}{\Lambda_\star^{(1)}}
\newcommand{\Pol}{\Lambda_\star}
\newcommand{\Polaext}{\left(\Lambda_\star^{(\alpha)}\right)^{\ext}}
\newcommand{\Poluext}{\left(\Lambda_\star^{(1)}\right)^{\ext}}
\newcommand{\lo}{\lambda^{(i)}}
\newcommand{\esper}{\mathbb{E}}
\newcommand{\cl}{\mathcal{C}}
\newcommand{\Young}{\mathcal{Y}}

\DeclareMathOperator{\Tr}{Tr}
\DeclareMathOperator{\Var}{Var}

\DeclareMathOperator{\id}{id}
\DeclareMathOperator{\Id}{Id}

\DeclareMathOperator{\Ch}{Ch}

\DeclareMathOperator{\ext}{ext}
\newcommand{\m}{\mathfrak{h}}

\author[M.~Dołęga]{Maciej Dołęga}
 \address{LIAFA, Universit\'e Paris 7, Case 7014, 75205 Paris Cedex 13,
France \newline \indent Instytut Matematyczny,
Uniwersytet Wrocławski,  \mbox{pl.\ Grunwaldzki~2/4,} 50-384
Wrocław, Poland}
\email{dolega@liafa.univ-paris-diderot.fr}

\author[V.~Féray]{Valentin Féray}

\address{Institut für Mathematik, Universität Zürich, Winterthurerstrasse 190, 8057 Zürich, Switzerland}
 \email{valentin.feray@math.uzh.ch}
 
 \thanks{
MD has been supported by the grant of National Center of Sciences 2011/03/N/ST1/00117 and by ANR project CARTAPLUS ANR 12-JS02-001-01.
VF has been partially supported by ANR project PSYCO ANR-11-JS02-001 and by SNSF grant ``Dual combinatorics of Jack polynomials''.}



\title[Large diagrams and structure constants]
{Gaussian fluctuations of Young diagrams and structure constants of Jack characters}

\begin{document}

\maketitle

\begin{abstract}
In this paper, we consider a deformation of Plancherel measure linked to Jack polynomials.
Our main result is the description of the first and second-order asymptotics of the bulk of a random Young diagram under this distribution, which extends celebrated results of Vershik-Kerov and Logan-Shepp (for the first order asymptotics) and Kerov (for the second order asymptotics).
This gives more evidence of the connection with Gaussian $\beta$-ensemble, 
already suggested by some work of Matsumoto.

Our main tool is a polynomiality result for the structure constant of
some quantities that we call {\em Jack characters}, recently introduced by Lassalle.
We believe that this result is also interested in itself and 
we give several other applications of it.
\end{abstract}

\section{Introduction}

\subsection{Jack deformation of Plancherel measure and random matrices}
\label{subsect:JackPlancherel}

Random partitions occur in mathematics and physics in a wide variety of contexts,
in particular in the Gromov-Witten and Seiberg-Witten theories,
see \cite{OkounkovRandomPartitions} for an introduction to the subject.
Another aspect which makes them attractive
is the link with random matrices.
Namely, some classical models of random matrices have random partition counterparts,
which display the same kind of asymptotic behaviour.

In this paper, we consider the following probability measure
on the set of partitions (or equivalently, Young diagrams) of size $n$:
\begin{equation}
\label{eq:JackMeasure}
\PP_n^{(\alpha)} ( \lambda ) = \frac{\alpha^n n!}{j_\lambda^{(\alpha)}},
\end{equation}
where $j_\lambda^{(\alpha)}$ is a deformation of the square of the
hook products:
\[j_\lambda^{(\alpha)} = \prod_{\Box \in \lambda} \bigg(\alpha a(\Box) + \ell(\Box)
+1\bigg)
\bigg(\alpha a(\Box) + \ell(\Box)+\alpha\bigg).\]
Here, $a(\Box) := \lambda_j - i$ and $\ell(\Box) := \lambda_i' - j$ are respectively the arm and leg length
of the box $\Box = (i,j)$ (the same definitions as in \cite[Chapter I]{Macdonald1995}).

When $\alpha=1$, the measure $\PP_n^{(\alpha)}$ specializes to the well-known \emph{Plancherel measure}
for the symmetric groups.
In general, it is called {\em Jack deformation of Plancherel measure} (or Jack measure for short),
because of its connection with the celebrated Jack polynomials that we shall explain later.
It has appeared recently in
several research papers \cite{BorodinOlshanskiJackZMeasure,FulmanFluctuationChA2,
MatsumotoJackPlancherelFewRows,OlshanskiJackPlancherel,MatsumotoOddJM}
and it is presented as an important area of research in Okounkov's survey on 
random partitions \cite[Section 3.3]{OkounkovRandomPartitions}.

Recall that Plancherel measure has a combinatorial interpretation:
it is the push-forward of the uniform measure on permutations
by Robinson-Schensted algorithm (we keep only the common shape of the
tableaux in the output of RS algorithm).
A similar description holds for Jack measure for $\alpha=2$
(and $\alpha=1/2$ by symmetry): it is the push-forward of
the uniform measure on fixed point free involutions by RS algorithm
(in this case, the resulting diagram has always even parts
and we divide each part by 2) -- see \cite[Section 3]{MatsumotoJackPlancherelFewRows}.

Thus Jack measure can be considered as an interpolation between 
these two combinatorially relevant models of random partitions.

\subsubsection{$\alpha = 1$ case: Plancherel measure and GUE ensemble}

There is a strong connection between Plancherel measure
and the \emph{Gaussian unitary ensemble} (called \emph{GUE}) in random matrix theory.
The Gaussian unitary ensemble is a random Hermitian matrix with independent normal entries.
The probability density function for the eigenvalues of that matrix (of the size $d \times d$) is proportional to the weight
\begin{equation}
\label{eq:BetaDensity}
e^{-\beta/2 \sum x_i^2} \prod_{i < j \leq d}(x_i-x_j)^\beta
\end{equation}
with $\beta = 2$ (see \cite{AndersonGuionnetZeitouni2010}).
Consider the scaled spectral measure of the GUE ensemble
\[ \mu^{(2)}_d := \frac{1}{d}\left(\delta_{x_1} + \cdots + \delta_{x_d}\right),\]
where $x_1 \geq \dots \geq x_d$ are eigenvalues of our random matrix
and $\delta$ is the notation for Dirac operator.
Then the famous \emph{Wigner law} states that, as $d \to \infty$, the spectral measure tends almost surely to a \emph{semicircular law}, i.e. to a probability measure $\mu_{S-C} := \frac{\sqrt{4-x^2}}{2\pi}1_{[-2,2]}(x)dx$ 
supported on the interval $[-2,2]$ (see \cite{AndersonGuionnetZeitouni2010}).
The second order asymptotics is also known and one can observe Gaussian fluctuations around the limiting distribution (see \cite{Johansson1998}).
Informally speaking, looking at the scaled spectral measure of GUE as a generalized function
\[ \mu^{(2)}_d(x) = \frac{1}{d}\left(\delta_{x-x_1} + \cdots + \delta_{x-x_d}\right),\]
we have that
\[ \mu^{(2)}_d(x) \sim \mu_{S-C}(x) + \frac{1}{d}\widetilde{\Delta}^{(2)}(x),\]
as $d \to \infty$, where $\tilde{\Delta^{(2)}}(x)$ is an explicit Gaussian process on the interval $[-2,2]$
with values in the space of generalized functions $(\C^\infty(\RR))'$.

\begin{figure}[tb]
    \begin{tikzpicture}

\begin{scope}[scale=0.5/sqrt(2),rotate=-45,draw=gray]

      \begin{scope}
          \clip[rotate=45] (-2,-2) rectangle (7.5,6.5);
          \draw[thin, dotted, draw=gray] (-10,0) grid (10,10);
          \begin{scope}[rotate=45,draw=black,scale=sqrt(2)]
              \draw[thin, dotted] (0,0) grid (15,15);
          \end{scope}
      \end{scope}

      \draw[->,thick] (-4.5,0) -- (4.5,0)
node[anchor=west,rotate=-45]{\textcolor{gray}{$z$}};
      \foreach \z in { -3, -2, -1, 1, 2, 3}
            { \draw (\z, -2pt) node[anchor=north,rotate=-45]
{\textcolor{gray}{\tiny{$\z$}}} -- (\z, 2pt); }

      \draw[->,thick] (0,-0.4) -- (0,9.5)
node[anchor=south,rotate=-45]{\textcolor{gray}{$t$}};
      \foreach \t in {1, 2, 3, 4, 5, 6, 7, 8, 9}
            { \draw (-2pt,\t) node[anchor=east,rotate=-45]
{\textcolor{gray}{\tiny{$\t$}}} -- (2pt,\t); }

      \begin{scope}[draw=black,rotate=45,scale=sqrt(2)]

          \draw[->,thick] (0,0) -- (6,0) node[anchor=west]{{{$x$}}};
          \foreach \x in {1, 2, 3, 4, 5}
              { \draw (\x, -2pt) node[anchor=north] {{\tiny{$\x$}}} -- (\x,
2pt); }

          \draw[->,thick] (0,0) -- (0,5) node[anchor=south] {{{$y$}}};
          \foreach \y in {1, 2, 3, 4}
              { \draw (-2pt,\y) node[anchor=east] {{\tiny{$\y$}}} -- (2pt,\y); }

          \draw[ultra thick,draw=black] (5.5,0) -- (4,0) -- (4,1) -- (3,1) --
(3,2) -- (1,2) -- (1,3) -- (0,3) -- (0,4.5) ;
          \fill[fill=gray,opacity=0.1] (4,0) -- (4,1) -- (3,1) -- (3,2) -- (1,2)
-- (1,3) -- (0,3) -- (0,0) -- cycle ;

      \end{scope}
 
\end{scope}

\begin{scope}[xshift=7cm, yshift=-0.5cm, scale=0.5]
       \begin{scope}
          \clip (-4.5,0) rectangle (5.5,5.5);
          \draw[thin, dotted] (-6,0) grid (6,6);
          \begin{scope}[rotate=45,draw=gray,scale=sqrt(2)]
              \clip (0,0) rectangle (4.5,5.5);
              \draw[thin, dotted] (0,0) grid (6,6);
          \end{scope}
      \end{scope}

      \draw[->,thick] (-6,0) -- (6,0) node[anchor=west]{$z$};
      \foreach \z in {-5, -4, -3, -2, -1, 1, 2, 3, 4, 5}
            { \draw (\z, -2pt) node[anchor=north] {\tiny{$\z$}} -- (\z, 2pt); }

      \draw[->,thick] (0,-0.4) -- (0,6) node[anchor=south]{$t$};
      \foreach \t in {1, 2, 3, 4, 5}
            { \draw (-2pt,\t) node[anchor=east] {\tiny{$\t$}} -- (2pt,\t); }

\begin{scope}[draw=gray,rotate=45,scale=sqrt(2)]

          \draw[->,thick] (0,0) -- (6,0) node[anchor=west,rotate=45]
{\textcolor{gray}{{$x$}}};
          \foreach \x in {1, 2, 3, 4, 5}
              { \draw (\x, -2pt) node[anchor=north,rotate=45]
{\textcolor{gray}{\tiny{$\x$}}} -- (\x, 2pt); }

          \draw[->,thick] (0,0) -- (0,5) node[anchor=south,rotate=45]
{\textcolor{gray}{{$y$}}};
          \foreach \y in {1, 2, 3, 4}
              { \draw (-2pt,\y) node[anchor=east,rotate=45]
{\textcolor{gray}{\tiny{$\y$}}} -- (2pt,\y); }

          \draw[ultra thick,draw=black] (5.5,0) -- (4,0) -- (4,1) -- (3,1) --
(3,2) -- (1,2) -- (1,3) -- (0,3) -- (0,4.5) ;
          \fill[fill=gray,opacity=0.1] (4,0) -- (4,1) -- (3,1) -- (3,2) -- (1,2)
-- (1,3) -- (0,3) -- (0,0) -- cycle ;

      \end{scope}
\end{scope}

    \end{tikzpicture}

    \caption{Young diagram $\lambda=(4,3,1)$
    and the graph of the associated function $\omega(\lambda)$.}

    \label{FigOmega}
\end{figure}
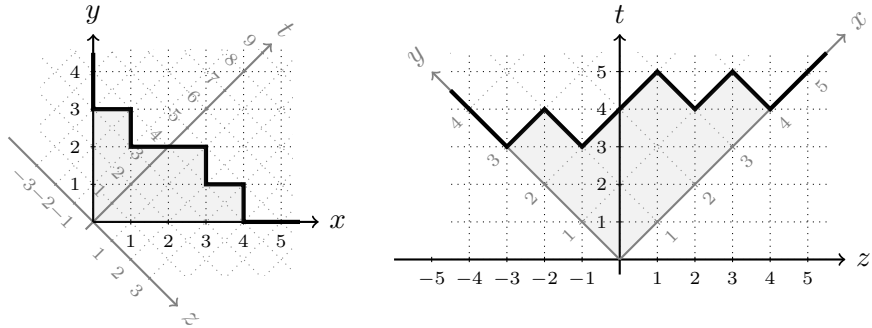

It was also discovered that a similar phenomenon holds for the Plancherel measure.
The first order asymptotics for the Plancherel measure was found by Vershik and Kerov \cite{KerovVershikLimitYD} and, independently, by Logan and Shepp \cite{LoganShepp1977}.
They noticed that a Young diagram $\lambda$ can be encoded by a continuous piecewise-affine function $\omega(\lambda)$ from $\RR$ to $\RR$:
this encoding is represented in Figure \ref{FigOmega} and formally defined in Section \ref{SubsectContinuousYD}.
Then they proved that for appropriately scaled Young diagrams $\overline{\lambda_{(n)}}$ of size $n$
distributed according to the Plancherel measure (the bar encodes the scaling),
one has the convergence in probability
\[ \sup\limits_{x \in \RR} \left|\omega({\overline{\lambda_{(n)}}})(x) - \Omega(x) \right| \stackrel{(P)}{\longrightarrow} 0\]
where the limit shape $\Omega$ is given by
\begin{equation}\label{eq:DefOmega}
    \Omega(x) = \begin{cases}
    |x| & \text{if }|x| \geq 2;\\
    \frac{2}{\pi} \left( x \cdot \arcsin \frac{x}{2} + \sqrt{4-x^2} \right) 
    & \text{otherwise.}
\end{cases}
\end{equation}
There is a strong connection between this limit shape
and the semicircular law $\mu_{S-C}$, namely the so-called 
{\em transition measure} of $\Omega$,
seen as a {\em continuous Young diagram},
is the semicircular law $\mu_{S-C}$
-- see \cite[Section 1.2]{Biane1998}.

The problem of the second order asymptotics was stated as an open question in late seventies
and was solved by Kerov \cite{Kerov1993} who proved that, exactly as in the GUE case, the fluctuations around the limit shape are Gaussian.
Informally Kerov's result can be presented as follows
\[\omega\big({\overline{\lambda_{(n)}}}\big)(x) \sim \Omega(x) + \frac{2}{\sqrt{n}} \Delta^{(1)}_\infty(x)\]
where $\Delta^{(1)}_\infty$ is the Gaussian process on $[-2,2]$ with values in  $(\mathcal{C}^\infty(\RR))'$ defined by:
\[\Delta^{(1)}_\infty(2\cos(\theta)) =\frac{1}{\pi} \sum_{k=2}^\infty \frac{\xi_k}{\sqrt{k}} \sin(k \theta).\]
The detailed proof of this remarkable theorem can be found in \cite{IvanovOlshanski2002}.
Although they are not equal, the Gaussian processes
$\Delta^{(1)}_\infty$ (which describe bulk fluctuations of Young diagram under Plancherel measure)
and $\widetilde{\Delta}^{(2)}$ (which describe bulk fluctuations of the GUE)
have quite similar definition.

Remarkably, the similarity between the Plancherel measure and the GUE ensemble
does not only take place at the level of bulk fluctuations but also "at the edge".
To be more precise, it was proved that 
the distribution of finitely many (properly scaled) first rows of 
random Young diagrams (with respect to the Plancherel measure)
is the same as the distribution of the same number of (properly scaled) largest eigenvalues of the GUE ensemble,
as $n \to \infty$ (see \cite{BaikDeiftJohansson1999, Okounkov2000, BorodinOkounkovOlshanski2000, Johansson2001a, Johansson2001, BaikDeiftRains2001} for details).

\subsubsection{General $\alpha$-case and Gaussian $\beta$-ensembles}

There are two famous analogues of the GUE ensembles in random matrix theory: 
\emph{Gaussian orthogonal ensembles (GOE)} and \emph{Gaussian symplectic ensembles (GSE)} (see \cite{Mehta2004}).
The GOE ensemble (GSE ensemble, respectively) is a random real symmetric matrix 
(complex self-adjoint quaternion matrix, respectively) with independent normal entries
(with mean $0$ and well-chosen variance).
The density function for the distribution of eigenvalues of GOE (GSE, respectively) is, up to normalization, the function given by equation~\eqref{eq:BetaDensity} with parameter $\beta = 1$ ($\beta = 4$, respectively).
Therefore, it is tempting to introduce
\emph{Gaussian $\beta$-ensemble} (G$\beta$E, also called $\beta$-Hermite ensemble)
that  has a distribution density function proportional to \eqref{eq:BetaDensity} for any positive real value of $\beta$.

The G$\beta$E ensembles are well-studied objects.
For the first order of asymptotic behaviour of the spectral measure
\[ \mu^{(\beta)}_d := \frac{1}{d}\left(\delta_{\frac{\beta}{2}x_1} + \cdots + \delta_{\frac{\beta}{2}x_d}\right),\]
where $x_1 \geq \cdots \geq x_d$ are eigenvalues of the G$\beta$E of size $d\times d$, Johansson \cite{Johansson1998} showed that the Wigner law holds, i.e., as $d\to \infty$,
one has the almost-sure convergence
\begin{equation}\label{eq:first-order-GBE}
    \mu^{(\beta)}_d \to \mu_{S_C}.
\end{equation}
A central limit theorem for the G$\beta$E was proved by Dumitriu and Edelman \cite{DumitriuEdelman2006}.
Here, we can observe Gaussian fluctuations around the limit shape, similarly to the GUE case.
Additionally, a surprising phenomenon takes place:
the Gaussian process that describes the second order asymptotic is translated by a deterministic shift,
which disappears for $\beta = 2$ (see \cite{DumitriuEdelman2006} for details).

A natural question is to find a discrete counterpart for G$\beta$E.
Some results of Matsumoto \cite{MatsumotoJackPlancherelFewRows} suggest that
a good candidate for such probability measure on the set of Young diagrams is \emph{Jack measure} given by \eqref{eq:JackMeasure},
where the relation between parameters $\alpha$ and $\beta$ is $\beta = \frac{2}{\alpha}$.
 Matsumoto was studying a restriction of Jack measure to the set of Young diagrams of size $n$ with at most $d$ rows.
 His main result states that the joint distribution of that suitably normalized $d$ rows 
 is the same as the joint distribution of the eigenvalues of $d$-dimensional traceless G$\beta$E with $\beta = \frac{2}{\alpha}$, 
 as $d$ is fixed and $n \to \infty$.

\subsection{Main result}

The main result of this paper is the description of the first and second order asymptotics
of the bulk of Young diagrams under Jack measure.

First, we prove a law of large numbers.
If $\lambda$ is a Young diagram of size $n$, we denote by $\overline{A_\alpha(\lambda)}$ 
the (generalized) Young diagram obtained from $\lambda$ by an horizontal stretching of ratio $\sqrt{\alpha/n}$
and a vertical stretching of ratio $1/\sqrt{n\alpha}$
(a formal definition of generalized Young diagrams is given in Section \ref{SubsectGeneralizedYD}).
We will prove in Section \ref{sect:FirstOrder} the following result.
\begin{theorem}
\label{theo:RYD-limit}
For each $n$, let $\lambda_{(n)}$ be a random Young diagram of size $n$
 distributed according to Jack measure. Then, in probability,
 \[ \lim_{n \to \infty} \sup\limits_{x\in \RR} \left| \omega\left(\overline{A_\alpha(\lambda_{(n)})}\right)(x)
 - \Omega(x) \right|  = 0.\]
\end{theorem}
Note that the limiting curve is exactly the same as in the case $\alpha=1$.

Moreover, we establish a central limit theorem.
Informally, it can be presented as follows:
\begin{equation}\label{eq:informal}
    \omega\left(\overline{A_\alpha(\lambda_{(n)})}\right)(x) \sim \Omega(x) +\frac{2}{\sqrt{n}} \Delta^{(\alpha)}_{\infty}(x),
\end{equation}
where $\Delta^{(\alpha)}_{\infty}$ is the Gaussian process on $[-2,2]$ with values in  $(\mathcal{C}^\infty(\RR))'$ defined by:
\[
\Delta_\infty^{(\alpha)} (2 \cos(\theta) )=
 \frac{1}{\pi} \sum_{k=2}^\infty \frac{\Xi_k}{\sqrt{k}} \sin(k \theta) - \gamma/4 +\gamma \theta/2\pi.
\]
Here, and throughout the paper, $\gamma$ is the difference $\sqrt{\alpha}^{-1}-\sqrt{\alpha}$.
The formal version of this result is stated and proved in Section~\ref{sect:CLT2},
while the explanation for the informal reformulation is given in Section~\ref{subsect:informal}.
Note that the random part of the second order asymptotics is independent of $\alpha$,
while a deterministic term proportional to $\gamma$
(and, hence, vanishing for $\alpha=1$) appears.

Here again, the similarity with the G$\beta$E ensemble is striking.
Indeed, for the bulk of the spectral measure of the G$\beta$E ensemble, we also have the following phenomena:
\begin{itemize}
    \item the first order asymptotics is independent of $\beta$ -- see equation~\eqref{eq:first-order-GBE};
    \item the second order asymptotics is the sum of two terms: 
        a random one and a deterministic one.
        Moreover, the quotient of the deterministic one over the random one is proportional to
        $\gamma$ (see \cite[Theorem 1.2]{DumitriuEdelman2006}).
\end{itemize}
Therefore our result is a new hint towards the deep connection between Jack measure
and the G$\beta$E ensemble.

\subsection{Jack polynomials and Jack measure}
To explain our intermediate results
and the main steps of the proof,
we first need to review the connection between Jack measure and Jack polynomials.

\subsubsection{Jack polynomials}
In a seminal paper \cite{Jack1970/1971},
Jack introduced a family of symmetric functions
$J^{(\alpha)}_\lambda$ depending on an additional parameter
$\alpha$. These functions are now called \emph{Jack polynomials}.
For some special values of $\alpha$, they coincide with some
established families of symmetric functions. Namely, up to multiplicative
constants, for $\alpha=1$, Jack polynomials coincide with Schur polynomials, for
$\alpha=2$, they coincide with zonal polynomials, for $\alpha=1/2$, they
coincide with symplectic zonal polynomials, for $\alpha = 0$, we recover the
elementary symmetric functions and finally their highest degree component in
$\alpha$ are the monomial symmetric functions. Moreover, Jack polynomials for
$\alpha=-(k+ 1)/(r+ 1)$ verify some interesting annihilation conditions \cite{Feigin2002}
and the general $\alpha$ case (with some additional, technical assumptions)
is relevant in Kadell's work in relation with generalizations of
Selberg's integral \cite{Kadell1997}.

Over the time it has been shown that several
results concerning Schur and zonal polynomials can be generalized in a rather
natural way to Jack polynomials (Section (VI.10) of Macdonald's book \cite{Macdonald1995}
gives a few results of this kind), therefore Jack polynomials
can be viewed as a natural
interpolation between several interesting families of symmetric functions.

\subsubsection{A characterization of Jack measure}

Expanding Jack polynomials $J_\lambda^{(\alpha)}$ 
in power-sum symmetric function basis, 
we define the coefficients 
$\theta_{\rho}^{(\alpha)}(\lambda)$ by:
\begin{equation}
\label{eq:jack-characters}
J_\lambda^{(\alpha)}=\sum_{\substack{\rho: \\
|\rho|=|\lambda|}} 
\theta_{\rho}^{(\alpha)}(\lambda)\ p_{\rho}.
\end{equation}
In the case $\alpha=1$, Jack polynomials specialize to
\[J_\lambda^{(1)} = \frac{n!}{\dim(\lambda)} s_\lambda,\] 
where $s_\lambda$ is the Schur function and $\dim(\lambda)$ is
the dimension of the symmetric group representation associated to $\lambda$.
Hence, using Frobenius formula \cite[page 114]{Macdonald1995},
we can express $\theta_{\rho}^{(1)}(\lambda)$ in terms of irreducible character
values of the symmetric group. Namely
\[ \theta_\rho^{(1)}(\lambda)=\frac{n!}{z_\rho} \, \frac{\chi^{\lambda}_\rho}{\dim(\lambda)},\]
where $\chi^{\lambda}_\rho$ is the character of the irreducible
representation indexed by $\lambda$ evaluated on a permutation
of cycle-type $\rho$ and $z_\rho$ is the standard numerical factor $\prod_i i^{m_i(\rho)} m_i(\rho)!$, where $m_i(\rho)$ is the number of parts of $\rho$ equal to $i$.
By analogy to this, we use in the general case the terminology
{\em Jack characters} for $\theta_{\rho}^{(\alpha)}(\lambda)$
(while they do not have any representation theoretical interpretation,
they share a lot of property with characters of the symmetric groups).

The following property,
which corresponds to the case
$\pi=(1^n)$ of \cite[Equation (8.4)]{MatsumotoOddJM},
characterizes Jack measure: 
\begin{equation}
    \esper_{\PP_n^{(\alpha)}}( \theta^{(\alpha)}_\rho(\lambda) ) =
\delta_{\rho,(1^n)},
\label{eq:expect_Jack_character}
\end{equation}
where $\lambda$ is a random Young diagram with $n$ boxes distributed according to
$\PP_n^{(\alpha)}$.

\subsubsection{A central limit theorem for Jack characters}
As in the case $\alpha=1$, an important intermediate result, which may be of independent interest,
is an {\em algebraic} central limit theorem.
Namely we prove that Jack characters for hooks ({\em i.e.} $\rho$ is a hook)
have joint Gaussian fluctuations.

\begin{theorem}
\label{theo:FluctuationsJackCharacters}
Choose a sequence $\left(\Xi_k \right)_{k=2,3,\dots}$
of independent standard Gaussian random variables.
Let $(\lambda_{(n)})_{n \ge 1}$ be a sequence of random Young diagrams of size $n$ distributed according to Jack measure.
Define the random variable
\[W_k(\lambda_{(n)}) = \frac{\sqrt{k} \, \theta^{(\alpha)}_{(k,1^{n-k})}(\lambda_{(n)})}{ n^{k/2}}.\]
Then, as $n \to \infty$, we have:
\[ \left( W_k \right)_{k=2,3,\dots}
\xrightarrow{d} \left( \Xi_k \right)_{k=2,3,\dots}, \]
where $\xrightarrow{d}$ means convergence in distribution of the finite-dimensional laws.
\end{theorem}

In the case $\alpha=1$, this theorem was proved independently by Kerov 
\cite{Kerov1993gaussian,IvanovOlshanski2002}
and Hora \cite{HoraFluctuationsCharacters}.
With the method developed here,
one can even give an upper bound on the speed of convergence of the distribution function:
\begin{theorem}
    We use the same notation as in Theorem~\ref{theo:FluctuationsJackCharacters}.
    Then, for any integer $d \ge2$ and real numbers $x_2,\cdots,x_d$,
    \[\big|\PP ( W_2 \le x_2, \cdots, W_d \le x_d) - \PP (\Xi_2 \le x_2, \cdots, \Xi_d \le x_d) \big|
     = O(n^{-1/4}),\]
    where the constant hidden in the Landau symbol $O$ 
    is {\em uniform} on $x_2,\cdots,x_d$, but depends on $d$.
    \label{theo:SpeedConvergence}
\end{theorem}

In the case $\alpha=1$ and $d=2$, the study of the speed of convergence
of $W_2$ towards a Gaussian variable has been initiated by Fulman
in~\cite{FulmanFluctuationChA2}, who proved a bound of order $n^{-1/4}$.
This bound has then been improved to $n^{-1/2}$ in two different
works \cite{SpeedCvCharacters1,SpeedCvCharacters2}.
Fulman then generalized the $n^{-1/4}$ bound to any $W_i$,
in the cases $\alpha=1$ and $\alpha=2$,
using the representation-theoretical backgrounds for these particular values of $\alpha$,
see~\cite{FulmanSpeedCvRepresentations}.
Some ideas from these papers are fundamental here, as explained below.

The main novelty in Theorem~\ref{theo:SpeedConvergence}
is of course the fact that our result holds
for any value of the parameter $\alpha$.
But, even in the cases $\alpha=1$ and $\alpha=2$,
a bound for the speed convergence of a vector of random variables
and not only for real-valued random variables seems to be new.

\begin{remark}\label{rem:Fluct_NonHook}
We are not able to describe the fluctuations of 
Jack character $\theta^{(\alpha)}_\rho$
when $\rho$ is not a hook.
This is discussed in Subsection~\ref{subsect:Fluct_NonHook}.
\end{remark}

\subsection{Ingredients of the proof}

We shall now say a word on the proof of our main results.
The proof of Theorem \ref{theo:RYD-limit} is easier than
the proof of the fluctuation result, so we will focus here on the second
and compare it to the work of Ivanov, Kerov and Olshanski.

\begin{remark}
While beautiful and elementary, Hora's proof of the central limit theorem
for characters in the case $\alpha=1$ seems very hard to be generalized for 
a generic value of $\alpha$ as it relies from the beginning
on the representation-theoretical background.
\end{remark}

\subsubsection{Polynomial functions}
A central idea in the paper \cite{IvanovOlshanski2002} is to consider 
characters as functions on all Young diagrams by defining:
\[\Ch_{\mu}^{(1)}(\lambda)=
\begin{cases}
    |\lambda|(|\lambda|-1)\cdots (|\lambda|-|\mu|+1) \frac{\chi^\lambda_{\mu\, 1^{|\lambda|-|\mu|}}}
    {\dim(\lambda)} &\text{ if }|\lambda| \ge |\mu|\\
    0 &\text{ if }|\lambda| < |\mu|.
\end{cases}\]
In Sections 1-4 of paper \cite{IvanovOlshanski2002},
the authors prove that the functions $\Ch_{\mu}^{(1)}$ span linearly
a subalgebra of the algebra of functions on {\em all} Young diagrams,
give several equivalent descriptions of this subalgebra
and describe combinatorially the product $\Ch_{\mu}^{(1)} \Ch_{\nu}^{(1)}$
($\Ch_{\mu}^{(1)}$ is denoted by $\bm{p^\#_\mu}$ in~\cite{IvanovOlshanski2002}).
This subalgebra is called the \emph{algebra of
polynomial functions} on the set of Young diagrams (see also \cite{KerovOlshanskiPolFunc})
and denoted here by $\Polun$.

In the general $\alpha$-case, one can define a deformation of the function above as follows:
for an integer partition $\mu$, denote $|\mu|$ its size, $\ell(\mu)$
its length, $m_i(\mu)$ its number of parts of $\mu$ equal to $i$
and $z_\mu$ the standard numerical factor $\prod_i i^{m_i(\mu)} m_i(\mu)!$.
We define
\begin{displaymath}
\label{eq:definition-Jack}
\Ch_{\mu}^{(\alpha)}(\lambda)=
\begin{cases}
\alpha^{-\frac{|\mu|-\ell(\mu)}{2}}
\binom{|\lambda|-|\mu|+m_1(\mu)}{m_1(\mu)}
\ z_\mu \ \theta^{(\alpha)}_{\mu,1^{|\lambda|-|\mu|}}(\lambda)
&\text{if }|\lambda| \ge |\mu| ;\\
0 & \text{if }|\lambda| < |\mu|.
\end{cases}
\end{displaymath}
While Jack characters have been studied for a long time,
the idea, due to Lassalle, to look at them as a function of $\lambda$ as above
is quite recent \cite{Lassalle2008a,Lassalle2009}.
Among other things, he proved that, as in the case $\alpha=1$,
the functions $\Ch_{\mu}^{(\alpha)}$ span linearly a subalgebra
of functions on all Young diagrams,
which has a nice characterization:
we present these results in Section \ref{SectDefKerov},
see in particular Proposition~\ref{prop:Jack-characters-basis}.
This subalgebra is called \emph{algebra of
$\alpha$-polynomial functions} on the set of Young diagrams (see also \cite{KerovOlshanskiPolFunc})
and denoted here by $\Pola$.

As a function on {\em all} Young diagrams,
$\Ch^{(\alpha)}_\mu$ can be restricted to diagrams of size $n$
and hence considered as a random variable in our problem.
It follows directly from equation~\eqref{eq:expect_Jack_character} that
\begin{equation}\label{eq:exp_Ch}
    \esper_{\PP_n^{(\alpha)}}(\Ch^{(\alpha)}_\mu)=\begin{cases}
    n(n-1)\cdots(n-k+1) & \text{if }\mu=1^k \text{ for some }k \leq n,\\
    0   & \text{otherwise.}
\end{cases}
\end{equation}
Throughout the paper, we shall use the standard notation
$(n)_k$ for the falling factorial $n(n-1)\cdots(n-k+1)$. 


\subsubsection{Moment method and structure constants}
Another idea in paper~\cite{IvanovOlshanski2002} is to use the method of moments 
and thus to compute asymptotically (for $h \ge 1$)
\[\esper_{\PP_n^{(1)}} \big( \Ch^{(1)}_{(k)}(\lambda_{(n)})^h \big). \]
Recall that the algebra $\Polun$ has a linear basis given by the family of normalized characters $\left(\Ch^{(1)}_\mu\right)_\mu$.
As the expectation of $\Ch_{\mu}^{(1)}(\lambda_{(n)})$ is particularly simple
-- see equation~\eqref{eq:exp_Ch} --, one can compute expectation of $\big(\Ch^{(1)}_{(k)}\big)^h$ by expanding it on the basis $\Ch_{\mu}^{(1)}$ of $\Polun$.

To do this, the authors of~\cite{IvanovOlshanski2002} need to understand how a product
$\Ch_{\nu}^{(1)} \Ch_{\rho}^{(1)}$ expands on the $\Ch_{\mu}^{(1)}$ basis,
that is they need to study the {\em structure constants} of this basis.
They provide a combinatorial description of these structure constants 
\cite[Proposition 4.5]{IvanovOlshanski2002}.
Unfortunately, this combinatorial description relies on the representation-theoretical
interpretation of $\theta_{\rho}^{(1)}(\lambda)$ and has {\em a priori} no
extension to a general value of $\alpha$.

To overcome this difficulty, we prove that the structure constants
of the $\Ch_{\mu}^{(\alpha)}$ basis depends polynomially on
the auxiliary parameter $\gamma={\sqrt{\alpha}}^{-1}-\sqrt{\alpha}$.
This is a non-trivial result and has other interesting applications
than the study of large Young diagrams under Jack measure.
Therefore, we think that it may be on independent interest and
present it in details in Section~\ref{subsect:Second_Main_Result}
as our second main result.

Our polynomiality result for structure constants
(Theorem \ref{theo:struct-const}) allows us to show that
some properties proved combinatorially in the case $\alpha=1$ still holds
in the general $\alpha$-case
(we will also rely on the case $\alpha=2$, 
which also has some representation-theoretical background,
and use polynomial interpolation).
Our result gives a good estimate of
moments of $\Ch_{(k)}^{(\alpha)}$ of order at most 4 (but not of any order).
Therefore, we can not conclude with the moment method.
To overcome that difficulty, we have to use another ingredient in our proof:
the multivariate Stein's method.

\subsubsection{Multivariate Stein's method and Fulman's construction}
Stein's method is a classical method in probability
to prove convergence in distrbution towards Gaussian or Poisson distribution,
together with bounds on the speed of convergence ;
see the monograph of Stein \cite{Stein}.
To use it, one needs to construct an {\em exchangeable pair} for the relevant random variable.
But, when this pair is constructed, one can prove Gaussian fluctuations,
using only bounds on (mixed conditional) moments of order at most $4$
(while the moment method requires control on moments of all order).

In the framework of Jack characters, an exchangeable pair
has already been built by Fulman to prove a fluctuation result for $\Ch^{(\alpha)}_{(2)}$.
The same construction extends to $\Ch^{(\alpha)}_{(k)}$,
but the analysis of the first moments becomes more tricky, requires new ideas
and heavily relies on our polynomiality result for structure constants.

Let us note that, unlike Fulman's result, our result is a result
of convergence in distribution of {\em vectors} of random variables.
Therefore we need to use a multivariate analog of Stein's classical theorem.
The one recently established by
Reinert and R{\"o}llin~\cite{ReinertRollin2009}
turns out to be suitable for our purpose.\medskip

\subsection{Second main result: polynomiality of structure constants of Jack characters}
\label{subsect:Second_Main_Result}
It follows from the work of Lassalle -- see Proposition~\ref{prop:Jack-characters-basis} --
that the functions $\Ch_{\mu}^{(\alpha)}$ span linearly the algebra of $\alpha$-polynomial functions
denoted by $\Pola$
(when $\mu$ runs over integer partitions of all sizes).
Hence, there exist some rational numbers $g^{(\alpha)}_{\mu,\nu;\pi}$,
depending on $\alpha$ such that
\begin{equation}\label{eq:definition-structure-constants}
    \Ch^{(\alpha)}_\mu \cdot \Ch^{(\alpha)}_\nu =\sum_{\substack{\pi \text{ partition} \\ \text{of any size}}} 
    g^{(\alpha)}_{\mu,\nu;\pi} \Ch^{(\alpha)}_\pi.
\end{equation}
These numbers are often called {\em structure constants} 
of the basis $(\Ch_{\mu}^{(\alpha)})$.
It is a worthy goal to understand them, because they describe
the multiplicative structure of the algebra.

Our second main result is a polynomiality result
for these structure constants with precise bounds on the degree:
let
\begin{align*}
    n_1(\mu) &= |\mu|+\ell(\mu), \\
    n_2(\mu) &= |\mu|-\ell(\mu), \\
    n_3(\mu) &= |\mu|-\ell(\mu)+m_1(\mu).
\end{align*}
Then we have:
\begin{theorem}
\label{theo:struct-const}
    Fix three partitions $\mu$, $\nu$ and $\pi$.
    The structure constant $g^{(\alpha)}_{\mu,\nu;\pi}$ is a polynomial in 
    $\gamma={\sqrt{\alpha}}^{-1}-\sqrt{\alpha}$ with rational coefficients and 
    of degree at most
    \[ \min_{i=1,2,3} n_i(\mu) + n_i(\nu) - n_i(\pi). \]
    Moreover, if $n_1(\mu) + n_1(\nu) - n_1(\pi)$ is even (respectively, odd),
    it is an even (respectively, odd) polynomial.
\end{theorem}

\subsubsection{Other applications of the second main result}
\label{subsect:other_applications}
In addition of the main purpose of this paper,
Theorem \ref{theo:struct-const} can be applied to several different problems from the literature.
\begin{itemize}
    \item It contains a fifty-year old result from Farahat and Higman,
        stating that the structure constants of the center of the symmetric group algebra
        behave polynomially in $n$ \cite{FarahatHigmanCentreQSn}.
    \item A natural analog of the center of the symmetric group algebra
        is the Hecke algebra of the pair $(S_{2n},H_n)$,
        see {\em e.g.} \cite[Section 7.2]{Macdonald1995} for an introduction to it.
        Theorem \ref{theo:struct-const}
        implies also an analog of Farahat and Higman's result
        in this context: up to some explicit normalization factor,
        structure constants also behave polynomially in $n$\footnote{Let
        us also mention the work of Aker and Can \cite{AkerCanFarahatHecke}
        in this direction. Unfortunately, a factor $2^n n!$ is missing in the main result
        and the authors are not able to correct this \cite{CanMistake}}.
        The same result has been proved independently by Tout
        \cite{Tout-Structure-constant-S2n-Hn}.
    \item Goulden and Jackson \cite{GouldenJacksonMatchingJack},
        have defined, using Jack polynomials,
        an interpolation between the structure constants of both algebras.
        By construction, these quantities are rational functions
        in $\alpha$ but they were conjectured to be in fact
        polynomials in $\alpha-1$ with non-negative integer coefficients
        having some combinatorial interpretation
        \cite[Section 4]{GouldenJacksonMatchingJack}.
        Here, we prove that they are {\em polynomials} in $\alpha$
        (or equivalently in $\alpha-1$) with rational coefficients.
        Unfortunately, we are not able 
        to prove either the integrity
        or the positivity of the coefficients.
    \item We are also able to prove two conjectures of 
        Matsumoto~\cite[Section 9]{MatsumotoOddJM},
        arising in the context of matrix integrals (Section 
        \ref{SubsectMatsumoto}).
    \item We can give a short proof of a recent result 
of Vassilieva \cite{VassilievaJack} which generalizes 
a famous result of Dénes \cite{Denes1959} for the number of {\em minimal} factorizations
of a cycle in the symmetric group.
\end{itemize}

The link between our main result and the first two items is presented
in Section \ref{SectSpecialValues},
while the connection with the last three items is explained in Appendix 
\ref{app:Matching-Jack_And_Matsumoto}.

\subsubsection{Tool: Kerov's polynomial for Jack characters}
Let us now say a word about the proof of our second main result.

The algebra $\Pola$, linearly spanned by the functions $\Ch_\mu^{(\alpha)}$,
admits also some interesting algebraic basis: for example the basis
of {\em free cumulants} $(R_k^{(\alpha)})_{k \ge 2}$ -- see Section~\ref{SectDefKerov}.
Thus $\Ch_\mu^{(\alpha)}$ writes uniquely as a polynomial in free cumulants.
As it was first considered by Kerov in the case $\alpha=1$,
it is usually termed {\em Kerov's polynomial}
(or {\em Kerov's expansion} to avoid repetition of the word polynomial).

In \cite{Lassalle2009}, Lassalle has described an inductive algorithm
to compute the coefficients of this expansion.
In this paper, by a careful analysis of Lassalle's algorithm, 
we obtain some polynomiality results (with several bounds on the degree) 
for these coefficients: see 
Propositions~\ref{PropBound1}, \ref{PropBound2} and \ref{PropBound3}.

Clearly, writing some functions in a multiplicative basis may help
to understand how to multiply them
and we can deduce Theorem \ref{theo:struct-const} from
these results on Kerov's polynomials.

While inappropriate to obtain close formulas,
this way of studying structure constants is, as far as we know, original.
Usually, results of structure constants are obtained using their 
combinatorial description of {\em via} representation theory tools 
\cite{FarahatHigmanCentreQSn,GouldenJacksonMapsZonal,IvanovKerovPartialPermutations,GoupilSchaefferStructureCoef,Tout-Structure-constant-S2n-Hn}.

To finish this paragraph, let us mention that there is an appealing
positivity conjecture on Kerov's polynomials
for Jack characters \cite[Conjecture 1.2]{Lassalle2009}.
While we can not solve this conjecture, our analysis of Lassalle's algorithm
gives some partial results: we prove in general the polynomiality of
the coefficients and we compute a few specific values that were
conjectured by Lassalle (see Appendix \ref{app:Kerov_polynomials}).

Another interesting application of our result on Kerov's polynomials is
a new proof of the polynomial dependence of Jack polynomials in
term of Jack parameter $\alpha$, which was an important open problem
in the early nineties.
This is presented in Section \ref{SubsectCoefJackPoly}.

\subsection{An open problem: edge fluctuations of Jack measure}
Another natural question on asymptotics of Jack measure is the behavior
of the first few rows of the Young diagrams.
This kind of results, orthogonal to the ones in this paper,
is called {\em edge fluctuations}.
Our law of large number on the bulk on Young diagrams implies
that, for any fixed positive integer $k$ and real number $C<1$,
\[\PP \left[ \frac{(\lambda_{(n)})_k}{\sqrt{n}} \leq \frac{2\, C}{\sqrt{\alpha}} \right] \to 0,\]
while Lemma~\ref{lem:FiniteSupport} tells us that $(\lambda_{(n)})_k/\sqrt{n}$
exceeds $(2e)/\sqrt{\alpha}$ with exponentially small probability.

A natural conjecture, considering the case $\alpha=1$ where the edge fluctuations
are well described (see \cite{Okounkov2000,BorodinOkounkovOlshanski2000} and references therein)
and the link with $\beta$-ensemble, would be the following:
for each integer $k \ge 1$, the quantity $(\lambda_{(n)})_k/\sqrt{n}$ converges in
probability towards $2/\sqrt{\alpha}$ and the joint vector 
\[ \left[ n^{1/3} \left( \frac{(\lambda_{(n)})_j}{\sqrt{n}} - \frac{2}{\sqrt{\alpha}} \right) \right]_{1 \le j \le k}\]
converges in law towards the $\beta$-Tracy-Widom distribution,
which has been introduced and studied in \cite{BetaEdgeFluctuations} to study
edge fluctuations of $\beta$-ensemble.
Naturally, a similar conjecture can be formulated for the lengths of the first columns of
the Young diagram $\lambda_{(n)}$.

These conjectures hold true for the first row/column in the case $\alpha=1/2$ and $\alpha=2$.
The proof uses the combinatorial interpretation of Jack measure at these particular values of $\alpha$,
using Robinson-Schensted on random fixed point free involutions:
see \cite{BaikRainsMonotomeSubsequenceInvolutions}.

We have not made computer experiments to confirm this conjecture
and let this problem wide open for future research.

\subsection{Outline of the paper}
The paper is organized as follows. Section \ref{SectDefKerov} gives all
definitions and background on Jack characters, free cumulants
and Kerov polynomials.
In Section \ref{SectPolynomial} we prove the polynomiality of
the coefficients of Kerov's polynomials, with bounds on the degrees.
Then our second main result (that is the polynomiality of structure constants,
with precise bound on the degree)
is proved in Section \ref{SectStructureConstants}.
Section \ref{SectSpecialValues} presents technical statements on structure
constants, that will be used in the analysis of large Young diagrams.
The last three sections deal with convergence results for large Young diagrams:
Section \ref{sect:FirstOrder} presents the first order asymptotics, 
 Section~\ref{sect:CLT} gives the central limit theorem for Jack characters and Section \ref{sect:CLT2}
 establishes the Gaussian fluctuations of large random Young diagrams around the limit shape.

Appendices are devoted to partial answers or some solutions
to questions from the literature.

\section{Jack characters and Kerov polynomials}
\label{SectDefKerov}

\subsection{Polynomial functions on the set of Young diagrams}
\label{subsec:polfunct}
The ring $\Polun$ of \emph{polynomial functions on the set of Young diagrams}
(briefly: the ring of \emph{polynomial functions})
has been introduced by Kerov and Olshanski in order to study irreducible
character values of the symmetric groups \cite{KerovOlshanskiPolFunc}.

The first characterization of $\Polun$,
that we shall use as definition, is the following.
\begin{definition}\label{def:polynomial}
    A function $F$ on the set of all Young diagrams belongs to $\Polun$
    if there exists a collection of polynomials
$\left(F_h \in \QQ[\lambda_1,\dots,\lambda_h]\right)_{h > 0}$ such that 
\begin{itemize}
    \item for a diagram $\lambda=(\lambda_1,\dots,\lambda_h)$ of length $h$,
        one has $F(\lambda)=F_h(\lambda_1,\ldots,\lambda_h)$;
    \item each $F_h$ is symmetric in variables
$\lambda_1 - 1, \lambda_2 -2, \ldots ,\lambda_h-h$; 
\item the compatibility
relation
\[F_{h+1}(\lambda_1,\dots,\lambda_h,0) = F_h(\lambda_1,\dots,\lambda_h)\]
holds true for all values of $h$.
\end{itemize}    
\end{definition}
The ring $\Polun$, as defined above, is sometimes called
the ring of \emph{shifted symmetric functions} in $\lambda_1, \lambda_2,\ldots$.
It was first considered by Knop and Sahi \cite{KnopSahiShiftedSym} in a more general context.
While this is not obvious, 
one can prove that $\Ch^{(1)}_\mu$ belongs to $\Polun$.
In fact, one has more \cite[Section 3]{KerovOlshanskiPolFunc}.
\begin{proposition}
    When $\mu$ runs over all partitions,
    the family $(\Ch^{(1)}_\mu)_\mu$ forms a linear basis of $\Polun$.
\end{proposition}

An equivalent description of $\Polun$
can be given using Kerov's interlacing coordinates of a Young diagram.
Recall that the \emph{content} of a box of a Young diagram is $j-i$, where $j$ is its column
index and $i$ its row index and, more generally, the content of a point of the plane is 
the difference between its $x$-coordinate and its $y$-coordinate.
We denote by $\II_\lambda$ the set of contents of the \emph{inner corners} of $\lambda$, that is
corners, at which a box could be added to $\lambda$ to obtain a new diagram of size $|\lambda|+1$.
Similarly, the set $\OO_\lambda$ is defined as the contents of the \emph{outer corners}, that is corners at which a box can be removed from $\lambda$ to obtain a new diagram of size $|\lambda|-1$.
An example is given in Figure \ref{FigCorners} (we use the French convention to draw Young diagrams).

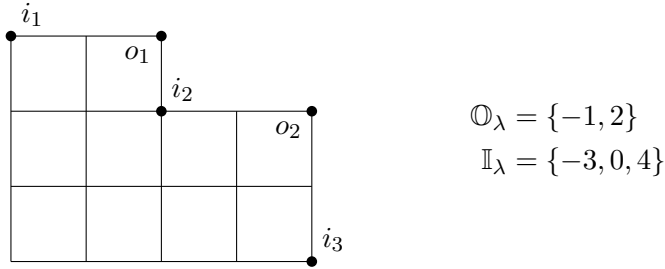
\begin{figure}[tb]
    \begin{minipage}{.6 \linewidth}
    \begin{flushright}
        \begin{tikzpicture}
            \foreach \i in {0,...,2}
                \draw (0,\i) -- (4,\i);
            \draw (0,3) -- (2,3);
            \foreach \j in {0,...,2}
                \draw (\j,0) -- (\j,3);
            \foreach \j in {3,4}
                \draw (\j,0) -- (\j,2);
            \fill (0,3) node[anchor=south west] {$i_1$} circle(2pt);
            \fill (2,2) node[anchor=south west] {$i_2$} circle(2pt);
            \fill (4,0) node[anchor=south west] {$i_3$} circle(2pt);
            \fill (2,3) node[anchor=north east] {$o_1$} circle(2pt);
            \fill (4,2) node[anchor=north east] {$o_2$} circle(2pt);
        \end{tikzpicture}
    \end{flushright}
\end{minipage} \hfill
\begin{minipage}{.35\linewidth}
    \begin{align*}
        \OO_\lambda&=\{-1,2\} \\
        \II_\lambda&=\{-3,0,4\}
    \end{align*}
\end{minipage}
    \caption{A Young diagram with its inner and outer corners (marked respectively with $i$ and $o$).}
    \label{FigCorners}
\end{figure}
If $k$ is a positive integer, one can consider the power-sum symmetric function
$p_k$, evaluated on the difference of alphabets $\II_\lambda - \OO_\lambda$.
By definition, it is a function on Young diagrams given by:
\[\lambda \mapsto p_k(\II_\lambda - \OO_\lambda):= \sum_{i \in \II_\lambda} i^k - 
\sum_{o \in \OO_\lambda} o^k.\]
It can easily be seen that, for any Young diagram, $p_1(\II_\lambda - \OO_\lambda)=0$.
As any symmetric function can be written (uniquely) in terms of $p_k$,
we can define $f(\II_\lambda - \OO_\lambda)$ for any symmetric function $f$ as
follows:
if $a_\rho$ ($\rho$ partition) are the coefficients of the $p$-expansion of $f$,
that is $f=\sum_\rho a_\rho p_{\rho_1} \cdots p_{\rho_\ell}$, then by definition,
\[f(\II_\lambda - \OO_\lambda)=\sum_\rho a_\rho p_{\rho_1}(\II_\lambda - \OO_\lambda)
\cdots p_{\rho_\ell} (\II_\lambda - \OO_\lambda).\]
With this notion of symmetric functions evaluated on difference of alphabets,
the ring $\Polun$ admits the following equivalent description
\cite[Corollary 2.8]{IvanovOlshanski2002}.
\begin{proposition}\label{prop:pk_basis}
    The functions 
    $(\lambda \mapsto p_k(\II_\lambda - \OO_\lambda))_{k \geq 2}$
    form an algebraic basis of $\Polun$.
\end{proposition}
In other terms, any function $F$ in $\Polun$
is equal to $\lambda \mapsto f(\II_\lambda - \OO_\lambda)$,
for some symmetric function $f$.
This symmetric function $f$ is unique up to addition of a multiple of $p_1$.

\subsection{Transition measure and free cumulants}
\label{subsect:TransMeasure}
Kerov \cite{KerovTransitionMeasure} introduced the notion
of \emph{transition measure} of a Young diagram.
This is a probability measure $\mu_\lambda$ on the real line $\RR$ associated to $\lambda$
and defined by its Cauchy transform:
\[G_{\mu_\lambda}(z) = \int_\RR \frac{d\mu_\lambda(x)}{z-x} =
\frac{\prod_{o \in \OO_\lambda} z-o}{\prod_{i \in \II_\lambda} z-i}.\]
In particular, transition measure is supported on $\II_\lambda$.
Besides, its \emph{moment generating series} is given by
\[ \sum_{k \ge 0} M^{(1)}_k(\lambda) \, t^k := \frac{1}{t} \, G_{\mu_\lambda}(1/t)
=\frac{\prod_{o \in \OO_\lambda} 1-o\, t}{\prod_{i \in \II_\lambda} 1-i\, t}, \]
where $M^{(1)}_k(\lambda):= \int_\RR x^k d\mu_\lambda(x)$ is the \emph{$k$-th moment} of
$\mu_\lambda$.
It is easily seen that, for any diagram, $M^{(1)}_0(\lambda)=1$ and $M^{(1)}_1(\lambda)=0$.
This generating series can be rewritten as
\begin{multline*}
    \sum_{k \ge 0} M^{(1)}_k(\lambda) \, t^k = \exp\left( \sum_{i \in \II_\lambda}\sum_{k \ge 1}
\frac{i^k}{k}\, t^k - \sum_{o \in \OO_\lambda}\sum_{k \ge 1} \frac{o^k}{k}\, t^k \right) \\
= \exp \left( \sum_{k \ge 1} \frac{1}{k} p_k(\II_\lambda -\OO_\lambda)\, t^k \right).
\end{multline*}
This implies that $M^{(1)}_k(\lambda)=h_k(\II_\lambda -\OO_\lambda)$,
where $h_k$ is the complete symmetric function of degree $k$;
see \cite[page 25]{Macdonald1995}.
\begin{corollary}
    The family $(M^{(1)}_k)_{k \geq 2}$ forms an algebraic basis of $\Lambda_{(1)}^\star$.
\end{corollary}

We will also be interested in free cumulants $R^{(1)}_k(\lambda)$
of the transition measure $\mu_\lambda$.
They are defined by their generating series
\[K_\lambda(t) =t^{-1} + \sum_{k \ge 1} R^{(1)}_k(\lambda) \, t^{k-1},\]
where $K_\lambda$ is the (formal) compositional inverse of $G_{\mu_\lambda}$.
The fact that $M^{(1)}_1(\lambda)=0$ implies that either $R^{(1)}_1(\lambda)=0$ 
(for all diagrams $\lambda$).

As explained by Lassalle \cite[Section 5]{Lassalle2009},
they can be expressed as
\begin{equation}
\label{eq:FreeCumulantsLassalle}
R^{(1)}_k(\lambda) = e_k^\star(\II_\lambda -\OO_\lambda)
\end{equation}
for some homogeneous symmetric function $e_k^\star$ of degree $k$.
Functions $e_k^\star$ form an algebraic basis of symmetric functions, hence
we have the following corollary of Proposition \ref{prop:pk_basis}.
\begin{corollary}
$(R^{(1)}_k)_{k \geq 2}$ is an algebraic basis of ring of polynomial functions
on the set of Young diagrams.
\end{corollary}

\begin{remark}
Free cumulants are classical objects in free probability theory 
\cite{Voiculescu1986, SpeicherFreeCumulants},
but considering them outside this context may seem strange at first sight.
The relevance of free cumulants of the transition measure
of Young diagrams first appeared in the work of Biane \cite{Biane1998} and
they have played an important role in asymptotic
representation theory since then.
\end{remark}

\subsection{Generalized Young diagrams}\label{SubsectGeneralizedYD}
The second description of $\Polun$ is interesting because
it shows that the value of a polynomial function is defined
on more general objects than just Young diagrams.

\begin{definition}
    A {\em generalized Young diagram} is a broken line
    going from a point $(0,y)$ on the $y$-axis to a point $(x,0)$ on the $x$-axis
    such that every piece is either a horizontal segment from left to right or
    a vertical segment from top to bottom.
\end{definition}
Any Young diagram can be seen as such a broken line: just consider its border.
The notions of inner and outer corners can be easily adapted
to generalized Young diagrams, as well as
the sets $\II_L$ and $\OO_L$ of their contents.
It is illustrated in Figure~\ref{FigGeneralizedYD}.
Note also that the relation $p_1(\II_L-\OO_L)=0$
holds for generalized Young diagrams as well.

\begin{figure}[tb]
    \begin{center}
        \begin{tikzpicture}
            \draw[dashed,gray] grid  (4.5,3.5);
            \draw[->,thick] (-.2,0) -- (4.7,0);
            \draw[->,thick] (0,-.2) -- (0,3.7);
            \draw (0,2)   node[above left,fill=white]   {\tiny $i_1=-2$};  
            \draw (.5,2)  node[above right, fill=white] {\tiny $o_1=-1.5$}; 
            \draw (.5,.5) node[above right, fill=white] {\tiny $i_2=0$};
            \draw (3,.5)  node[above right, fill=white] {\tiny $o_2=2.5$};
            \draw (3,0)   node[below right, fill=white] {\tiny $i_3=3$}; 
            \draw[ultra thick] 
                    (0,2) circle (1.5pt) --
                    (.5,2)  circle (1.5pt) --
                    (.5,.5) circle (1.5pt) --
                    (3,.5)  circle (1.5pt) --
                    (3,0)   circle (1.5pt);
        \end{tikzpicture}
    \end{center}
    \caption{A generalized Young diagram $L$ with the corresponding sets $\OO_L = \{o_1,o_2\}$ and $\II_L = \{i_1, i_2, i_3\}$.}
    \label{FigGeneralizedYD}
\end{figure}
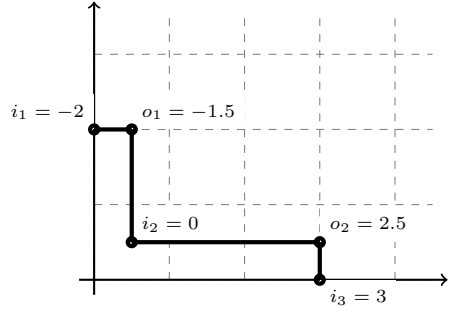

Any polynomial function $F$ on the set of Young diagrams corresponds to the function
\[\lambda \mapsto f(\II_\lambda - \OO_\lambda) \]
for some symmetric function $f$. 
Recall that $f$ is uniquely determined up to addition of a multiple of $p_1$.
Thus, $F$ can be canonically extended to generalized Young diagrams by
setting
\[F(L) = f(\II_L - \OO_L).\]
As the relation $p_1(\II_L-\OO_L)=0$ holds, $F(L)$ is well-defined,
{\em i.e.} it does not depend on the choice of $f$.

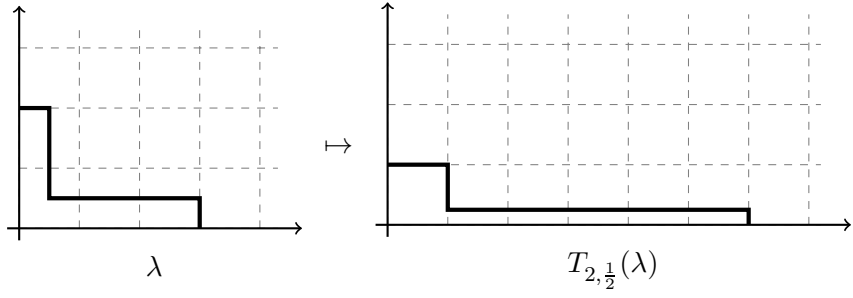
\begin{figure}[tb]
    \[ \begin{array}{c}
        \begin{tikzpicture}[scale=0.8]
            \draw[dashed,gray] grid (4.3,3.4);
            \draw[->,thick] (-0.2,0) -- (4.7,0);
            \draw[->,thick] (0,-0.2) -- (0,3.7);
            \draw[ultra thick] (0,2) 
            -- (.5,2) 
            -- (.5,.5) 
            -- (3,.5) 
            -- (3,0) ;
        \end{tikzpicture}\\
        \lambda
    \end{array}
    \mapsto 
     \begin{array}{c}
        \begin{tikzpicture}[scale=0.8]

            \draw[dashed,gray] grid (7.2,3.5);
            \draw[->,thick] (-0.2,0) -- (7.7,0);
            \draw[->,thick] (0,-0.2) -- (0,3.7);

            \begin{scope}[xscale=2,yscale=.5]
                        \draw[ultra thick] (0,2) 
                        -- (.5,2) 
                        -- (.5,.5) 
                        -- (3,.5) 
                        -- (3,0) ;
            \end{scope}
        \end{tikzpicture}\\
        T_{2,\frac{1}{2}} (\lambda)
    \end{array}\]
    \caption{Example of a Young diagram $\lambda$ on the left and a stretched Young diagram $T_{2,\frac{1}{2}} (\lambda)$ on the right.}
    \label{FigStretching}
\end{figure}

We will be in particular interested in the following generalized Young diagrams.
Let $\lambda$ be a (generalized) Young diagram and $s$ and $t$ two positive real numbers.
$T_{s,t}(\lambda)$ denotes the generalized Young diagram obtained from $\lambda$ by stretching it
horizontally by a factor $s$ and vertically by a factor $t$ in French convention (see Figure~\ref{FigStretching}).
These \emph{anisotropic} Young diagrams have been first considered by Kerov in \cite{KerovAnisotropicYD},
also in the context of Jack polynomials.

In the case $s=t$, we denote by $D_s(\lambda) := T_{s,s}(\lambda)$ the diagram obtained
from $\lambda$ by applying a homothetic transformation of ratio $s$ and we will call it \emph{dilated Young diagram}.
In the case $s = t^{-1} = \sqrt{\alpha}$ for some $\alpha \in \RR_+$ we denote by 
$A_\alpha(\lambda) := T_{\sqrt{\alpha},\sqrt{\alpha}^{-1}} (\lambda)$ the diagram obtained from $\lambda$ by stretching it
horizontally by a factor $\sqrt{\alpha}$ and vertically by a factor $\sqrt{\alpha}^{-1}$. We call it \emph{$\alpha$-anisotropic Young diagram}.

It is easy to check that the sets $\II_{D_s(\lambda)}$ and $\OO_{D_s(\lambda)}$
are obtained from $\II_\lambda$ and $\OO_\lambda$ by multiplying all values by $s$.
In particular, if $F$ is a polynomial function such that the corresponding symmetric
function $f$ is homogeneous of degree $d$, then
\[ \lambda \mapsto F(D_s(\lambda)) = f(\II_{D_s(\lambda)} - \OO_{D_s(\lambda)})
= s^d f(\II_\lambda - \OO_\lambda) = s^d F(\lambda) \]
is also a polynomial function.
Finally, for any fixed $s>0$,
$F$ is a polynomial function if and only if $\lambda \mapsto F(D_s(\lambda))$
is a polynomial function.

\subsection{Continuous Young diagrams}
\label{SubsectContinuousYD}

A generalized Young diagram can also be seen as a function on the real line.
Indeed, if one rotates the zigzag line counterclockwise by $45\degree$
and scale it by a factor $\sqrt{2}$ (so that the new $z$-coordinate corresponds to contents), then it can
be seen as the graph of a piecewise affine continuous function with slope $\pm 1$.
We denote this function by $\omega(\lambda)$.
This definition is illustrated in Figure \ref{FigOmega}.
It is very useful to state convergence results for Young diagrams.

Note that the limiting function $\Omega$ corresponds neither to a real Young diagram,
nor to a generalized Young diagram.
Therefore, it is natural to work with even more general objects than generalized Young diagrams, i.~e.~\emph{continuous Young diagrams}. 
\begin{definition}
We say that a function $\omega: \RR \rightarrow \RR$ is a \emph{continuous Young diagram} if:
\begin{itemize}
\item $\omega$ is Lipshitz continuous function with constant $1$, i.~e.~ for any $x_1, x_2 \in \RR$ $|\omega(x_1) - \omega(x_2)| \leq |x_1 - x_2|$;
\item $\omega(x)-|x|$ is compactly supported.
\end{itemize}
\end{definition}
There is a natural extension for the definitions of transition measure and evaluation of polynomial functions for continuous Young diagrams, see \cite[Section 1.2]{Biane1998}.
However, the general setting will not be relevant in this paper:
we will only need to know that the free
cumulants of the transition measure of $\Omega$ are
\[R_k(\Omega) = \begin{cases}
    1 &\text{ if }k=2 ;\\
    0 &\text{ if }k>2.
\end{cases}\]
This was established by Biane \cite[Section 3.1]{Biane2001}.

\subsection{$\alpha$-polynomial functions}
\begin{definition}
    We say that $F$ is an \emph{$\alpha$-polynomial function} on the set of
(continuous) Young diagrams if 
    \[\lambda \mapsto F ( T_{\alpha^{-1},1}(\lambda) ) \]
    is a polynomial function.
    The set of $\alpha$-polynomial functions is an algebra which will be denoted by $\Pola$.
\end{definition}
Using Definition \ref{def:polynomial}, this means
that the polynomial $F(\alpha^{-1} \lambda_1, \cdots, \alpha^{-1} \lambda_h)$ is 
symmetric in $\lambda_1 -1$, \ldots, $\lambda_h - h$.
Equivalently (by a change of variables), $F$ is symmetric in $\alpha \lambda_1 -1$,
\ldots, $\alpha \lambda_h - h$ or in 
\[\lambda_1 -\frac{1}{\alpha}, \dots, \lambda_h - \frac{h}{\alpha}.\]
The last characterization is the definition of what is usually called an
$\alpha$-shifted symmetric function \cite{OkounkovOlshanskiShiftedJack,Lassalle2008a}.

It would be equivalent to ask in the definition of $\alpha$-polynomial functions that
\[\lambda \mapsto F (A_{\alpha^{-1}}(\lambda) ) \]
is a polynomial function, 
where $A_{\alpha^{-1}}(\lambda) = T_{\sqrt{\alpha}^{-1},\sqrt{\alpha}}(\lambda)$
is an $\alpha$-anisotropic Young diagram.
Indeed, $T_{\sqrt{\alpha}^{-1},\sqrt{\alpha}}(\lambda)$ is a dilatation of
$T_{\alpha^{-1},1}(\lambda)$ and the property of {\em being polynomial} is invariant
by dilatation of the argument.

Therefore, the $\alpha$-anisotropic moments and free cumulants defined by
\begin{align*}
    M_k^{(\alpha)}(\lambda)&:= M^{(1)}_k \left( A_{\alpha}(\lambda) \right),\\
    R_k^{(\alpha)}(\lambda) &:= R^{(1)}_k \left( A_{\alpha}(\lambda) \right), \\
\end{align*}
are $\alpha$-polynomial.
Moreover, the families $(M_k^{(\alpha)})_{k\geq 2}$ and $(R_k^{(\alpha)})_{k\geq 2}$ are algebraic bases of the algebra $\Pola$ of
$\alpha$-polynomial functions.

The following property is due to Lassalle (under a slightly different form).
\begin{proposition}\label{prop:Jack-characters-basis}
When $\mu$ runs over all partitions,
 Jack characters $\left(\Ch^{(\alpha)}_\mu\right)_\mu$ form a linear basis of
the algebra of $\alpha$-polynomial functions.
\end{proposition}
\begin{proof}
    In \cite[Section 3]{Lassalle2008a}, Lassalle builds a
    linear isomorphism
    $\lambda \mapsto \lambda^\#$ between symmetric functions and 
    $\alpha$-shifted symmetric functions.
    Then he shows \cite[Proposition 2]{Lassalle2008a} that,
    for any two partitions $\lambda$ and $\mu$ with $|\lambda| \ge |\mu|$,
    \[\alpha^{(|\mu|-\ell(\mu))/2}  p_\mu^\#(\lambda) = \Ch^{(\alpha)}_\mu(\lambda).\]
    It is straight-forward to check that, for $|\lambda| < |\mu|$,
    both sides of equality above are equal to $0$.
    Hence, as functions on all Young diagrams,
    $\Ch^{(\alpha)}_\mu$ is equal, up to a scalar multiple, to $p_\mu^\#$.
    The facts that $p_\mu$ is a basis of the symmetric function ring and
    $f\mapsto f^\#$ a linear isomorphism conclude the proof.
\end{proof}
In particular, they are $\alpha$-polynomial functions and can be expressed in
terms of the algebraic bases above.

\begin{proposition}
    Let $\mu$ be a partition and $\alpha >0$ a fixed real number.
    There exist unique polynomials $L_\mu^{(\alpha)}$ and $K_\mu^{(\alpha)}$
    such that, for every $\lambda$,
    \begin{align*}
        \Ch^{(\alpha)}_\mu(\lambda) &= L_\mu^{(\alpha)}\left(M_2^{(\alpha)}(\lambda),
            M_3^{(\alpha)}(\lambda),\cdots \right); \\
        \Ch^{(\alpha)}_\mu(\lambda) &= K_\mu^{(\alpha)}\left(R_2^{(\alpha)}(\lambda),
    R_3^{(\alpha)}(\lambda),\cdots \right).
\end{align*}
\end{proposition}
The polynomials $K_\mu^{(\alpha)}$ have been introduced by Kerov in the case
$\alpha=1$ \cite{Kerov2000talk,Biane2003} and by Lassalle in the general case \cite{Lassalle2009} and they are called \emph{Kerov polynomials}.
Once again, our normalizations are different from his.
We will explain this choice later.

From now on, when it does not create any confusion, we suppress the superscript
$(\alpha)$.

We present a few examples of polynomials $L_{\mu}$ and $K_{\mu}$
(in particular the case of a
one-part partition $\mu$ of length lower than $6$).
This data has been computed using the one given in \cite[page 2230]{Lassalle2009}.
Recall that we set $\gamma:=\frac{1-\alpha}{\sqrt{\alpha}}$.
\begin{align*}
    L_{(1)} &= M_2, \\
    L_{(2)} &= M_3 + \gamma M_2, \\
    L_{(3)} &= M_4 - 2M_2^2 + 3\gamma M_3 + (1 + 2\gamma^2)M_2, \\
    L_{(4)} &= M_5 - 5M_3M_2 + 6\gamma M_4 - 11\gamma M_2^2 + (5 + 11\gamma^2)M_3 + (7\gamma +
6\gamma^3)M_2, \\
L_{(5)} &= M_6 - 6M_4M_2 -3M_3^2 + 7M_2^3 + 10\gamma M_5 - 45\gamma M_3M_2 \\
&\qquad + (15+35\gamma^2)M_4 -
(25+60\gamma^2)M_2^2 \\
&\qquad+ (55\gamma + 50\gamma^3)M_3 + (8 + 46\gamma^2 + 24\gamma^4)M_2,\\
L_{(2,2)} &= M_3^2 + 2\gamma M_3 M_2 - 4 M_4 + (\gamma^2+6) M_2^2 - 10 \gamma M_3
-(6\gamma^2+2) M_2.\\
\intertext{Similarly,}
    K_{(1)} &= R_2, \\
    K_{(2)} &= R_3 + \gamma R_2, \\
    K_{(3)} &= R_4 + 3\gamma R_3 + (1 + 2\gamma^2)R_2, \\
    K_{(4)} &= R_5 + 6\gamma R_4 + \gamma R_2^2 + (5 + 11\gamma^2)R_3 + (7\gamma +
6\gamma^3)R_2, \\
K_{(5)} &= R_6 + 10\gamma R_5 + 5\gamma R_3R_2 + (15 + 35 \gamma^2) R_4 + (5+10\gamma^2) R_2^2 \\
&\qquad+ (55\gamma + 50\gamma^3)R_3 + (8 + 46\gamma^2 + 24\gamma^4)R_2,\\
K_{(2,2)} &= R_3^2 + 2\gamma R_3 R_2 - 4 R_4 + (\gamma^2-2) R_2^2 - 10 \gamma R_3
-(6\gamma^2+2) R_2.
\end{align*}

A few striking facts appear on these examples. First, all coefficients are polynomials in the auxiliary parameter $\gamma$:
        we prove this fact in the next section with explicit bounds on the 
        degrees.
    Besides, for one part partition, polynomials $K_{(r)}$ 
        have non-negative coefficients.
        We are unfortunately unable to prove this statement, which is a slightly more
        precise version of \cite[Conjecture 1.2]{Lassalle2009}.
        A similar conjecture holds for several part partitions,
        see also \cite[Conjecture 1.2]{Lassalle2009}.

\section{Polynomiality in Kerov's expansion}
\label{SectPolynomial}
\subsection{Notations}
As in the previous sections, most of our objects are indexed by integer partitions.
Therefore it will be useful to use some short notations for small modifications
(adding or removing a box or a part) of partitions.
We denote by $\mu \cup (r)$ ($\mu \setminus (r)$, respectively)
the partition obtained from $\mu$ by adding  (deleting, respectively)
one part equal to $r$.
We denote by $\mu_{\downarrow r}=\mu \setminus (r) \cup (r-1)$ the partition
obtained from $\mu$ by removing one box in a row of size $r$.
The reader might wonder what do $\mu \setminus (r)$ and $\mu_{\downarrow r}$ mean
if $\mu$ does not have a part equal to $r$:
we will not use these notations in this context.
Finally, if $i$ is an inner corner of $\lambda$,
we denote by $\lambda^{(i)}$ the diagram obtained from $\lambda$ by adding a box
at place $i$.

\subsection{How to compute Jack character polynomials?}
\label{SubsectAlgo}
Unfortunately, the argument given above to prove the existence of $L_\mu$
and $K_\mu$ is not effective.
M.~Lassalle \cite{Lassalle2009} gave an algorithm for computing $K_\mu$
by induction over $\mu$.
In this section we present a slightly simpler version of this algorithm which allows to compute
$L_\mu$ instead of $K_\mu$.

One of the base ingredients is the following formula, which corresponds
to \cite[Proposition 8.3]{Lassalle2009}.
\begin{proposition}\label{PropMkLo}
    Let $k \geq 2$, $\lambda$ be a Young diagram and $i=(x,y)$ an inner corner of $\lambda$.
    \begin{multline*} M_k(\lambda^{(i)})-M_k(\lambda) =
        \sum_{\substack{r\geq 1, s,t \geq 0, \\ 2r+s+t \leq k}}
    z_i^{k-2r-s-t} \\
    \binom{k-t-1}{2r+s-1}\binom{r+s-1}{s}
    \left( -\gamma \right)^s M_t(\lambda),
\end{multline*}
    where $z_i=\sqrt{\alpha} x - \sqrt{\alpha}^{-1} y$ is the content of the 
    corner corresponding to $i$ in the $\alpha$-anisotropic diagram 
    $A_\alpha(\lambda)$.
\end{proposition}

\begin{proof}
    As mentioned above, this is exactly \cite[Proposition 8.3]{Lassalle2009}.
    To help the reader, we compare our notations to Lassalle's ones
    (we use boldface to refer to his notations):
    \begin{align*}
        M_k(\lambda^{(i)}) &= \alpha^{k/2} \bm{M_k(\lambda^{(i)})}; \\
        M_t(\lambda) &= \alpha^{t/2} \bm{M_t(\lambda)}; \\
        z_i &= \sqrt{\alpha} \cdot \bm{x_i};\\
        \gamma &= \sqrt{\alpha}^{-1} - \sqrt{\alpha}.\qedhere
    \end{align*}
\end{proof}

For any partition $\rho$ we define $M_\rho(\lambda):=\prod_i M_{\rho_i}(\lambda)$ by multiplicativity.
The above proposition implies immediately the following corollary:

\begin{corollary}\label{CorolMlo}
    For any partition $\rho$, any diagram $\lambda$ and any inner corner $i$ of $\lambda$,
    \[M_{\rho}(\lambda^{(i)}) = M_{\rho}(\lambda)
    + \sum_{\substack{ g,h \geq 0, \\ \pi \vdash h}} b_{g,\pi}^\rho(\gamma)
    \ z_i^g\ M_\pi(\lambda),\]
    where $b_{g,\pi}^\rho(\gamma)$ is a polynomial in $\gamma$.
\end{corollary}
\begin{proof}
    The case when $\rho$ consists of only one part is a direct consequence of Proposition \ref{PropMkLo} (one even has an explicit
    expression for $b_{g,\pi}^\rho(\gamma)$ in this case).
    The general case follows by multiplication.
\end{proof}
This corollary is an analogue of Equation (8.1) in \cite{Lassalle2009}.

Let $\mu$ be a partition. 
By definition of $L_\mu$, there exist some numbers $a^\mu_\rho$ (depending on $\alpha$)
such that, for any Young diagram $\lambda$,
\[ \Ch_\mu(\lambda) = \sum_\rho a^\mu_\rho\ M_\rho(\lambda). \]
Using Corollary~\ref{CorolMlo} we can compute
\begin{multline}\label{EqIng1Lassalle}
    \Ch_\mu(\lambda^{(i)}) = \sum_\rho a^\mu_\rho\ M_\rho(\lambda^{(i)}) \\
    = \Ch_\mu(\lambda) + \sum_\rho a^\mu_\rho \left( \sum_{\substack{g,h \geq 0, \\ \pi \vdash h}}
    b_{g,\pi}^\rho(\gamma)\ z_i^g\ M_\pi(\lambda) \right).
    \end{multline}

The second ingredient of Lassalle's algorithm is a linear identity between the values of
Jack character evaluated on different diagrams.
We denote by $c_i(\lambda)$ the probability of the corner $i$ in the transition measure 
$\mu_{A_\alpha(\lambda)}$,
so that
\begin{equation}\label{eq:Mk_co}
    M_k(\lambda) = \sum_i c_i(\lambda) z_i^k.
\end{equation}
In particular,
\begin{align}
    \sum_i c_i(\lambda)&=1, \label{eq:Sum_co}\\
    \sum_i c_i(\lambda) z_i&=0.\label{Sum_cozo}
\end{align}

Then we have \cite[Equation (3.6)]{Lassalle2009} the following proposition.
\begin{proposition}\label{PropLinearRelation}
    For any (continuous) Young diagram $\lambda$ and any partition $\mu$
    \begin{align*}
        \sum_{i \in \II_\lambda} c_i(\lambda) \Ch_\mu(\lambda^{(i)})
        &= m_1(\mu) \Ch_{\mu \backslash 1}(\lambda) + \Ch_\mu(\lambda), \\
        \sum_{i \in \II_\lambda} c_i(\lambda) z_i \Ch_\mu(\lambda^{(i)})
        &= \sum_{r \geq 2} r m_r(\mu) \Ch_{\mu_{\downarrow r}} (\lambda).
    \end{align*}
\end{proposition}
\begin{proof}
    It is an exercise to adapt Equations (3.6) of \cite{Lassalle2009} to our notations.
\end{proof}

Using Equations~\eqref{EqIng1Lassalle}, \eqref{eq:Mk_co}, \eqref{eq:Sum_co} and \eqref{Sum_cozo}
together with Proposition~\ref{PropLinearRelation}, we obtain the following equalities 
between functions on the set of (continuous) Young diagrams: for any partition $\mu$,
\begin{align}
    \sum_\rho a^\mu_\rho \left( \sum_{\substack{g,h \geq 0, \\ \pi \vdash h}} b_{g,\pi}^\rho(\gamma) 
    \, M_\pi \, M_g  \right) &= m_1(\mu) \Ch_{\mu \backslash 1},
    \tag{A}\label{EqRec1}\\
    \sum_\rho a^\mu_\rho \left( \sum_{\substack{g,h \geq 0, \\ \pi \vdash h}} b_{g,\pi}^\rho(\gamma) 
    \, M_{\pi} \, M_{g+1}  \right) &= \sum_{r \geq 2} r \cdot m_r(\mu) \Ch_{\mu_{\downarrow r}}.
    \tag{B}\label{EqRec2}
\end{align}
Fix some partition $\tau$.
We can identify the coefficient of a given monomial $M_\tau$ in the above equations.
This gives us two linear equations which will be denoted by $(A_\tau)$ and $(B_\tau)$:
\begin{align}
    \sum_\rho a^\mu_\rho \left( \sum_{\substack{g,h \geq 0, \ \pi \vdash h \\ \pi \cup (g)=\tau}} b_{g,\pi}^\rho(\gamma) 
     \right) &= m_1(\mu) \, a^{\mu \backslash 1}_\tau,
    \tag{$A_\tau$}\label{EqRecTau1}\\
    \sum_\rho a^\mu_\rho \left( \sum_{\substack{g,h \geq 0, \ \pi \vdash h \\ \pi \cup (g+1)=\tau}} b_{g,\pi}^\rho(\gamma) 
    \right) &= \sum_{r \geq 2} r \cdot m_r(\mu) \, a^{\mu_{\downarrow r}}_\tau.
    \tag{$B_\tau$}\label{EqRecTau2}
\end{align}

Now, assume that, for some partition $\mu$, we can compute $L_\nu$ for all partitions $\nu$ of size smaller than $|\mu|$.
Then the equations \eqref{EqRecTau1} and \eqref{EqRecTau2} can be interpreted as a linear system,
where the variables are the coefficients $a^\mu_\rho$.

This is a {\em finite} system of linear equations (indeed, $a^\mu_\rho=0$
as soon as $|\rho| \geq |\mu|+\ell(\mu)$ \cite[Proposition 9.2 (ii)]{Lassalle2009}).
As explained by M.~Lassalle, the system obtained that way has a unique solution
(we shall see another explanation of that in the next paragraph) and
thus, one can compute the coefficients $a^\mu_\rho$ by induction over $|\mu|$.

\subsection{A triangular subsystem}
In the previous section we explained how to determine the coefficients $a^\mu_\rho$
(where $\rho$ runs over partitions without parts equal to $1$) of $L_\mu$ as the solution
of an overdetermined linear system of equations.
In this section, we extract from this system a triangular subsystem.

We will need an order on all partitions: let us define $<_1$ as follows:
\[ \rho <_1 \rho' \iff
\begin{cases}
    & |\rho| <|\rho'|; \\
    \text{or}&|\rho| = |\rho'| \text{ and } \ell(\rho) > \ell(\rho'); \\
    \text{or}&|\rho|= |\rho'|,\ \ell(\rho) = \ell(\rho') \text{ and } \min(\rho) > \min(\rho').
\end{cases}\]

We say that an equation involves a variable if its coefficient is non-zero.
\begin{lemma}
    Let $\rho$ be a partition and $q=\min(\rho)$ its smallest part.
    \begin{itemize}
        \item If $q=2$, set $\tau=\rho\setminus (2)$.
            Then Equation $(A_\tau)$ involves the variable $a^\mu_\rho$
            and involves some of the variables $a^\mu_{\rho'}$ for $\rho' >_1 \rho$ 
            (and no other variables $a^\mu_{\rho'}$).
        \item If $q>2$, set $\tau= \rho_{\downarrow q}$.
            Then Equation $(B_\tau)$ involves the variable $a^\mu_\rho$
            and some of the variables $a^\mu_{\rho'}$ for $\rho' >_1 \rho$
            (and no other variables $a^\mu_{\rho'}$).
    \end{itemize}
    \label{LemTriangle}
\end{lemma}
\begin{proof}
    We can refine Corollary~\ref{CorolMlo} as follows:
    for any partition $\rho$, any diagram $\lambda$ and any inner corner $i$ of $\lambda$,
\begin{multline}
    M_\rho(\lambda^{(i)}) = M_\rho(\lambda) + \sum_{j \leq \ell(\rho)}
    M_{\rho \setminus \rho_j} 
    \left( \sum_{\substack{g,t \geq 0, \\ g=\rho_j-2-t}} (\rho_j-t-1) M_t(\lambda) z_i^g \right)\\
    + \sum_{\substack{\pi,g \\ |\pi|+g < |\rho| - 2}} b_{g,\pi}^\rho(\gamma) M_\pi(\lambda) z_i^g.
    \label{eq:DsLemTriang}
\end{multline}
Indeed, it is true for $\rho=(k)$ and follows directly for any $\rho$ by multiplication.
The right-hand side is a linear combination of $M_\pi z_i^g$ with
\[  |\pi|+g \leq |\rho| - 2. \]
Moreover equality occurs only if $\pi \cup (g)$ is obtained from $\rho$
by choosing a part, removing $2$ to this part and
splitting it in two (possibly empty) parts.

Let us consider the first statement of the lemma.
Fix a partition $\rho$ with a smallest part equal to $2$,
that is $\rho=\tau \cup (2)$ for some partition $\tau$.
Let us determine which variables $a^\mu_{\rho'}$ appear
in the left-hand side of Equation $(A_\tau)$.
In other terms, we want to determine for which $\rho'$,
the difference $M_{\rho'}(\lambda^{(i)})-M_{\rho'}(\lambda)$ contains
some term $M_\pi z_i^g$, for which $\tau=\pi \cup (g)$.
As explained above,
a necessary condition is the inequality $|\tau|=|\pi|+g \leq |\rho'| - 2$, \emph{i.e.}
$|\rho'| \geq |\tau|+2$.
Moreover, if $|\rho'| = |\tau| +2$, then $\tau$ must be obtained from $\rho'$ by removing
$2$ from some part and splitting it into two.
In particular, $\rho'$ cannot be longer than $\tau$, unless both new
parts are empty.
This happens only if the split part of $\rho'$ was $2$, that is if $\rho'= \tau \cup (2)$.
We have proved that $(A_\tau)$ can involve $a^\mu_{\rho'}$
only if $\rho'=\rho$ or $\rho' >_1 \rho$ (either $\rho'$ has a bigger size than $\rho$, or it has a smaller length).
It is easy to check that the coefficient of $a^\mu_\rho$ is equal to $m_2(\rho)$
(and thus, non-zero),
which finishes the proof of the first point.

The proof of the second point is quite similar.
Fix a partition $\rho$, denote $q$ its smallest part (assume $q>2$)
and $\tau=\rho_{\downarrow q}$.
The same argument than above tells us that
the variable $a^\mu_{\rho'}$ can appear in the equation $(B_\tau)$ only if
$|\tau|=|\pi|+g+1 \leq |\rho'| - 1$, { i.e.} $|\rho'| \geq |\tau|+1$.
Moreover, if there is equality,  $\tau$ must be obtained from $\rho'$ by removing
$1$ from some part and splitting it into two.
One of the two new parts is always non-empty (as $\rho'$ has no parts equal to $1$),
thus $\rho'$ is at most as long as $\tau$.
If they have the same length, it means that $\tau$ is obtained from $\rho'$ by shortening
a part.
If this part is equal to $q$, then $\rho'=\rho$.
Otherwise $\rho'$ contains a part $q-1$ and thus $\rho' >_1 \rho$
(they have same size and same length).
Finally, we have proved that $(B_\tau)$ can involve $a^\mu_{\rho'}$
only if $\rho' >_1 \rho$. 
Once again, the coefficient of $a^\mu_\rho$ in $(B_\tau)$ is easy to compute:
it is equal to $(q-1) m_q(\rho)$ and, hence, non-zero.
\end{proof}

The first interesting consequence is the following.

\begin{corollary}
    The coefficient $a_\rho^\mu$ is a polynomial in $\gamma$ with rational coefficients.
    The same is true for the coefficients of Kerov's polynomials $K_\mu$.
\end{corollary}
\begin{proof}
    We proceed by induction over $|\mu|$.
    The quantities $a_\rho^\mu$ are the solution of a triangular linear system,
    whose right-hand side is a vector of $a_\tau^{\mu'}$ with $|\mu'|<|\mu|$.
    By induction hypothesis, the right-hand side belongs to $\QQ[\gamma]$.
    The coefficients $b_{g,\pi}^\rho(\gamma)$ of the system also belong to $\QQ[\gamma]$.
    Moreover, the diagonal coefficients of the system (given in the proof above)
    are invertible in $\QQ[\gamma]$, hence the solution is also in $\QQ[\gamma]$.

    For the second statement, it is enough to say that each $M_k$ is a polynomial in the
    $R_k$'s with integer coefficients.
\end{proof}

\begin{remark}
    Lemma \ref{LemTriangle} does not hold for Lassalle's system of equation
    \cite[Equations (9.1) and (9.2)]{Lassalle2009}, which computes recursively $K_\mu$.
\end{remark}

\subsection{A first bound on the degree}\label{SubsectBound1}
Recall that $(M_k)_{k \geq 2}$ is an algebraic basis of the ring
$\Pola$ of
$\alpha$-polynomial functions on Young diagrams.
Hence, we can define a gradation on $\Pola$ by choosing
arbitrarily the degree of each of the generators $M_k$.
In this section, we do the following natural choice: 
\[ \deg_1(M_k)=k \qquad \text{for } k\geq 2. \]

Our goal is to obtain a bound on the degree of the polynomial $a^\mu_\rho\in\QQ[\gamma]$.
We begin by the following lemma concerning the polynomials
$b_{g,\pi}^\rho(\gamma)$.

{\em Notational convention.}
To emphasize the difference with gradations on $\Pola$, we denote throughout the paper
degrees of polynomials in $\gamma$ by $\deg_\gamma$.

\begin{lemma}
\label{lem:firstbound}
    Let $\rho$ and $\pi$ be two partitions and $g\geq 0$ be an integer. One has
    \[\deg_\gamma(b_{g,\pi}^\rho(\gamma)) \leq \deg_1(M_\rho) - \deg_1(M_{\pi \cup (g)}) - 2. \]
    Moreover, if the right-hand side is an even (odd, respectively) number,
    then $b_{g,\pi}^\rho(\gamma)$ is an even (odd, respectively) polynomial.
    \label{LemDeg1B}
\end{lemma}
\begin{proof}
    By Proposition~\ref{PropMkLo}, $M_k(\lambda^{(i)})$ can be written as a linear combination of terms of the form
    $b(\gamma)\ M_\pi\ z_i^g$, where $b$ is some polynomial.
    We define the pre-degree (with respect to $\deg_1$)
    of such a term to be the quantity $\deg_\gamma(b) + |\pi| + g$.
    This degree is multiplicative.
    Then, 
    \[M_k(\lo) = M_k(\lambda) + \text{ terms of degree smaller or equal to }k-2.\]
    By multiplying this kind of expressions we obtain that
    \[M_\rho(\lo) = M_\rho(\lambda) + \text{ terms of degree smaller or equal to }|\rho|-2,\]
    which corresponds to our bound on the degree.
    The parity also follows immediately from the one-part case by multiplication.
\end{proof}

This yields the following result.
\begin{proposition}\label{PropBound1}
    The coefficient $a^\mu_\rho$ of $M_\rho$ in
    Jack character polynomial $L_\mu$ is a polynomial in $\gamma$
    of degree smaller or equal to $|\mu|+\ell(\mu) - |\rho|$.
    Moreover, it has the same parity as the integer  $|\mu|+\ell(\mu) - |\rho|$.

    The same is true for $K_\mu$.
\end{proposition}
\begin{proof}
    We proceed by induction over $(\mu,\rho)$.
    The base case $\mu=(1)$ is trivial as $L_{(1)}=M_2$.
    Fix two partitions $\mu$ and $\rho$.
    We assume that our result holds for any pair $(\mu',\rho')$
    with $|\mu'| < |\mu|$ or $|\mu'|=|\mu|$ and $\rho' >_1 \rho$.

    It may seem strange to assume that the result holds for $\rho' \bm{>_1} \rho$.
    We are indeed doing some kind of {\em descending induction}.
    This is possible because, for a given $\mu$, the number of partitions $\rho$
    we shall consider is finite:
    indeed, $a^\mu_\rho=0$ as soon as $|\rho| \geq |\mu|+\ell(\mu)$ 
    \cite[Proposition 9.2 (ii)]{Lassalle2009}.
    The same remark holds for most proofs in this section.

    Let us first consider the case when $\rho=\tau \cup (2)$ contains a part equal to 2.
    By Lemma~\ref{LemTriangle}, Equation $(A_\tau)$ can be written as:
    \[ m_2(\rho) \cdot a^\mu_\rho = m_1(\mu) a^{\mu \backslash 1}_{\tau} 
    - \sum_{\substack{\pi,g, \\ \pi \cup (g) = \tau}} \sum_{\rho' >_1 \rho}
    b_{g,\pi}^{\rho'}(\gamma) a^\mu_{\rho'}.\]
    The first term on the right-hand-side is by convention equal to 0 if $\mu$ does not contain any part equal to 1.
    If $\mu$ contains a part equal to $1$, as $|\mu \setminus 1|$ is smaller than $|\mu|$,
    by induction hypothesis $a^{\mu \backslash 1}_{\tau}$ is a polynomial of degree at most
    \[|\mu \setminus 1| + \ell(\mu \setminus 1) - |\tau| = 
    |\mu|-1 + \ell(\mu)-1-(|\rho|-2)=|\mu|+\ell(\mu) - |\rho|.\]
    As $\rho' >_1 \rho$, 
    we can also apply the induction hypothesis to each summand 
    of the second term: $a^\mu_{\rho'}$ is  polynomial of degree at most
    $|\mu|+\ell(\mu)-|\rho'|$.
    But using Lemma~\ref{LemDeg1B}, $b_{g,\pi}^{\rho'}(\gamma)$ has degree at most
    $|\rho'|-|\pi \cup (g)|-2$.
    Hence the degree of the product is bounded by
    \[|\mu|+\ell(\mu) - (|\pi \cup (g)|+2) = |\mu|+\ell(\mu) - |\rho|.\]
    The last equality comes from the fact that $\pi \cup (g)=\tau=\rho \setminus 2$.

    The proof of the case when the smallest part $q$ of $\rho$ is greater than $2$ is similar.
    We use Equation $(B_\tau)$ for $\tau=\rho_{\downarrow q}$, which takes the form:
    \[ (q-1) m_q(\rho) \cdot a^\mu_\rho = \sum_{r\geq 2} r \cdot m_r(\mu) a^{\mu_{\downarrow r}}_{\tau} 
    - \sum_{\substack{\pi,g, \\ \pi \cup (g+1) = \tau}} \sum_{\rho' >_1 \rho}
    b_{g,\pi}^{\rho'}(\gamma) a^\mu_{\rho'}.\]
    Note that $|\mu_{\downarrow r}|<|\mu|$, therefore by induction hypothesis
    $a^{\mu_{\downarrow r}}_{\tau}$ is a polynomial in $\gamma$ of degree at most
    \[|\mu_{\downarrow r}| + \ell(\mu_{\downarrow r}) -|\tau|=|\mu|-1+\ell(\mu)-(|\rho|-1)
    = |\mu|+\ell(\mu) - |\rho|.\]
    For the second summand, the argument is the same as before, except that here
    the equality $|\pi \cup (g)|+2=|\rho|$ comes from the fact that $|\tau|=|\rho|-1$
    and $|\tau|=|\pi \cup (g+1)|= |\pi \cup (g)|+1$.

    The parity is obtained the same way.
\end{proof}

\begin{corollary}
    \label{corol:dominant_deg1}
For any partition $\mu$ one has:
\[ \deg_1(\Ch_\mu) = |\mu| + \ell(\mu).\]
Moreover
\[ \Ch_\mu = \prod_{i}R_{\mu_i+1} + \text{lower degree terms with respect to $\deg_1$}. \]
\end{corollary}

\begin{proof}
    By Proposition~\ref{PropBound1}, $\Ch_\mu$ has at most degree $|\mu|+\ell(\mu)$
(this has also been proved by Lassalle
\cite[Proposition 9.2 (ii)]{Lassalle2009})
and its component of degree $|\mu|+\ell(\mu)$ does not depend on $\alpha$.
Hence the result follows as this dominant term is known in the case $\alpha=1$
(see for example \cite[Theorem 4.9]{SniadyGenusExpansion}).
\end{proof}

\subsection{A second bound on degrees}\label{SubsectBound2}
For some purposes the bound on the degree of $a_\rho^\mu$ given by Proposition \ref{PropBound1} is not strong enough.
In this section we give another bound which is related to another gradation of $\Pola$
defined by:
\[ \ \deg_2(M_k)=k-2 \qquad \text{for } k\geq 2 .\]
One has the following analogue of Lemma~\ref{LemDeg1B}:
\begin{lemma}
    Let $\rho$ and $\pi$ be two partitions and $g\geq 0$ an integer.
    Then
    \begin{align*}
            \deg_\gamma(b_{g,\pi}^\rho(\gamma)) &\leq \deg_2(M_\rho) - \deg_2(M_{\pi \cup (g)}), \\
            \deg_\gamma(b_{g,\pi}^\rho(\gamma)) &\leq \deg_2(M_\rho) - \deg_2(M_{\pi \cup (g+1)}) - 1. 
    \end{align*}
    \label{LemDeg2B}
\end{lemma}
\begin{proof}
    The proof is similar to the one of Lemma \ref{LemDeg1B}.
    We define the {\em pre-degree} (with respect to $\deg_2$)
    of an expression of the form $b(\gamma)\ M_\pi\ z_i^g$
    to be $\deg_\gamma(b)+\deg_2(M_\pi)+g$.
    By Proposition \ref{PropMkLo}, the pre-degree of $M_k(\lo)$ is equal to $k-2$.
    Note that this pre-degree is multiplicative. Then
    $M_\rho(\lo)$ has pre-degree $|\rho|-2 \ell(\rho)=\deg_2(M_\rho)$.
    The lemma follows from the following inequalities: for $g\geq 0$,
    \begin{align*}
        \deg_2(M_{\pi \cup (g)}) &\leq \deg_2(M_{\pi})+g; \\
        \deg_2(M_{\pi \cup (g+1)}) &\leq \deg_2(M_{\pi})+g-1.
    \end{align*}
    Note that in the first inequality the difference between the right hand side and the left hand side is equal to $2$, unless $g=0$;
    in that case we have an equality.
    In the second inequality, the case $g=0$ is obvious as $M_{\pi \cup (1)}=0$
    and hence its degree is $-\infty$ by convention.
    In all other cases, we have an equality.
\end{proof}

We deduce from this lemma a new bound on the degree of $a_\rho^\mu$.
\begin{proposition}
    \label{PropBound2}
    The coefficient $a^\mu_\rho$ of $M_\rho$ in
    Jack character polynomial $L_\mu$ is a polynomial in $\gamma$
    of degree smaller or equal to $|\mu|-\ell(\mu) - (|\rho|-2\ell(\rho))$.

    The same is true for $K_\mu$.
\end{proposition}
\begin{proof}
    It is a straightforward exercise to adapt the proof of Proposition \ref{PropBound1}.
    We use Lemma~\ref{LemDeg2B} instead of Lemma~\ref{LemDeg1B} and 
    $|\rho|$ has to be replaced by $|\rho| - 2\ell(\rho)$.
\end{proof}
As an immediate consequence we have
\begin{corollary}
For any partition $\mu$ one has:
\[ \deg_2(\Ch_\mu) = |\mu| - \ell(\mu),\]
and the top degree part does not depend on $\alpha$.
\end{corollary}

Note that Proposition \ref{PropBound2} is neither weaker nor stronger than Proposition \ref{PropBound1}.
But it is sometimes more appropriate, as we shall see in the next section.

\begin{remark}
    The top degree part of $\Ch_\mu$ for $\deg_2$ does not admit an
    explicit expression as for $\deg_1$.
    One can however compute its linear terms in free cumulants,
    see \cite[Section 3]{FerayGouldenHook}.
\end{remark}

\subsection{Polynomiality of $\theta_\mu(\lambda)$}\label{SubsectCoefJackPoly}
In this section we prove that $\theta_\mu(\lambda)$ is a polynomial
with rational coefficients in $\alpha$.
This simple statement does not follow directly from the definition of Jack polynomials
and had been open for twenty years.
It was then proved by Lapointe and Vinet~\cite{LapointeVinetJack},
who also proved the integrality of the coefficients in the monomial basis.
Short after that, this result was completed by a positivity result
from Knop and Sahi~\cite{KnopSahiCombinatoricsJack}.

Using the material of this Section, we can find a new proof of the polynomiality
in $\alpha$.
Integrity and positivity seem unfortunately impossible to obtain {\em via} this method.

First, we consider the dependence of $M_k(\lambda)$ on $\alpha$.
\begin{lemma}
    Let $k \geq 2$ be an integer and $\lambda$ a partition.
    Then $\sqrt{\alpha}^{k-2} M_k(\lambda)$ is a polynomial in
    $\alpha$ with integer coefficients.
    \label{LemMkPol}
\end{lemma}
\begin{proof}
    We use induction over $|\lambda|$ and $k$.
    Proposition~\ref{PropMkLo} can be rewritten as
    \begin{multline*} \sqrt{\alpha}^{k-2} M_k(\lambda^{(i)})
        -\sqrt{\alpha}^{k-2} M_k(\lambda) =
        \sum_{\substack{r\geq 1, s,t \geq 0, \\ 2r+s+t \leq k}}
        \alpha^r (\sqrt{\alpha} z_i)^{k-2r-s-t} \\
    \binom{k-t-1}{2r+s-1}\binom{r+s-1}{s}
    \left( \alpha-1 \right)^s \sqrt{\alpha}^{t-2} M_t(\lambda).
\end{multline*}
Note that $\sqrt{\alpha} z_i = \alpha x - y$ is a polynomial
in $\alpha$ with integer coefficients.
Hence the induction is immediate.
\end{proof}
Now we write, for $\mu,\lambda \vdash n$,
\begin{multline*}
    z_\mu \theta_\mu(\lambda)=\alpha^{\frac{|\mu|-\ell(\mu)}{2}} \Ch_\mu(\lambda)
= \alpha^{\frac{|\mu|-\ell(\mu)}{2}} \sum_\rho a^\mu_\rho M_\rho(\lambda) \\
= \sum_\rho \alpha^{\frac{|\mu|-\ell(\mu)-(|\rho|-2\ell(\rho))}{2}} 
a^\mu_\rho \left( \prod_{i \leq \ell(\rho)} \sqrt{\alpha}^{\rho_i-2}
M_{\rho_i}(\lambda) \right).
\end{multline*}
The quantities $\alpha^{\frac{|\mu|-\ell(\mu)-(|\rho|-2\ell(\rho))}{2}} 
a^\mu_\rho$ and $\sqrt{\alpha}^{\rho_i-2}
M_{\rho_i}(\lambda)$ are polynomials in $\alpha$ (by Proposition~\ref{PropBound2}
and Lemma~\ref{LemMkPol}), hence $\theta_\mu(\lambda)$ is a polynomial
in $\alpha$. \qedhere

\subsection{Yet another gradation and bound on degrees}
The gradation introduced in Section \ref{SubsectBound2} is suitable for some
purposes (as we have seen in the previous section), but it has the unpleasant
aspect that all homogeneous spaces have an infinite dimension.
In particular, Proposition~\ref{PropBound2} does not give any information on the
maximal power of $M_2$ which can appear in $L_\mu$.
In this section we propose a way to avoid this difficulty.
It is technical but will be useful in the next section.

We define a new algebraic basis of $\Pola$ by:
\begin{align*} 
M'_2 & =M_2,\\ 
M'_k & = M_k - (-\gamma)^{k-2} M_2 \qquad \text{for } k\geq 3.
\end{align*}
We also consider the gradation defined by:
\[\deg_3(M'_2) = 1, \quad \deg_3(M'_k) =k-2 \text{ for }k\ge 3,\]
so that $\deg_3(M'_\rho)=|\rho|-2\ell(\rho) +m_2(\rho)$.
Obviously, there exists a polynomial $L'_\mu$ such that
\[ \Ch_\mu = L'_\mu(M'_2,M'_3,\dots).\]
For example, one has:
\[L'_{(2,2)} = (M'_3)^2 +6(M'_2)^2 - 4 M'_4 -10 \gamma M'_3 -2M'_2.\]
We denote by $(a')_\rho^\mu$ the coefficient of $M'_\rho$ in $L'_\mu$.
Then, one has the following result.

\begin{proposition}\label{PropBound3}
    The coefficient $(a')_\rho^\mu$ is a polynomial in $\gamma$
    of degree at most $|\mu|-\ell(\mu)+m_1(\mu) - (|\rho|-2\ell(\rho) +m_2(\rho))$.
\end{proposition}
\begin{remark}
    The analogous result is not true for $a_\rho^\mu$,
    as it can be seen on the case $\mu=(2,2)$.
\end{remark}
The algorithm to compute the coefficient $(a')_\rho^\mu$ is the same as for $a_\rho^\mu$
and the proof of the bound on degrees is similar to those of 
Propositions~\ref{PropBound1} and \ref{PropBound2}.
Let us give some details.

First, one can rewrite Proposition~\ref{PropMkLo} in terms of the quantities $M'_k$:
\begin{multline}
    M'_k(\lambda^{(i)}) - M'_k(\lambda)
    =M_k(\lambda^{(i)}) - M_k(\lambda) - (-\gamma)^{k-2} \\
    =    \sum_{\substack{ r\geq 1, s,t \geq 0, \\ 2r+s+t \leq k \\  \bm{(r,s,t) \neq (1,k-2,0) }}}
     \Bigg[ z_i^{k-2r-s-t}  \binom{k-t-1}{2r+s-1}\binom{r+s-1}{s} \\ 
    \cdot (-\gamma)^s (M'_t(\lambda)+ (-\gamma)^{t-2} M'_2(\lambda)) \Bigg].
    \label{EqM'klo}
\end{multline}
Please note that the term $(-\gamma)^{k-2}$ corresponding to $(r,s,t)=(1,k-2,0)$
does not belong to the sum any more.
By multiplication, there exist some polynomials $(b')_{g,\pi}^\rho(\gamma)$ such that
\[M'_\rho(\lo) = M'_\rho(\lambda) + \sum_{g,\pi} (b')_{g,\pi}^\rho(\gamma)\
z_i^g\ M'_\pi(\lambda).\]
Using Equation~\eqref{EqIng1Lassalle} and Proposition~\ref{PropLinearRelation}, we obtain
the following equalities:
\begin{align}
    \sum_\rho (a')^\mu_\rho \left( \sum_{\substack{g,h \geq 0, \\ \pi \vdash h}} (b')_{g,\pi}^\rho(\gamma) 
    M'_\pi M_g  \right) &= m_1(\mu) \Ch_{\mu \backslash 1},
    \tag{$A'$}\label{EqRec1'}\\
    \sum_\rho (a')^\mu_\rho \left( \sum_{\substack{g,h \geq 0, \\ \pi \vdash h}} (b')_{g,\pi}^\rho(\gamma) 
    M'_\pi M_{g+1}  \right) &= \sum_{r \geq 2} r \cdot m_r(\mu) \Ch_{\mu_{\downarrow r}}.
    \tag{$B'$}\label{EqRec2'}
\end{align}
Plugging $M_g=M'_g + (-\gamma)^{g-2} M'_2$ in these equations and identifying the coefficient
of $M'_\tau$ on both sides, we obtain the following system:
\begin{multline*}
    \sum_\rho (a')^\mu_\rho \left( \sum_{\substack{g, \pi, \\ \pi \cup (g)=\tau}} (b')_{g,\pi}^\rho(\gamma) 
    + \sum_{\substack{g>2, \pi, \\ \pi \cup (2)=\tau}} (-\gamma)^{g-2} (b')_{g,\pi}^\rho(\gamma)  \right) \\ =
    m_1(\mu) (a')^{\mu \backslash 1}_\tau,
    \tag{$A'_\tau$}\label{EqRec1'tau}
\end{multline*}
\begin{multline*}
    \sum_\rho (a')^\mu_\rho \left( \sum_{\substack{g, \pi, \\ \pi \cup (g+1)=\tau}} (b')_{g,\pi}^\rho(\gamma) 
    + \sum_{\substack{g \geq 2, \pi ,\\ \pi \cup (2)=\tau}} (-\gamma)^{g-1} (b')_{g,\pi}^\rho(\gamma)
    \right) \\= \sum_{r \geq 2} r \cdot m_r(\mu) (a')^{\mu_{\downarrow r}}_\tau,
    \tag{$B'_\tau$}\label{EqRec2'tau}
 \end{multline*}
It is easy to check that Lemma~\ref{LemTriangle} still holds for this system.

The next step is to give a bound on the degree of $(b')_{g,\pi}^\rho(\gamma)$.
\begin{lemma}
    \[ \deg_\gamma(b_{g,\pi}^\rho(\gamma)) \leq \deg_3(M'_\rho) - \deg_3(M'_\pi) - \max(g,1).\]
    \label{LemDeg3B'}
\end{lemma}
\begin{proof}
    Let us call \emph{pre-degree} (with respect to $\deg_3$)
    of an expression of the form $b(\gamma)\ M'_\pi\ z_i^g$
    the quantity $\deg_\gamma(b)+\deg_3(M'_\pi)+g$.
    It is multiplicative.
    Clearly, $M'_k(\lo)$ has pre-degree $\max(k-2,1)$ (see Equation~\eqref{EqM'klo}),
    thus $M'_\rho(\lo)$ has pre-degree $\deg_3(M'_\rho)$, which finishes the proof
    of the case $g \geq 1$.
    For $g=0$, one has to look at the term which does not involve $z_i$.
    It is easy to check on Equation~\eqref{EqM'klo} (here, it is crucial to use $M'$
    and not $M$) that
    \[M'_k(\lo)|_{z_i=0}=M'_k(\lambda)|_{z_i=0} + \left(\text{terms of pre-degree $k-3$}\right).\]
    Hence by multiplication, 
    \[M'_\rho(\lo)|_{z_i=0}=M'_\rho(\lambda)|_{z_i=0} + 
    \left(\text{terms of pre-degree }\deg_3(M'_\rho) - 1\right),\]
    which finishes the proof of the lemma.
\end{proof}

We have now all the tools to prove Proposition~\ref{PropBound3} by induction.
As usual, we first consider the case where $\rho=\tau \cup (2)$ has the smallest part equal to $2$.
Then Equation $(A_\tau)$ can be written as:
\begin{multline*}
    m_2(\rho) \cdot (a')^\mu_\rho = m_1(\mu) (a')^{\mu \backslash 1}_{\tau} 
    - \sum_{\substack{\pi,g, \\ \pi \cup (g) = \tau}} \sum_{\rho' >_1 \rho}
    (b')_{g,\pi}^{\rho'}(\gamma) (a')^\mu_{\rho'} \\
    -\sum_{\substack{g>2, \pi, \\ \pi \cup (2)=\tau}} \sum_{\rho' >_1 \rho}
    (-\gamma)^{g-2} (b')_{g,\pi}^\rho(\gamma)(a')^\mu_{\rho'}.
\end{multline*}
With arguments similar to the ones used previously,
the first two terms are polynomials in $\gamma$ of degree at most
$|\mu|-\ell(\mu)+m_1(\mu) - \deg_3(M'_{\rho})$.
Let us focus on the last summand.
By induction hypothesis $(a')^\mu_{\rho'}$ is a polynomial of degree 
$|\mu|-\ell(\mu)+m_1(\mu) - \deg_3(M'_{\rho'})$.
By Lemma \ref{LemDeg3B'}, $(b')_{g,\pi}^\rho(\gamma)$ has degree equal to
$\deg_3(M'_\rho) - \deg_3(M'_\pi) - g$.
Hence the product of these two terms with $(-\gamma)^{g-2}$ has degree at most
\[|\mu|-\ell(\mu)+m_1(\mu) - (\deg_3(M'_\pi) - 2) = |\mu|-\ell(\mu)+m_1(\mu) - \deg_3(M'_\rho).\]
The equality comes from the fact that $\rho=\tau \cup (2)=\pi \cup (2,2)$.

Finally one obtains that $(a')^\mu_\rho$ has degree at most
$|\mu|-\ell(\mu)+m_1(\mu) - \deg_3(M'_{\rho})$.

The case when $\rho$ has no parts equal to $2$ is similar. \qed

\begin{corollary}
For any partition $\mu$ one has:
\[ \deg_3(\Ch_\mu) = |\mu| - \ell(\mu)+m_1(\mu),\]
and the top degree part does not depend on $\alpha$.
    \label{corol:deg3_Ch}
\end{corollary}

\begin{remark}
The top degree part of $\Ch_\mu$ for $\deg_3$ has, as far as we know,
no close expression.
\end{remark}

\subsection{Gradations and characters}\label{SubsectDegCh}
In the previous sections we have defined three different gradations.
The elements of our favorite basis $(\Ch_\mu)$ are not homogeneous,
but have the following nice property:
if we define
\[ V_i^d := \{x \in \Pola : \deg_i(x) \leq d\} \]
then, for $i=1$ and $i=3$, each $V_i^d$ is spanned linearly
by the functions $\Ch_\mu$ that it contains
(this comes from a direct dimension argument).
This simple observation will be useful later.

The same argument can not be used for $i=2$, as the spaces $V_2^d$
are all infinite dimensional.

\begin{remark}
    The functions $\deg_i$, for $i=1,3$, define some gradations and 
    hence some filtrations on $\Pola$.
    These filtrations were known in the cases $\alpha=1,2$ ;
    see \cite{IvanovOlshanski2002,FerayInductionJM,Tout-Structure-constant-S2n-Hn}.

    In fact, Ivanov and Olshanski \cite[Proposition 4.9]{IvanovOlshanski2002}
    give many more filtrations for $\alpha=1$,
    but we have not been able to prove that 
    they hold for general $\alpha$.
    In particular, the filtration
    \begin{equation}\label{eq:Kerov_filtration}
    \deg(\Ch_\mu)=|\mu|+m_1(\mu)
    \end{equation}
     is central in their analysis of fluctuations
    of random Young diagrams.
    Unfortunately, we are unable to prove that
    \eqref{eq:Kerov_filtration} still defines a filtration in the general $\alpha$-case.
    We leave this as an open question.
    If we were able to positively answer it,
    we could use a moment method to obtain our fluctuation results
    (without the bound on the speed of convergence).
\end{remark}

\section{Polynomiality of structure constants of Jack characters}
\label{SectStructureConstants}

\subsection{Structure constants are polynomials in $\gamma$}
\label{SubsectStructPol}

In this section we are going to prove our main result for the structure constants of the algebra $\Pola$ of $\alpha$-polynomial functions which was stated as Theorem \ref{theo:struct-const}.

\begin{proof}[Proof of Theorem \ref{theo:struct-const}:]
First observe that for each $i \in \{1,2,3\}$ one has:
\[ n_i(\mu) = \deg_i(\Ch_\mu),\]
hence our bound on the degree of structure constants can be equivalently formulated using three gradations introduced in Section \ref{SectPolynomial}.
    
    Let us consider the bound involving $\deg_1$ (the case of $\deg_3$ is similar).
    We know by Proposition~\ref{PropBound1} that
    \[ \Ch_\mu = \sum_\rho a_\rho^\mu M_\rho,\]
    where each $a_\rho^\mu$ is a polynomial in $\gamma$ of degree 
    $\deg_1(\Ch_\mu)-\deg_1(M_\rho)$.
    Hence, we have
    \[ \Ch_\mu \cdot \Ch_\nu = \sum_\rho b_\rho^{\mu,\nu} M_\rho,\]
    where each $b_\rho^{\mu,\nu}$ is a polynomial in $\gamma$ of degree 
        $\deg_1(\Ch_\mu)+\deg_1(\Ch_\nu)-\deg_1(M_\rho)$.
        In particular $\Ch_\mu \cdot \Ch_\nu$ has degree at most $\deg_1(\Ch_\mu)+\deg_1(\Ch_\nu)$
        and hence, thanks to the remark of Section~\ref{SubsectDegCh},
        $g_{\mu,\nu;\pi}=0$ whenever 
        \[ n_1(\mu) + n_1(\nu) < n_1(\pi).\]
    The structure constants are obtained by solving the linear system:
    \begin{equation}
        \sum_\tau a_\rho^\tau g_{\mu,\nu;\tau} = b_\rho^{\mu,\nu}.
        \tag{S}\label{EqSystStructure}
    \end{equation}
    In this system, $\mu$ and $\nu$ are fixed, there is one equation for each partition $\rho$
    without parts equal to $1$.
    In each equation, the sum runs over partitions $\tau$ such that
    $n_1(\tau) \le n_1(\mu) + n_1(\nu)$.
    Finally, the unknown are $g_{\mu,\nu;\tau}$, for $\tau$ as above.

    We will prove our statement by induction over 
    \[ \deg_1(\Ch_\mu) + \deg_1(\Ch_\nu) - \deg_1(\Ch_\pi).\]
    Fix some partitions $\mu$,$\nu$ and $\pi$.
    If the quantity above is negative, the coefficient $g_{\mu,\nu;\pi}$ is equal to $0$ and
    the statement is true.
    Otherwise we suppose that for all partitions $\tau$ bigger than
    $\pi$ (in the sense that $\deg_1(\Ch_{\tau}) > \deg_1(\Ch_{\pi})$),
    the degree of $g_{\mu,\nu;\tau}$ is bounded from above by
    $\deg_1(\Ch_\mu) + \deg_1(\Ch_\nu) - \deg_1(\Ch_\tau)$.
    
    Note that $a_\rho^\tau$ vanishes as soon as $\deg_1(\Ch_\tau) < \deg_1(M_\rho)$.
    Then from \eqref{EqSystStructure} we extract a subsystem
    \begin{equation}                                          
        \sum_{\substack{\tau, \\ \deg_1(\Ch_\tau)=\deg_1(\Ch_{\pi})}} a_\rho^\tau g_{\mu,\nu;\tau} =
        b_\rho^{\mu,\nu} - \sum_{\substack{\tau, \\ \deg_1(\Ch_\tau)>\deg_1(\Ch_{\pi})}}
        a_\rho^\tau g_{\mu,\nu;\tau},      
        \tag{$S'$}
        \label{syst:Sp}
    \end{equation}
    where $\rho$ runs over partitions such that $\deg_1(M_\rho)=\deg_1(\Ch_\pi)$.
    The variables are $g_{\mu,\nu;\tau}$ for $\tau$ with $\deg_1(\Ch_\tau)=\deg_1(\Ch_{\pi})$.
    This system is invertible (because $(\Ch_\pi)$ is a basis of $\Pola$) and
    the coefficients on the left-hand side of
    \eqref{syst:Sp} are rational numbers (by Proposition~\ref{PropBound1}).
    Besides, all terms on the right-hand side are polynomials in $\gamma$ of degree
    at most $\deg_1(\Ch_\mu) + \deg_1(\Ch_\nu) - \deg_1(\Ch_\pi)$
    which finishes the proof.

    The proof of the parity follows in the same way.\medskip

    The bound involving $\deg_2$ is obtained in a slightly different way.
    First note that if $\mu$ and $\nu$ do not have any parts equal to $1$,
    this bound is weaker that the one with $\deg_3$.
    Hence, it holds in this case.
    Then the general case follows, using the fact that
    $\Ch_{\mu \cup (1)}=(|\lambda| - |\mu|)\Ch_\mu$.
\end{proof}

\subsection{Projection on functions on Young diagrams of size $n$}\label{SubsectProjN}
Recall that $\Pola$ is a subalgebra of the algebra of functions on all Young diagrams.
The latter has a natural projection map $\varphi_n$ onto $\F(\Young_n,\QQ)$, the algebra
of functions on Young diagrams of size $n$.
As Jack symmetric functions $J_\lambda$ form a basis of the symmetric function ring,
the functions $(\theta_\mu)_{\mu \vdash n}$ form a basis of $\F(\Young_n,\QQ)$
(see \cite[Proposition 4.1]{FerayInductionJM}).

We consider the structure constants $c_{\mu,\nu;\pi}$ 
of $\F(\Young_n,\QQ)$ with basis $(\theta_\mu)_{\mu \vdash n}$,
that is the numbers uniquely defined by:
\begin{equation}\label{EqDefC}
\theta_\mu(\lambda) \cdot \theta_\nu(\lambda)
= \sum_{\pi \vdash n} c_{\mu,\nu;\pi} \, \theta_\pi(\lambda) \quad \text{for all $\lambda \vdash n$}.
\end{equation}
Note that $c_{\mu,\nu;\pi}$ depends on $\alpha$,
but, according to our convention, we omit the superscript,
when it does not bring any confusion.
It is important to keep in mind that the $c$'s are indexed by triples of partitions
of \emph{the same size}, while the $g$'s are indexed by any triple of partitions.

It turns out that the quantities $c_{\mu,\nu;\pi}$ can be expressed in terms of the quantities $g_{\mu,\nu;\tau}$.
To explain that, for any partition $\mu$, let $\tilde{\mu}$ denotes the partition
obtained from $\mu$ by erasing all parts equal to $1$.
Fix two partitions $\mu$ and $\nu$ of the same integer $n$; then
\[\Ch_{\tilde{\mu}} \cdot \Ch_{\tilde{\nu}} 
= \sum_\tau g_{\tilde{\mu},\tilde{\nu};\tau} \Ch_{\tau}. \]
But using the definition of $\Ch$ from Section \ref{eq:definition-Jack},
this implies that, for all $\lambda \vdash n$, one has:
\begin{multline*}
    \alpha^{-\frac{|\mu|-\ell(\mu)}{2}}\ z_{\tilde{\mu}}\ \theta_\mu(\lambda) \cdot
 \alpha^{-\frac{|\nu|-\ell(\nu)}{2}}\ z_{\tilde{\nu}}\ \theta_\nu(\lambda) \\
= \sum_{\substack{\tau, \\|\tau| \leq n}} g_{\tilde{\mu},\tilde{\nu};\tau}\ 
 \alpha^{\frac{|\tau|-\ell(\tau)}{2}}\ z_\tau\ \binom{n-|\tau|+m_1(\tau)}{m_1(\tau)}\
 \theta_{\tau 1^{n-|\tau|}}(\lambda).
 \end{multline*}
Every partition $\tau$ with $|\tau| \leq n$ can be written uniquely as $\tilde{\pi} 1^i$ where
$\pi$ is a partition of $n$ and $i \leq m_1(\pi)$.
Denoting 
$$d(\mu,\nu;\pi) := |\mu|-\ell(\mu) + |\nu|-\ell(\nu) - (|\pi| - \ell(\pi)),$$ 
one has
\[ \theta_\mu(\lambda) \cdot \theta_\nu(\lambda)
= \frac{\alpha^{d(\mu,\nu;\pi)/2}}{z_{\tilde{\mu}} z_{\tilde{\nu}}} \sum_{\pi \vdash n} 
\left( \sum_{0 \leq i \leq m_1(\pi)} g_{\tilde{\mu},\tilde{\nu};\tilde{\pi}
1^i}
\cdot z_{\tilde{\pi}} \cdot i! \cdot \binom{n-|\tilde{\pi}|}{i} \right) \theta_\pi(\lambda).\]
As this holds for all partitions $\lambda$ of $n$, 
by definition of structure constants, we have:
\begin{equation}\label{EqCG}
    c_{\mu,\nu;\pi} = \frac{\alpha^{d(\mu,\nu;\pi)/2}}{z_{\tilde{\mu}} z_{\tilde{\nu}}} 
\sum_{0 \leq i \leq m_1(\pi)} g_{\tilde{\mu},\tilde{\nu};\tilde{\pi} 1^i}    
\cdot z_{\tilde{\pi}} \cdot i! \cdot \binom{n-|\tilde{\pi}|}{i}.
\end{equation}
Using Theorem \ref{theo:struct-const} with $\deg_3$, we know that
$g_{\tilde{\mu},\tilde{\nu};\tilde{\pi} 1^i}$ is a polynomial of degree at most
$d(\mu,\nu;\pi)-i$.
We have thus proved the following result:
\begin{proposition}\label{PropStructureTheta}
    Let $\mu$, $\nu$ and $\pi$ be three partitions without parts equal to $1$.
    Then, \newline
    $\alpha^{-d(\mu,\nu;\pi)/2} c_{\mu 1^{n-|\mu|},\nu 1^{n-|\nu|};\pi 1^{n-|\pi|}}$ is 
    a polynomial in $n$ and $\gamma$ with rational coefficients, of total degree
    at most $d(\mu,\nu;\pi)$.

    Moreover, seen as a polynomial in $\gamma$, it has the same parity as $d(\mu,\nu;\pi)$.
\end{proposition}

\begin{corollary}\label{CorolConnectionSeries}
    The quantity $c_{\mu 1^{n-|\mu|},\nu 1^{n-|\nu|};\pi 1^{n-|\pi|}}$ is a polynomial
    in $n$ and $\alpha$.
    Moreover, it has degree at most $d(\mu,\nu;\pi)$ in $n$ and at most $d(\mu,\nu;\pi)$
    in $\alpha$ (the total degree may be bigger).
\end{corollary}
Applications of these statements are given in next section 
as well as in appendix~\ref{app:Matching-Jack_And_Matsumoto}.

\section{Special values of $\alpha$ and polynomial interpolation}
\label{SectSpecialValues}
\subsection{Case $\alpha=1$: symmetric group algebra}
\label{SubsectStructA1}

In the case $\alpha=1$, the structure constants considered in the previous section are linked
with the symmetric group algebra. Let $\Sym{n}$ denote the symmetric group of size $n$,
\emph{i.e.} the group of permutations of the set $[n]:=\{1,\dots,n\}$.
Recall that the cycle-type of a permutation $\sigma \in \Sym{n}$ is the
integer partition $\mu \vdash n$ obtained by sorting the lengths of the
cycles of $\sigma$.
We consider the group algebra $\QQ[\Sym{n}]$ of $\Sym{n}$ over the rational field $\QQ$.
Its center $Z(\QQ[\Sym{n}])$ is spanned linearly by the conjugacy classes, that is the elements
\[ \cl_\mu = \sum_{\substack{\sigma \in \Sym{n}, \\ \text{cycle-type}(\sigma)=\mu}} \sigma.\]
By a classical result of Frobenius (see \cite{Frobenius1900} or
\cite[(I,7.8)]{Macdonald1995}),
for any $\lambda \vdash n$,
\[
\frac{\Tr \rho^\lambda(\cl_\mu)}{\text{dimension of $\rho^\lambda$}} = \theta_\mu^{(1)}(\lambda),
\]
where $\rho^\lambda$ is the irreducible representation of the symmetric group associated 
to the Young diagram $\lambda$.
In other words: $\theta_\mu^{(1)}$ is the image of $\cl_\mu$ by the abstract Fourier transform,
which is an algebra morphism.
Hence, the structure constants of the algebra $Z(\QQ[\Sym{n}])$ with the basis $(\cl_\mu)_{\mu \vdash n}$
coincide with $c^{(1)}_{\mu,\nu;\pi}$.

These structure constants have been widely studied in the last fifty years in algebra
and combinatorics (they count some families of graphs embedded in orientable surfaces).
A famous result in this topic is due to Farahat and Higman 
\cite[Theorem 2.2]{FarahatHigmanCentreQSn}:
the quantity $c^{(1)}_{\mu 1^{n-|\mu|},\nu 1^{n-|\nu|};\pi 1^{n-|\pi|}}$ is a polynomial
in $n$.
Note that this is a consequence of Proposition \ref{PropStructureTheta}.\bigskip

Besides, structure constants of the center of $Z(\QQ[\Sym{n}])$
have a well-known and obvious combinatorial interpretation.

\begin{lemma}
    Let $\mu$, $\nu$ and $\pi$ be three partitions of the same integer $n$. 
    Fix a permutation $\sigma$ of cycle-type $\pi$.
    Then $c^{(1)}_{\mu,\nu;\pi}$ is the number of pairs of permutations
    $(\sigma_1,\sigma_2)$, such that $\sigma_1$ has cycle-type $\mu$,
    $\sigma_2$ has cycle-type $\nu$ and $\sigma_1 \cdot \sigma_2=\sigma$.
    \label{lem:Interpretation_StructureConstants_A1}
\end{lemma}

This can be used to compute $c^{(1)}_{\mu,\nu;\pi}$ in some particular cases,
which will be useful later.

\begin{lemma}
\label{lem:SC1}
We have the following identities:
\begin{enumerate}
\item
\label{eq:SC1-1} 
$c^{(1)}_{\mu, \nu; 1^n} = 0$ for any $\mu \neq \nu$;
\item
\label{eq:SC1-2}
$c^{(1)}_{(k 1^n),(k 1^n);(1^{k+n})} = \binom{k+n}{k} (k-1)!$.
\end{enumerate}
\end{lemma}

\begin{proof}
    Consider the first item.
    The only permutation $\pi$ of cycle type $(1^n)$ is $\sigma=\id$.
    Hence, the condition $\sigma_1 \cdot \sigma_2=\sigma$ corresponds
    to $\sigma_2=\sigma_1^{-1}$ and in particular $\sigma_1$ and 
    $\sigma_2$ must have the same cycle-type.

    Consider now the second item.
    As before, we must choose $\sigma=\id$.
    As $\sigma_2=\sigma_1^{-1}$ has always the same cycle-type as $\sigma_1$,
    the coefficient $c^{(1)}_{(k 1^n),(k 1^n);(1^{k+n})}$ is simply the number
    of permutations of $k+n$ of type $(k 1^n)$.
    This is well-known \cite[equation (1.2)]{SaganSymmetric} to be
    \[\frac{(k+n)!}{z_{(k 1^n)}}=\binom{k+n}{k} (k-1)!. \qedhere\]
\end{proof}

\begin{remark}
    The quantities $g^{(1)}_{\mu,\nu;\pi}$ also have a direct combinatorial interpretation
    in terms of partial permutations, see \cite{IvanovKerovPartialPermutations}.
\end{remark}

\subsection{Case $\alpha=2$: Hecke algebra of $(\Sym{2n},H_n)$}

\label{SubsectHecke}
An analogous interpretation of the structure constants exists in the case $\alpha=2$.
We explain it here, following the development given in \cite{GouldenJacksonMapsZonal}.

We can view the elements of the symmetric group $\Sym{2n}$ as permutations of the following set:
$\{1,\bar{1},\dots,n,\bar{n}\}$.
A subgroup formed by permutations $\sigma$ such that 
\[\overline{\sigma(i)} = \sigma(\overline{i})\qquad \text{for } i\in\{1,\dots,n\},\]
where by convention $\overline{\bar{j}}=j$, is called \emph{hyperoctahedral group} and is denoted by $H_n$.
We consider the subalgebra $\QQ[H_n \backslash \Sym{2n} / H_n] < \QQ[\Sym{2n}]$ of the elements
invariant by multiplication on the left or on the right by any element of $H_n$;
in other words
\[x \in \QQ[H_n \backslash \Sym{2n} / H_n] \stackrel{\text{\tiny def}}{\iff}
h x h'=x \quad \text{for all }h,h' \in H_n.\]
A non-trivial result is that this algebra is commutative.

The equivalence classes for the relation $x \sim hxh'$ (for $x \in \Sym{2n}$ and $h,h' \in H_n$)
are called \emph{double-cosets}.
They are naturally indexed by partitions of $n$, see \cite[(VII,2)]{Macdonald1995}.
We denote by $\cl^{(2)}_\mu \in \QQ[H_n \backslash \Sym{2n} / H_n]$ the sum
of all elements in the double coset corresponding to $\mu$.
The family $(\cl^{(2)}_\mu)_{\mu \vdash n}$ is a basis of $\QQ[H_n \backslash \Sym{2n} / H_n]$.

One can show (see \cite[Equation~(3) and (5)]{GouldenJacksonMapsZonal}) that
there exist some
orthogonal idempotents $E_\lambda$ such that:
\[ \cl^{(2)}_\mu = 2^n n! \sum_{\lambda \vdash n} \theta_\mu^{(2)}(\lambda)\ E_\lambda,\]
and one has
\begin{multline*}
    \cl^{(2)}_\mu \cdot \cl^{(2)}_\nu
= (2^n n!)^2 \sum_{\lambda \vdash n} \theta_\mu^{(2)}(\lambda)\ \theta_\nu^{(2)}(\lambda)\ E_\lambda\\
= (2^n n!)^2 \sum_{\pi \vdash n} \sum_{\lambda \vdash n} c_{\mu,\nu;\pi}^{(2)}\
\theta_\pi^{(2)}(\lambda)\ E_\lambda
=  (2^n n!) \sum_{\pi \vdash n} c_{\mu,\nu;\pi}^{(2)}\ \cl^{(2)}_\pi.
\end{multline*}
Hence, the structure constants $h_{\mu,\nu;\pi}$ of the algebra $\QQ[H_n \backslash \Sym{2n} / H_n]$
for the basis $(\cl^{(2)}_\mu)_{\mu \vdash n}$ are,
up to a factor $2^n n!$, the same as the ones of the algebra
$\F(\Young_n,\QQ)$ with the basis $(\theta_\mu^{(2)})_{\mu \vdash n}$.

In particular, Proposition \ref{PropStructureTheta} implies the following,
result, which is an analog of Farahet and Higman's result \cite{FarahatHigmanCentreQSn}
(a combinatorial proof of the polynomiality
has recently been given by O. Tout in \cite{Tout-Structure-constant-S2n-Hn}).
\begin{proposition}
    Let $\mu$, $\nu$ and $\pi$ be partitions without parts equal to $1$.
    The renormalized structure constant of the algebra $\QQ[H_n \backslash \Sym{2n} / H_n]$
    \[ \frac{h_{\mu 1^{n-|\mu|},\nu 1^{n-|\nu|};\pi 1^{n-|\pi|}}}{n!\ 2^n\ \sqrt{2}^{d(\mu,\nu,\pi)}}\]
    is a polynomial in $n$ of degree at most $d(\mu,\nu,\pi)$.
    Moreover, its coefficient of $n^{d(\mu,\nu,\pi)}$ is the same as in
    $c^{(1)}_{\mu 1^{n-|\mu|},\nu 1^{n-|\nu|};\pi 1^{n-|\pi|}}.$
    In particular,
    \begin{itemize}
        \item when $|\mu|-\ell(\mu) + |\nu|-\ell(\nu)=|\pi| - \ell(\pi)$, one has
            \[\frac{h_{\mu 1^{n-|\mu|},\nu 1^{n-|\nu|};\pi 1^{n-|\pi|}}}{n!\ 2^n}=
            c^{(1)}_{\mu 1^{n-|\mu|},\nu 1^{n-|\nu|};\pi 1^{n-|\pi|}}\]
            and this quantity is independent of $n$ ;
        \item when $|\mu|-\ell(\mu) + |\nu|-\ell(\nu)=|\pi| - \ell(\pi) - 1$,
            \[\frac{h_{\mu 1^{n-|\mu|},\nu 1^{n-|\nu|};\pi 1^{n-|\pi|}}}{n!\ 2^n}\]
            is independent of $n$.
    \end{itemize}
\end{proposition}
\begin{proof}
    The claim that the renormalized structure constant mentioned above is a polynomial and
    the bound on its degree follow from 
    Proposition \ref{PropStructureTheta} specialized to $\alpha=2$.
    The dominant coefficient is a polynomial in $\gamma$ of degree $0$,
    so it is the same for $\gamma \in \{0, -1/\sqrt{2}\}$, that is 
    $\alpha \in \{1,2\}$.
    The first item follows immediately.

    In the second item, we consider the case $d(\mu,\nu,\pi)= 1$.
    So, 
    \[\frac{h_{\mu 1^{n-|\mu|},\nu 1^{n-|\nu|};\pi 1^{n-|\pi|}}}{n!\ 2^n\ \sqrt{2}}\]
    is an affine function of $n$ with the same linear coefficient as 
    $c^{(1)}_{\mu 1^{n-|\mu|},\nu 1^{n-|\nu|};\pi 1^{n-|\pi|}}$.
    But the latter is identically equal to $0$.
    Indeed, the fact that the sign of permutations is a group morphism implies that
    \[c^{(1)}_{\mu 1^{n-|\mu|},\nu 1^{n-|\nu|};\pi 1^{n-|\pi|}}=0\]
    whenever $d(\mu,\nu;\pi)$ is odd.
\end{proof}

Besides, Goulden and Jackson \cite{GouldenJacksonMapsZonal} described the coefficients $h_{\mu,\nu; \pi}$ combinatorially.
Let $\mathcal{F}_\mathcal{S}$ be the set of all (perfect) matchings on a set $\mathcal{S}$,
that is partitions of $\mathcal{S}$ in pairs.
When $\mathcal{S}$ is the set $\{1,2\dots,2n\}$, we also denote $\mathcal{F}_\mathcal{S}$ by $\mathcal{F}_n$.

For $F_1,F_2 \in \mathcal{F}_\mathcal{S}$,
let $G(F_1,F_2)$ be the multigraph with vertex-set $\mathcal{S}$ whose edges are formed by the pairs in $F_1,\dots,F_k$.
Because of the natural bicoloration of the edges,
the connected components of $G(F_1, F_2)$ are cycles of even length.
Let the list of their lengths in weakly decreasing order be $(2\theta_1, 2 \theta_2, \dots) = 2\theta$, and define $\Lambda$ by $\Lambda(F_1, F_2) = \theta$.

\begin{lemma}[{\cite[Lemma 2.2.]{GouldenJacksonMapsZonal}}]
\label{lem:ZonalSC}
Let $F_1, F_2$ be two fixed matchings in $\mathcal{F}_n$ such that $\Lambda(F_1,F_2) = \pi$, where $\pi \vdash n$. Then, for any $\mu, \nu \vdash n$ we have
\[ h_{\mu,\nu;\pi} = 2^n n! |\{F_3 \in \mathcal{F}_n: \Lambda(F_1, F_3) = \mu, \Lambda(F_2, F_3) = \nu\}|.\]
In particular 
\[ c^{(2)}_{\mu,\nu;\pi} = |\{F_3 \in \mathcal{F}_n: \Lambda(F_1, F_3) = \mu, \Lambda(F_2, F_3) = \nu\}|.\]
\end{lemma}

From this lemma one can evaluate some special cases of structure constants which will be helpful in the next subsection.

\begin{lemma}
\label{lem:SC2}
We have the following identities:
\begin{enumerate}
\item
\label{eq:SC2-1} 
$c^{(2)}_{\mu, \nu; 1^n} = 0$ for any $\mu \neq \nu$;
\item
\label{eq:SC2-2}
$c^{(2)}_{(k 1^n),(k 1^n);(1^{k+n})} = \binom{k+n}{k}2^{k-1}(k-1)!$.
\end{enumerate}
\end{lemma}

\begin{proof}
The first item is immediate from Lemma \ref{lem:ZonalSC},
since whenever $F_1$ and $F_2$ are matchings such that $\Lambda(F_1,F_2) = 1^n$ then clearly $F_1 = F_2$.

Let $F_1, F_2$ be two fixed matchings such that $\Lambda(F_1,F_2) = (1^{k+n})$ (hence $F_1 = F_2$).
We are looking for the number of matchings $F_3$ such that $\Lambda(F_1,F_3)=\Lambda(F_2,F_3)=(k 1^n)$.
Of course, as $F_1=F_2$, it is the number of matchings $F_3$ with $\Lambda(F_1,F_3)=(k 1^n)$.
This number does not depend on $F_1$.
Using \cite[Lemma 2.4]{NousZonal}, we know that the numbers of pairs $(F,F_3)$
with $\Lambda(F,F_3)=(k 1^n)$ is given by 
\[\frac{(2n+2k)!}{z_{(k 1^n)} 2^{n+1}}.\]
As there are $\frac{(2n+2k)!}{2^{n+k} (n+k)!}$ matchings in $\mathcal{F}_{n+k}$, that is possible values for $F$,
for a fixed $F_1$, the number of matchings $F_3$ with $\Lambda(F_1,F_3)=(k 1^n)$ is
\[\frac{2^{n+k} (n+k)!}{z_{(k 1^n)} 2^{n+1}}=\binom{k+n}{k}2^{k-1}(k-1)!.\qedhere\]
\end{proof}

\subsection{A technical lemma obtained by polynomial interpolation}

In order to study asymptotics of large Young diagrams we need to understand some special cases of structure constants. 
Here, we present some technical, but useful lemma about them:

\begin{lemma}
\label{lem:SC}
We have the following identities:
\begin{enumerate}
\item
\label{eq:SC-1} 
$g_{\mu,\nu;\rho} = \delta_{\rho, \mu \cup \nu}$ for $|\rho| + \ell(\rho) \geq |\mu| + \ell(\mu) + |\nu| + \ell(\nu)$;
\item
\label{eq:SC-2}
$g_{\mu,\nu;1^k} = 0$ for $\tilde{\mu} \neq \tilde{\nu}$ and 
$2k \geq |\mu| + \ell(\mu) + |\nu| + \ell(\nu) - 2$; 
\item 
\label{eq:SC-3}
$g_{(k),(k);1^k} = k$;
\item
\label{eq:SC-4} 
$g_{(k),(l);\rho} = 0$ for $|\rho| + \ell(\rho) = k+l+1$.
\end{enumerate}
\end{lemma}

\begin{proof}
Let 
$x(\mu,\nu;\rho) :=  |\mu| + \ell(\mu) + |\nu| + \ell(\nu) -  (|\rho| + \ell(\rho))$.

By Theorem \ref{theo:struct-const} we know that $g_{\mu,\nu;\rho}$ is a polynomial in $\gamma$
of degree at most $x(\mu,\nu;\rho)$ and of the same parity as
$x(\mu,\nu;\rho)$.
Hence, if one wants to prove that for some particular partitions,
$g_{\mu,\nu;\rho}$ is identically equal to some constant $c$,
it is enough to prove that:
\begin{itemize}
\item $g^{(1)}_{\mu,\nu;\rho} = c$ in the case $x(\mu,\nu;\rho) = 0$;
\item $g^{(2)}_{\mu,\nu;\rho} = c$ in the case $x(\mu,\nu;\rho) = 1$
    (necessarily, $c=0$ in this case);
\item $g^{(1)}_{\mu,\nu;\rho} = g^{(2)}_{\mu,\nu;\rho} = c$ in the case $x(\mu,\nu;\rho) = 2$.
    \bigskip
\end{itemize}

Applying this idea, we see that
the first item holds true, since this is true for $\alpha = 1$
\cite[Proposition 4.9.]{IvanovOlshanski2002}.\bigskip

Consider now the second item.
In this case, $x(\mu,\nu;\rho) \le 2$, hence
we need to prove that $g^{(1)}_{\mu,\nu;1^k} = g^{(2)}_{\mu,\nu;1^k} = 0$.
We know from Lemma \ref{lem:SC1} \eqref{eq:SC1-1}
that $c^{(1)}_{\pi,\rho;1^n} = 0$ for any pair of different partitions $\pi, \rho \vdash n$.
It means that for $n$ large enough, thanks to \eqref{EqCG}, we have the following equation:
\[0 = \sum_{0 \leq i \leq k+1} g^{(1)}_{\tilde{\pi},\tilde{\rho}; (1^i)} \cdot i! \cdot \binom{n}{i},\]
hence \[ g^{(1)}_{\tilde{\pi},\tilde{\rho}; (1^i)} = 0 \quad \text{for each $0\leq i\leq k+1$.}\]
The same holds for $g^{(2)}$ using Lemma \ref{lem:SC2}.
This finishes the case where $\mu$ and $\nu$ have no parts equal to $1$.

The general case follows, since for any $\lambda \vdash n$ one has
\[ \Ch_\mu(\lambda) = (n-|\tilde{\mu}|)_{m_1(\mu)}\Ch_{\tilde{\mu}}(\lambda).\bigskip\]

In order to prove the third item,
we need to prove the equalities 
\[g^{(1)}_{(k),(k);1^k} = k = g^{(2)}_{(k),(k);1^k}
\text{ (as $x( (k),(k);1^k ))=2$)}.\]
Here, we use Equation \eqref{EqCG} again. We have that
\[ c^{(1)}_{(k 1^n),(k 1^n);(1^{k+n})} = \frac{1}{k^2} 
\sum_{0 \leq i \leq k+n} g^{(1)}_{(k),(k);(1^i)}    
i! \binom{k+n}{i}. \]
Moreover, the summation index can be restricted to $i \le k$
as $g^{(1)}_{(k),(k);(1^i)}=0$ for $i>k$.
By Lemma \ref{lem:SC1} \eqref{eq:SC1-2} one has that 
\[c^{(1)}_{(k 1^n),(k 1^n);(1^{k+n})} = \binom{k+n}{k} (k-1)!. \]
It gives us
\[ \binom{k+n}{k} (k-1)! = \frac{1}{k^2} 
\sum_{0 \leq i \leq k} g^{(1)}_{(k),(k);(1^i)}    
i! \binom{k+n}{i}\]
and since both sides of the equation are polynomials in $n$,
the equation $g^{(1)}_{(k),(k);(1^k)} = k$ follows
(all others $g^{(1)}_{(k),(k);(1^i)}$ vanish).
The same proof works for $\alpha=2$, using Lemma \ref{lem:SC2} \eqref{eq:SC2-2}.

Finally, let us prove the last item.
As $x( (k),(l),\rho)=1$ in this case,
it is enough to prove that $g^{(2)}_{(k),(l);\rho} = 0$.
Here, we shall use a different approach.
It is proved in \cite[Theorem 5.3]{NousZonal} that the coefficient of
\[
 \left(R_2^{(2)}\right)^{s_2}
\left(R_3^{(2)}\right)^{s_3} \cdots
\]
in Kerov's expansion of $\Ch^{(2)}_{(k)}\Ch^{(2)}_{(l)} - \Ch^{(2)}_{(k,l)}$ is,
up to some constant factor,
the number of maps with $2$ faces, $k+l$ edges, $2s_2+ 3s_3 + \cdots$ vertices
and some additional properties
(the details of which is irrelevant here).
But, using the theory of Euler characteristic, such maps may exist only if
\[2-(k+l)+2s_2+ 3s_3 + \cdots \le 2,\]
that is
\[2s_2+ 3s_3 + \cdots \le k+l.\]
This implies that 
\[\deg_1\left(\Ch^{(2)}_{(k)}\Ch^{(2)}_{(l)} - \Ch^{(2)}_{(k,l)}\right) = k+l. \]                          
Expanding it in the Jack characters basis one has                               
\[ \Ch^{(2)}_{(k)}\Ch^{(2)}_{(l)} = \Ch^{(2)}_{(k,l)} + \sum_{|\rho|+\ell(\rho) \leq k+l} g^{(2)}_{(k),(l);\rho} \Ch^{(2)}_\rho,\]
which finishes the proof.
\end{proof}

\begin{remark}
The equality $g^{(1)}_{(k),(k);(1^k)} = k$ was also established
by Kerov, Ivanov and Olshanski \cite[Proposition 4.12.]{IvanovOlshanski2002},
using their combinatorial interpretation for $g^{(1)}_{\mu,\nu;\rho}$.
\end{remark}

\section{Jack measure: law of large numbers}
\label{sect:FirstOrder}

The purpose of this Section is to prove Theorem~\ref{theo:RYD-limit}.
As in \cite{IvanovOlshanski2002},
the key point is to prove the convergence of polynomial functions.

\subsection{Convergence of polynomial functions}
Let us recall equation \eqref{eq:exp_Ch},
which gives us the expectation of Jack characters with respect to Jack measure:
\begin{equation*}
    \esper_{\PP_n^{(\alpha)}}(\Ch^{(\alpha)}_\mu)=\begin{cases}
    n(n-1)\cdots(n-k+1) & \text{if }\mu=1^k \text{ for some }k \leq n,\\
    0   & \text{otherwise.}
\end{cases}
\end{equation*}
As $\Ch_\mu$ is a linear basis of $\Pola$, it implies the following lemma
(which is an analogue of \cite[Theorem 5.5]{OlshanskiJackPlancherel} with
another gradation).
\begin{lemma}
    Let $F$ be an $\alpha$-polynomial function. Then
$\esper_{\PP_n^{(\alpha)}}(F)$
    is a polynomial in $n$ of degree at most $\deg_1(F)/2$.
    \label{LemDegEsper}
\end{lemma}
\begin{proof}
    It is enough to verify this lemma on the basis $\Ch_\mu$ because
    of the remark in Section \ref{SubsectDegCh}.
    But in this case $\esper_{\PP_n^{(\alpha)}}(F)$
    is explicit (see formula \eqref{eq:exp_Ch}) and the lemma is obvious
    (recall that $\deg_1(\Ch_\mu)=|\mu|+\ell(\mu)$, see Section
    \ref{SubsectBound1}).
\end{proof}

Informally, smaller terms for $\deg_1$ are asymptotically negligible.
We can now prove the following weak convergence result:

\begin{proposition} \label{PropAPlanchWeakConv}
Let $(\lambda_{(n)})_{n \geq 1}$ be a sequence of random partitions 
distributed with Jack measure.
    For any 1-polynomial function $F \in \Pol^{(1)}$, when $n \to \infty$, one has
    \[ F\bigg(D_{1/\sqrt{n}} \big(A_\alpha(\lambda_{(n)})\big)\bigg) \xrightarrow{\PP_n^{(\alpha)}} 
    F( \Omega) , \]
    where $\xrightarrow{\PP_n^{(\alpha)}}$ means convergence in probability and $\Omega$ is given by \eqref{eq:DefOmega}.
\end{proposition}

\begin{proof}
    As $(R^{(1)}_k)_{k \geq 2}$ is an algebraic basis of $\Pol^{(1)}$,
    it is enough to prove the proposition for any $R^{(1)}_k$.

Let $\mu$ be partition.
By Corollary \ref{corol:dominant_deg1}
\begin{multline}\label{EqTopDeg1}
    \prod_{i \leq \ell(\mu)} R_{\mu_i+1} =
    \Ch_\mu + \text{ terms of degree at most }\\
    |\mu|+\ell(\mu)-1\text{ with respect to }\deg_1.
\end{multline}
Together with Lemma~\ref{LemDegEsper} and the formula \eqref{eq:exp_Ch} for $\esper_{\PP_n^{(\alpha)}}(\Ch_\mu)$,
this implies:
\[ \esper_{\PP_n^{(\alpha)}} \left(\prod_{i \leq \ell(\mu)}R_{\mu_i+1}\right) = \begin{cases}
    n(n-1)\cdots(n-k+1) + O(n^{k-1}) & \text{if }\mu=1^k \text{ for some }k; \\
    o(n^{\frac{|\mu|+\ell(\mu)}{2}}) & \text{otherwise.}
\end{cases}\]
In particular
\begin{align*}
    \esper_{\PP_n^{(\alpha)}} (R_k(D_{1/\sqrt{n}}(\lambda_{(n)}))) &=
\frac{1}{n^{k/2}}
    \esper_{\PP_n^{(\alpha)}} (R_k) = \delta_{k,2} + O\left(\frac{1}{\sqrt{n}}\right), \\
    \Var_{\PP_n^{(\alpha)}} (R_k(D_{1/\sqrt{n}}(\lambda_{(n)}))) &= \frac{1}{n^k}
\left( 
    \esper_{\PP_n^{(\alpha)}} \big((R_k)^2\big) - \esper_{\PP_n^{(\alpha)}} (R_k)^2
\right)
    = O\left(\frac{1}{n}\right).
\end{align*}
Thus, for each $k$, $R_k(D_{1/\sqrt{n}}(\lambda_{(n)}))$ converges in probability
towards $\delta_{k,2}$. 
But, by definition 
\[R_k(D_{1/\sqrt{n}}(\lambda_{(n)}))=R^{(1)}_k \bigg(D_{1/\sqrt{n}}\big(A_\alpha(\lambda_{(n)})\big)\bigg)\]
and $(\delta_{k,2})_{k \geq 2}$ is the sequence of free cumulants of the continuous diagram $\Omega$
(see \cite[Section 3.1]{Biane2001}),
\emph{i.e.} 
\[ \delta_{k,2} = R_k^{(1)}(\Omega).\qedhere\]
\end{proof}

\subsection{Shape convergence}
In the previous Section, we proved that evaluations of polynomial functions
at $D_{1/\sqrt{n}} \big(A_\alpha(\lambda_{(n)})\big)$
converge towards the evaluation at $\Omega$.
Ivanov and Olshanski has established that,
if one can prove that the support of
these renormalized Young diagrams lies in some compact,
this would imply the uniform convergence, that is Theorem~\ref{theo:RYD-limit}.

The following technical lemma,
proved by Fulman \cite[Lemma 6.6]{FulmanFluctuationChA2}
will allow us to conclude:

\begin{lemma}
    \label{lem:FiniteSupport}
Suppose that $\alpha > 0$. Then
\begin{enumerate}
\item
\[ \PP_n^{(\alpha)}\left(\lambda_1\geq 2 e\sqrt{\frac{n}{\alpha}}\right) \leq \alpha n^2 4^{-e\sqrt{\frac{n}{\alpha}}}, \]
\item
\[ \PP_n^{(\alpha)}(\lambda_1' \geq 2 e\sqrt{n\alpha}) \leq \frac{n^2}{\alpha} 4^{-e\sqrt{n\alpha}}. \]
\end{enumerate}
In particular
\[\lim_{n \to \infty} \PP_n^{(\alpha)} \left(
 \left[-\frac{\lambda'_1}{\sqrt{n}};\frac{\lambda_1}{\sqrt{n}}\right] \subseteq \left[-2e\sqrt{\alpha}, \frac{2e}{\sqrt{\alpha}}\right]
 \right) =1. \]
\end{lemma}\medskip

{\em End of proof of Theorem~\ref{theo:RYD-limit}.}
    It follows from Proposition \ref{PropAPlanchWeakConv} and Lemma
    \ref{lem:FiniteSupport}
    by the same argument as the one given in
    \cite[Theorem 5.5]{IvanovOlshanski2002}.
\qed

\section{Jack measure: central limit theorem for Jack characters}
\label{sect:CLT}

In this section we prove the central limit theorem for Jack characters
(Theorem \ref{theo:FluctuationsJackCharacters})
and the bound of the speed of convergence in this theorem
(Theorem \ref{theo:SpeedConvergence}).

\subsection{Multivariate Stein's method}

As explained in introduction, our main tool will be a multivariate analog
of the so-called {\em Stein's method}
due to Reinert and R{\"o}llin \cite{ReinertRollin2009}.
For any discrete random variables $W, W^*$ with values in $\RR^d$, we say that the pair $(W,W^*)$ is \emph{exchangeable} if for any $w_1,w_2 \in \RR^d$ one has $\PP(W = w_1, W^* = w_2) = \PP(W = w_2, W^* = w_1)$. Let $\esper^W(\cdot)$ denotes the conditional expected value given $W$.
The theorem of Reinert and R{\"o}llin is the following \cite[Theorem 2.1]{ReinertRollin2009}:
\begin{theorem}[multivariate Stein's theorem]
\label{theo:MultivariateStein}
Let $(W,W^*)$ be an exchangeable pair of $\RR^d$-valued random variables such that $\esper(W) = 0$ and $\esper(WW^t) = \varSigma$, where $\Sigma \in M_{d \times d}(\RR)$ is symmetric and positive definite matrix. Suppose that $\esper^W(W^* - W) = - \Lambda W$, where $\Lambda \in M_{d \times d}(\RR)$ is invertible. Then, if $Z$ is a $d$-dimensional standard normal distribution, we have for every three times differentiable function $h:\RR^d \to \RR$,
\begin{equation}
\left| \esper h(W) - \esper h(\varSigma^{1/2}Z) \right| \leq \frac{|h|_2}{4}A + \frac{|h|_3}{12}B,
\end{equation}
where, using the notation $\lambda^{(i)} := \sum_{1 \leq m \leq d}|(\Lambda^{-1})_{m,i}|$,
\begin{align*}
    |h|_n &= \sup_{i_1,\dots,i_n} \left\lVert \frac{\partial^n}{\partial x_{i_1} \cdots \partial x_{i_n}}h\right\rVert, \\
 A &= \sum_{1 \leq i,j \leq d} \lambda^{(i)} \sqrt{\Var \esper^W(W_i^* - W_i)(W_j^* - W_j)}, \\
 B &= \sum_{1 \leq i,j,k \leq d} \lambda^{(i)} \esper |(W_i^* - W_i)(W_j^* - W_j)(W_k^* - W_k)|.
\end{align*}
\end{theorem}

Let $d$ be a positive integer.
For $k \ge 2$, as in the statement of Theorem~\ref{theo:FluctuationsJackCharacters}, we define
the following function of Young diagrams of size $n$:
\[ W_k = n^{-k/2} \sqrt{k} \, \theta^{(\alpha)}_{(k,1^{n-k})} = n^{-k/2} \sqrt{k}^{-1} \Ch_{(k)}.\]
It can be seen as a random variable on the probability space
of Young diagrams of size $n$ endowed with Jack measure.
We also consider the corresponding random vector
\[ \tilde{W}_d = (W_2,\dots,W_{d+1}).\]
Theorem~\ref{theo:FluctuationsJackCharacters} states that $\tilde{W}_d$
converges in distribution towards a vector of independent Gaussian random variables.
Therefore we would like to apply the theorem above to this $d$-uplet of random variables.
In the next sections, we shall contruct an exchangeable pair and check the hypothesis
of Theorem~\ref{theo:MultivariateStein}.

\subsection{An exchangeable pair}

The first step consists in building a $d$-tuple $\tilde{W}_d^*$,
such that $(\tilde{W}_d,\tilde{W}_d^*)$ is an exchangeable pair.
The construction that we will describe here is due to Fulman
\cite{FulmanFluctuationChA2}.

His construction uses Markov chains, so let us begin by fixing some terminology.
Let $X$ be a finite set.
A Markov chain $M$ on $X$ is the data of transition probability $M(x,y)$
indexed by pairs of elements of $X$ with
\[M(x,y) \ge 0 \text{ and } \sum_{y \in X} M(x,y) =1.\]
If $x$ is a random element of $X$ distributed with probability $\PP$,
then, applying once the Markov chain $M$, we obtain by definition
a random element $y$ of $X$, defined on the same probability space as $x$,
whose conditional distribution is given by:
\[P ( y= y_0 | x=x_0 ) = M(x_0,y_0).\]
Using the notation above,
the Markov chain $M$ is termed {\em reversible} with respect to $\PP$ if
the distribution of $(x,y)$ is the same as the distribution of $(y,x)$, or
equivalently, for any $x_0,y_0$ in $X$,
\[\PP(\{x_0\}) M(x_0,y_0) = \PP(\{y_0\}) M(y_0,x_0).\]
Reversible Markov chains can be used to construct exchangeable pairs as follows.
Let $M$ be a reversible Markov chain on a finite set $X$ with respect to
a probability measure $\PP$.
Consider also a $\RR^d$-valued function $W$ on $X$.
We consider a random element $x$ distributed with respect to $\PP$
and $y$ obtained by applying the Markov chain $M$ to $\PP$.
Then, directly from the definition, one sees that $(W(x),W(y))$ is exchangeable.

So, to construct an exchangeable pair for $\tilde{W}_d$,
it is enough to construct a reversible Markov chain 
with respect to Jack measure.
We present now Fulman's construction of such a Markov chain.

Let $\tau \vdash n-1$ and $\lambda \vdash n$.
If $\tau$ is not contained in $\lambda$ (as Young diagrams),
then define $\phi^{(\alpha)}(\lambda/\tau) =0$.
Otherwise, denote by $\lambda / \tau$ the box which is in $\lambda$ and not in $\tau$.
Let $C_{\lambda / \tau}$ ($R_{\lambda / \tau}$, respectively) be the column (row, respectively)
of $\lambda$ that contains $\lambda / \tau$.
We define
\[ \phi^{(\alpha)}(\lambda/\tau) = \prod_{\Box \in C_{\lambda / \tau} \setminus R_{\lambda / \tau}} \frac{(\alpha a_\lambda(\Box) + \ell_\lambda(\Box)
+1) 
(\alpha a_\tau(\Box) + \ell_\tau(\Box)+\alpha)}{(\alpha a_\lambda(\Box) + \ell_\lambda(\Box)
+\alpha) 
(\alpha a_\tau(\Box) + \ell_\tau(\Box)+1)}.\]
Let
\[ c_\lambda^{(\alpha)} = \prod_{\Box \in \lambda} (\alpha a(\Box) + \ell(\Box)
+1)\]
and
\[ (c'_\lambda)^{(\alpha)} = \prod_{\Box \in \lambda} (\alpha a(\Box) + \ell(\Box)
+\alpha).\]
We recall that $j_\lambda^{(\alpha)}=c_\lambda^{(\alpha)}(c'_\lambda)^{(\alpha)}.$
For $\lambda, \rho \vdash n$ we define two functions:
\begin{equation}
\label{eq:TransitionM}
M^{(\alpha)}(\lambda,\rho) = \frac{(c'_\lambda)^{(\alpha)}}{n \alpha c_\rho^{(\alpha)}} \sum_{\tau \vdash n-1} \frac{\phi^{(\alpha)}(\lambda/\tau)\phi^{(\alpha)}(\rho/\tau) c_\tau^{(\alpha)}}{(c'_\tau)^{(\alpha)}}
\end{equation}
and
\begin{equation}
\label{eq:TransitionL}
L^{(\alpha)}(\lambda,\rho) = \frac{1}{\alpha^n n! j_\rho^{(\alpha)}}\sum_{\mu \vdash n}(z_\mu)^2 \alpha^{2\ell(\mu)}\theta_\mu(\lambda)\theta_\mu(\rho)\theta_\mu((n-1,1)).
\end{equation}

As explained by Fulman \cite{FulmanFluctuationChA2}, 
both $M^{(\alpha)}$ and $L^{(\alpha)}$ are defined to be a deformation of a certain Markov chain which is reversible with respect to Plancherel measure. Roughly speaking, this Markov chain remove one box from a given Young diagram with certain probability and add another box with some probability to obtain a new Young diagram of the same size as the one from which we started. Fulman \cite{FulmanFluctuationChA2} proved the following:

\begin{proposition}{\cite[Section 4]{FulmanFluctuationChA2}}
\label{prop:FulmanComputations}
\begin{enumerate}
\item
If $\rho \neq \lambda$ then
\[ L^{(\alpha)}(\lambda,\rho) = \frac{\alpha(n-1) + 1}{\alpha(n-1)} M^{(\alpha)}(\lambda,\rho);\]
\item
Let $\lambda \vdash n$. Then
\[ \sum_{\rho \vdash n} L^{(\alpha)}(\lambda, \rho) = \sum_{\rho \vdash n} M^{(\alpha)}(\lambda, \rho) = 1; \]
\item
$L^{(\alpha)}$ (hence $M^{(\alpha)}$ as well) is reversible with respect to Jack measure.
\end{enumerate}
\end{proposition}

For more details about this construction (and in particular, the intuition behind it),
we refer to Fulman \cite{FulmanFluctuationChA2}.

\subsection{Checking hypotheses}
\label{subsec:CheckingHypotheses}

Recall that we have defined
\[ W_k = n^{-k/2} \sqrt{k}^{-1} \Ch_{(k)} \]
and the random vector
\[ \tilde{W}_d = (W_2,\dots,W_{d+1})\]
on the probability space of Young diagrams of size $n$ endowed with Jack measure.
Let $\lambda$ be a random partition distributed according to Jack measure,
and $\lambda^*$ ($\lambda'$, respectively) be obtained from $\lambda$ by applying
the Markov chain $M^{(\alpha)}$ ($L^{(\alpha)}$, respectively).
By a small abuse of notations, 
set $\tilde{W}_d:=\tilde{W}_d(\lambda)$, 
$\tilde{W}_d^*:=\tilde{W}_d(\lambda^*)$
and $\tilde{W}'_d:=\tilde{W}_d(\lambda')$.
We shall prove now, that for any $d \in \N$,
the pair $(\tilde{W}_d, \tilde{W}_d^*)$ of random vectors 
satisfies conditions of Theorem \ref{theo:MultivariateStein}.\medskip 

First, note that $\esper_n^{(\alpha)}(\tilde{W}_d)=0$ from equation~\eqref{eq:exp_Ch}.

We should now verify the hypothesis involving $(\esper_n^{(\alpha)})^{\tilde{W}_d}(W_k^*)$.
Let us begin with two technical known statements about Jack polynomials.
\begin{lemma}
\label{lem:identities}
\begin{enumerate}
\item{\cite[Page 382]{Macdonald1995}}
\label{eq:orthogonality}
\[ \sum_{\rho \vdash n} \frac{\theta_\mu(\rho)\theta_\nu(\rho)}{j^{(\alpha)}_\rho} =  \frac{\delta_{\mu,\nu}}{z_\mu \alpha^{\ell(\mu)}};\]
\item{\cite[Page 107]{Stanley1989}}
\label{eq:(n-1,1)}
\[ \theta_\mu((n-1,1)) = \frac{\alpha^{n-\ell(\mu)}n!}{z_\mu}\frac{(\alpha(n-1)+1)m_1(\mu) - n}{\alpha n(n-1)}.\]
\end{enumerate}
\end{lemma}

Recall that $(\esper_n^{(\alpha)})^{\tilde{W}_d}$ denotes the conditional
expectation given $\tilde{W}_d$.
Similarly, we denote by $(\esper_n^{(\alpha)})^{\lambda}$ the conditional
expectation given $\lambda$.

\begin{proposition}
\label{prop:SteinCondition1}
Let $d \in \N$ and let $2 \leq k \leq d+1$. Then, one has
\begin{align*}
    (\esper_n^{(\alpha)})^{\tilde{W}_d}(W_k^*) &=
(\esper_n^{(\alpha)})^{\lambda}(W_k^*) = \left(1-\frac{k}{n}\right)W_k ; \\
    (\esper_n^{(\alpha)})^{\tilde{W}_d}(W_k') &=
 (\esper_n^{(\alpha)})^{\lambda}(W_k') = \left(1-\frac{k(\alpha(n-1)+1)}{\alpha n(n-1)}\right)W_k.
 \end{align*}
\end{proposition}

\begin{proof}
    By definition, 
    \[(\esper_n^{(\alpha)})^{\lambda}(W_k^*)=
    \sum_{\lambda^*} M(\lambda,\lambda^*) W_k^*(\lambda^*).\]
    From \cite[Proposition 6.2.]{FulmanFluctuationChA2},
    we have that $(\theta_\mu(\lambda))_{\lambda \vdash n}$ is an eigenvector of $M^{(\alpha)}$ with eigenvalue
\[d_\mu := 1 + \frac{\alpha(n-1)}{\alpha(n-1)+1}\left( \frac{z_\mu}{\alpha^{n-\ell(\mu)}n!}\theta_\mu((n-1,1)) - 1\right). \]
Using Lemma~\ref{lem:identities}-(2), this eigenvalue can be simplified to 
(surprisingly, this was not noticed by Fulman)
\[d_\mu=                                                                  
 \frac{(\alpha(n-1)+1)m_1(\mu)}{n(\alpha (n-1) + 1)} = \frac{m_1(\mu)}{n}.\]

As $(W_k^*(\lambda))_{\lambda \vdash n}$ is a multiple of 
$(\theta_\mu(\lambda))_{\lambda \vdash n}$,
it is also an eigenvector of $M^{(\alpha)}$, with the same eigenvalue.
Hence, 
\[(\esper_n^{(\alpha)})^{\lambda}(W_k^*) (\lambda)= d_{(k,1^{n-k})}
W_k^*(\lambda).\]
 For $\mu=(k,1^{n-k})$, we have $d_{(k,1^{n-k})}=1-\frac{k}{n}$,
 which finishes the proof of the second equality of the first statement.

 In particular, we see that $(\esper_n^{(\alpha)})^{\lambda}(W_k^*)$
 depends only on $W_k$ and hence on $\tilde{W}_d$.
 As, conversely,  $\tilde{W}_d$ is determined by $\lambda$, one has
 \[(\esper_n^{(\alpha)})^{\tilde{W}_d}(W_k^*) =
 (\esper_n^{(\alpha)})^{\lambda}(W_k^*). \]

The statement for $W_k'$ follows easily from the one for $W_k^\star$,
using Proposition~\ref{prop:FulmanComputations}.
\end{proof}

\begin{corollary}
\label{cor:SteinCondition1}
Let $d \in \N$. Then
\[ (\esper_n^{(\alpha)})^{\tilde{W}_d}(\tilde{W}_d^* - \tilde{W}_d) = -\Lambda \tilde{W}_d, \]
with $\Lambda_{i,j} = \delta_{i,j}\frac{i+1}{n}$. In particular,
with the notation of Theorem~\ref{theo:MultivariateStein},
\[ \lambda^{(i)} = \frac{n}{i+1}.\]
\end{corollary}

\begin{proof}
It is a straightforward consequence of Proposition \ref{prop:SteinCondition1}.
\end{proof}

Consider now  $\varSigma = \esper_n^{(\alpha)}(\tilde{W}_d\tilde{W}_d^t)$
as in the statement of Theorem~\ref{theo:MultivariateStein}.
This matrix is symmetric by definition, but one has to check that
it is positive definite.

For a matrix $A$, let $\lVert A \rVert:=\max_{i,j} |A_{i,j}|$ denotes
the supremum norm on matrices.

\begin{proposition}
\label{prop:SteinCondition2}
There exists a constant $A_{d,\alpha}$ which depends only on $d$ and $\alpha$ such that for any $n \geq A_{d,\alpha}$ the matrix $\varSigma$ is positive definite. Moreover,
\[ \lVert \varSigma^{1/2} - \Id \rVert = O(n^{-1/2}).\]
\end{proposition}

\begin{proof}
Strictly from the definition we have that
 \begin{multline} 
 \varSigma_{i,j} = \esper_n^{(\alpha)} \left( \frac{1}{\sqrt{i+1}\sqrt{j+1} n^{(i+j+2)/2}} \Ch_{(i+1)} \Ch_{(j+1)} \right) \\
 = \frac{1}{\sqrt{i+1}\sqrt{j+1} n^{(i+j+2)/2}} \sum_\lambda g_{(i+1),(j+1);\lambda} \esper_n^{(\alpha)}(\Ch_\lambda). 
 \end{multline}
Since
\[\esper_{\PP_n^{(\alpha)}}(\Ch_\mu)=\begin{cases}
    (n)_k & \text{if }\mu=1^k \text{ for some }k \leq n,\\
    0   & \text{otherwise,}
\end{cases}\]
we have that
\[  \varSigma_{i,j} = \frac{1}{\sqrt{i+1}\sqrt{j+1} n^{(i+j+2)/2}} \sum_l g_{(i+1),(j+1);(1^l)} (n)_l, \]
and using items \eqref{eq:SC-1}, \eqref{eq:SC-2} and \eqref{eq:SC-3} of Lemma \ref{lem:SC},
we have, that
\[ \varSigma_{i,j} = \delta_{i,j} + O(n^{-1/2}).\]
In other terms,
\[ \lVert \varSigma - \Id \rVert = O(n^{-1/2}).\]
As the set of positive definite matrix is an open set of the space
of symmetric matrices, this implies that $\varSigma$ is positive definite
for $n$ big enough.

Besides, the application $A \mapsto \sqrt{A}$ is differentiable on this open set,
which implies the bound
$ \lVert \varSigma^{1/2} - \Id \rVert = O(n^{-1/2})$.
\end{proof}

\subsection{Error term}
\label{subsec:ErrorTerm}
In the previous Section, we have checked that the pair $(\tilde{W}_d, \tilde{W}_d^*)$ of the random vectors satisfies the assumptions of Theorem \ref{theo:MultivariateStein}.
In order to prove that the random vector $\tilde{W}_d$ is asymptotically Gaussian, we need to show that
quantities $A$ and $B$ from Theorem \ref{theo:MultivariateStein} vanish as $n \to \infty$. 
This section is devoted to making these calculations. 

\begin{lemma}
\label{lem:fourth_moment}
The following inequality holds:
\begin{multline}
\Var_n^{(\alpha)}\left((\esper_n^{(\alpha)})^{\tilde{W}_d}(W_i - W_i^*)(W_j - W_j^*)\right) \leq \frac{1}{ij n^{i+j+2}}\\
\times \left(\sum_{\mu_1, \mu_2, l \atop 
{\tiny |\mu_1|+\ell(\mu_1) \le i+j \atop |\mu_2|+\ell(\mu_2) \le i+j} }
H^{(i,j)}_{\mu_1, \mu_2; (1^l)} (n)_l - (i+j)^2\left(\sum_{l \ge 0} g_{(i), (j); l} (n)_l \right)^2 \right),
\end{multline}
where 
\[ H^{(i,j)}_{\mu_1, \mu_2; \mu} := (i+j-|\mu_1|+m_1(\mu_1))(i+j-|\mu_2|+m_1(\mu_2))g_{(i),(j); \mu_1}g_{(i),(j); \mu_2}g_{\mu_1,\mu_2; \mu}. \]
\end{lemma}
Note that the sums in the Lemma are clearly finite.

\begin{proof}
Following Fulman \cite[Proof of Proposition 6.4]{FulmanFluctuationChA2},
from Jensen's inequality for conditional expectations,
the fact that $\tilde{W}_d$ is determined by $\lambda$ implies that
\[ \esper_n^{(\alpha)}\left((\esper_n^{(\alpha)})^{\tilde{W}_d}(W_i - W_i^*)(W_j - W_j^*)\right)^2 \leq \esper_n^{(\alpha)}\left((\esper_n^{(\alpha)})^\lambda(W_i - W_i^*)(W_j - W_j^*)\right)^2. \]
Fix a partition $\lambda$ of $n$. We have, by Proposition~\ref{prop:FulmanComputations}, that
\begin{multline}
\label{eq:blablabla}
(\esper_n^{(\alpha)})^\lambda(W_i^* - W_i)(W_j^* - W_j) = \frac{\alpha(n-1)}{\alpha(n-1)+1}(\esper_n^{(\alpha)})^\lambda(W_i' - W_i)(W_j' - W_j)\\
 = \frac{\alpha(n-1)}{\alpha(n-1)+1}\left( (\esper_n^{(\alpha)})^\lambda(W_i'W_j') - (\esper_n^{(\alpha)})^\lambda(W_i')W_j - (\esper_n^{(\alpha)})^\lambda(W_j')W_i + W_iW_j \right)\\
 = \frac{\alpha(n-1)}{\alpha(n-1)+1}\left( (\esper_n^{(\alpha)})^\lambda(W_i'W_j') + \left( \frac{(i+j)(\alpha(n-1)+1)}{\alpha n(n-1)}-1\right) W_iW_j \right),
 \end{multline}
where the last equality follows from Proposition \ref{prop:SteinCondition1}.
But the product $W_i W_j$ expands as
\begin{equation}
    W_i W_j = \frac{1}{\sqrt{ij} n^{(i+j)/2}} \sum_{|\mu| \leq i+j+2} g_{(i),(j); \mu} \Ch_\mu.
    \label{eq:prod_W}
\end{equation}
Thus, strictly from the definition of $L^{(\alpha)}$, one has
\begin{multline}
    (\esper_n^{(\alpha)})^\lambda(W_i'W_j') = \sum_{\rho \vdash n} L(\lambda,\rho) W_i(\rho) W_j(\rho)
    \\= \sum_{\tau \vdash n}\theta_\tau(\lambda)\theta_\tau(n-1,1)
    \frac{(z_\tau)^2\alpha^{2\ell(\tau)}}{\alpha^n n!\sqrt{ij} n^{(i+j)/2}}
 \sum_{|\mu| \leq i+j+2} g_{(i),(j); \mu} \sum_{\rho \vdash n}
\frac{\Ch_\mu(\rho)\theta_\tau(\rho)}{j_\rho^{(\alpha)}}.
\end{multline}
We may assume $n \ge i+j$.
Recall that $\Ch_\mu(\rho)$ is a multiple of $\theta_{\mu 1^{n-|\mu|}}(\rho)$
and, hence, using  Lemma~\ref{lem:identities} \eqref{eq:orthogonality},
only terms corresponding to $\tau=\mu 1^{n-|\mu|}$ survives
and we get:
\begin{multline}
    (\esper_n^{(\alpha)})^\lambda(W_i'W_j') =
    \frac{1}{\sqrt{ij} n^{(i+j)/2}} \sum_{|\mu| \leq i+j+2} g_{(i),(j); \mu} \Ch_\mu(\lambda)\\
    \theta_{\mu\cup 1^{n-|\mu|}}(n-1,1)
    \frac{z_{\mu\cup 1^{n-|\mu|}}\alpha^{\ell({\mu\cup 1^{n-|\mu|}})}}{\alpha^n n!}.
\end{multline}
We now apply Lemma~\ref{lem:identities} \eqref{eq:(n-1,1)}:
\begin{multline}
(\esper_n^{(\alpha)})^\lambda(W_i'W_j') = \frac{1}{\sqrt{ij} n^{(i+j)/2}} 
\sum_{|\mu| \le i+j+2} g_{(i),(j); \mu} \Ch_\mu(\lambda) \\
\times \frac{(\alpha(n-1)+1)(n-|\mu|+m_1(\mu))-n}{\alpha n(n-1)}.
\end{multline}

One can substitute above equation and equation \eqref{eq:prod_W} to the equation \eqref{eq:blablabla} and
simplify it to obtain
\begin{multline}
    \label{eq:2nd_moment_wi*-wi}
(\esper_n^{(\alpha)})^\lambda(W_i^* - W_i)(W_j^* - W_j) \\
= \frac{1}{\sqrt{ij} n^{(i+j)/2}} 
\sum_{|\mu| \le i+j+2} \frac{i+j-|\mu|+m_1(\mu)}{n} g_{(i),(j); \mu} \Ch_\mu(\lambda).
\end{multline}
Notice that $g_{(i),(j); \mu}=0$ for $|\mu|+\ell(\mu) > i+j$, unless $\mu=(i,j)$
(Lemma \ref{lem:SC} \eqref{eq:SC-4}).
In the latter case ($\mu=(i,j)$), 
the numerical factor $i+j-|\mu|+m_1(\mu)$ vanishes.
It gives that the summation in equation~\eqref{eq:2nd_moment_wi*-wi} can be restricted
to partitions $\mu$ with $|\mu|+\ell(\mu) \le i+j$.

Taking the square, it gives us
\begin{multline}
\esper_n^{(\alpha)}\left((\esper_n^{(\alpha)})^\lambda(W_i^* - W_i)(W_j^* - W_j) \right)^2 \\
= \frac{1}{ij n^{i+j+2}}\esper_n^{(\alpha)}\left( \sum_{|\mu|+\ell(\mu) \le i+j} (i+j-|\mu|+m_1(\mu)) g_{(i),(j); \mu} \Ch_\mu(\lambda)\right)^2 \\
= \frac{1}{ij n^{i+j+2}}
\sum_{\mu_1, \mu_2, \mu \atop 
{\tiny |\mu_1|+\ell(\mu_1) \le i+j \atop |\mu_2|+\ell(\mu_2) \le i+j} }
H^{(i,j)}_{\mu_1, \mu_2; \mu}\esper_n^{(\alpha)}(\Ch_\mu(\lambda)) \\
= \frac{1}{ij n^{i+j+2}} \sum_{\mu_1, \mu_2, l\atop                   
{\tiny |\mu_1|+\ell(\mu_1) \le i+j \atop |\mu_2|+\ell(\mu_2) \le i+j} }
H^{(i,j)}_{\mu_1, \mu_2; (1^l)} (n)_l,
\end{multline}
where $H^{(i,j)}_{\mu_1, \mu_2; \mu}$ is defined as in the statement 
of the lemma. 
The last equality comes from the easy formula for the expectation of $\Ch_\mu$,
see equation~\eqref{eq:exp_Ch}.

Let us now analyse 
$\left(\esper_n^{(\alpha)}\left( (\esper_n^{(\alpha)})^{\tilde{W}_d}(W_i^* - W_i)(W_j^* - W_j)\right) \right)$.
After expanding the product, each term can be dealt with as follows:
\begin{align*}
    \esper_n^{(\alpha)}\left( (\esper_n^{(\alpha)})^{\tilde{W}_d}( W_i^* W_j^*) \right)&= \esper_n^{(\alpha)}(W_i^* W_j^*)=
    \esper_n^{(\alpha)}(W_i W_j) \\
    \esper_n^{(\alpha)}\left( (\esper_n^{(\alpha)})^{\tilde{W}_d}( W_i^* W_j) \right)&= \esper_n^{(\alpha)}\left(W_j (\esper_n^{(\alpha)})^{\tilde{W}_d}( W_i^*)\right)=\left( 1-\frac{i}{n} \right) \esper_n^{(\alpha)} (W_i W_j)\\
\esper_n^{(\alpha)}\left( (\esper_n^{(\alpha)})^{\tilde{W}_d}( W_i^* W_j) \right)&=
    \left( 1-\frac{j}{n} \right) \esper_n^{(\alpha)} (W_i W_j).\\
    \esper_n^{(\alpha)}\left( (\esper_n^{(\alpha)})^{\tilde{W}_d}( W_i W_j) \right)&=
    \esper_n^{(\alpha)}(W_i W_j) 
\end{align*}
The first equation comes from the fact that $\tilde{W}_d^*$ and $\tilde{W}_d$
have the same distribution, while the second one is a consequence of Proposition~\ref{prop:SteinCondition1}.
The third one is similar to the second one.
Therefore
\begin{multline}
\left(\esper_n^{(\alpha)}\left( (\esper_n^{(\alpha)})^{\tilde{W}_d}(W_i^* - W_i)(W_j^* - W_j)\right) \right)^2\\
= \left( \frac{i+j}{n} \esper_n^{(\alpha)}(W_iW_j) \right)^2 = \frac{(i+j)^2}{ij n^{i+j+2}}\left(\sum_\mu g_{(i), (j); \mu} \esper_n^{(\alpha)}(\Ch_\mu)\right)^2\\
=\frac{(i+j)^2}{ij n^{i+j+2}} \left(\sum_l g_{(i), (j);l} (n)_l\right)^2,
\end{multline}
where we used equations~\eqref{eq:prod_W} and \eqref{eq:exp_Ch}
in the second and third equalities.
We finish the proof by the following inequality:
\begin{multline}
\Var_n^{(\alpha)}\left((\esper_n^{(\alpha)})^{\tilde{W}_d}(W_i - W_i^*)(W_j - W_j^*)\right) \\
= \esper_n^{(\alpha)}\left((\esper_n^{(\alpha)})^{\tilde{W}_d}(W_i^* - W_i)(W_j^* - W_j) \right)^2 - \left(\esper_n^{(\alpha)}\left( (\esper_n^{(\alpha)})^{\tilde{W}_d}(W_i^* - W_i)(W_j^* - W_j)\right) \right)^2 \\
\leq \esper_n^{(\alpha)}\left((\esper_n^{(\alpha)})^\lambda(W_i^* - W_i)(W_j^* - W_j) \right)^2 - \left(\esper_n^{(\alpha)}\left( (\esper_n^{(\alpha)})^{\tilde{W}_d}(W_i^* - W_i)(W_j^* - W_j)\right) \right)^2.
\end{multline}
\end{proof}

\begin{proposition}
\label{prop:ErrorTermA}
Let $d \in \N$ and let 
\[ A = \sum_{2 \leq i,j \leq d+1} \frac{n}{i} \sqrt{\Var_n^{(\alpha)}\left((\esper_n^{(\alpha)})^{\tilde{W}_d}(W_i - W_i^*)(W_j - W_j^*)\right)}.\]
\end{proposition}
Then $A = O(n^{-1/2})$.

\begin{proof}
By Lemma \ref{lem:fourth_moment} we need to estimate the following sum:
\begin{equation}\label{eq:to_bound}
    \left(
\sum_{\mu_1, \mu_2, l \ge 1 \atop 
{\tiny |\mu_1|+\ell(\mu_1) \le i+j \atop |\mu_2|+\ell(\mu_2) \le i+j} }
H^{(i,j)}_{\mu_1, \mu_2; (1^l)}(n)_l - (i+j)^2\sum_{l,k} g_{(i), (j); (1^l)}g_{(i), (j); (1^k)} (n)_l(n)_k \right)^{1/2}.
\end{equation}
Let us first consider the second sum, which is simpler.
Suppose that $i \neq j$.
Then, by Lemma \ref{lem:SC} \eqref{eq:SC-2},
summands corresponding to $l \ge (i+j)/2$ or $k \ge (i+j)/2$ vanish
and the sum is $O(n^{i+j-1})$.

Consider now the case $i=j$ (recall that $i,j \ge 2$).
We use this time Lemma \ref{lem:SC} \eqref{eq:SC-1}:
summands corresponding to $l \ge i+1$ or $k \ge i+1$ vanish.
Therefore there is no summands of order $n^{i+j+1}$.
The unique summand of order $n^{i+j}$ (corresponding to $l=k=i$) is
$i^2 \, (n)_i^2$ by Lemma \ref{lem:SC} \eqref{eq:SC-3}.

Finally, we have that
\begin{equation}\label{eq:estimate_sum_g2}
    (i+j)^2\sum_{l,k} g_{(i), (j); (1^l)}g_{(i), (j); (1^k)} (n)_l(n)_k =
\delta_{i,j} 4\, i^4 n^{i+j} + O(n^{i+j-1}).
\end{equation}
Let us describe the terms corresponding to $l \ge i+j$ 
in the first sum of \eqref{eq:to_bound}.
By Lemma \ref{lem:SC} \eqref{eq:SC-1},
$g_{\mu_1,\mu_2,(1^l)}$, and hence $H^{(i,j)}_{\mu_1,\mu_2,(1^l)}$, vanishes
unless
\[ |\mu_1|+\ell(\mu_1) + |\mu_2|+\ell(\mu_2) \ge 2l \ge 2(i+j).\]
Comparing it with the conditions under the first summation symbol of \eqref{eq:to_bound}, we have, in fact, equalities instead of inequalities above, {\em i.~e.}
\[ l = i+j = |\mu_1|+\ell(\mu_1) = |\mu_2|+\ell(\mu_2).\]
Moreover, in this case, by Lemma \ref{lem:SC} \eqref{eq:SC-1},
we have that $g_{\mu_1, \mu_2; (1^{i+j})} = 0$ unless $\mu_1 \cup \mu_2 = 1^{i+j}$, which means that
$\mu_1 = \mu_2 = (1^{(i+j)/2})$
(in particular, $i+j$ must be even).
Then, by Lemma \ref{lem:SC} \eqref{eq:SC-2},
one has that $g_{(i),(j);(1^{(i+j)/2})} = 0$ unless $i=j$.
Therefore, $H^{(i,j)}_{\mu_1, \mu_2; (1^l)}=0$ if $i \neq j$ and $l \ge i+j$.

Additionally, using the definition of $H^{(i,j)}_{\mu_1,\mu_2,\mu}$,
the observations above and
Lemma \ref{lem:SC} \eqref{eq:SC-3},
one has that $H^{(i,i)}_{(1^i), (1^i); (1^{2i})} = 4 \, i^4$. Concluding,
\begin{equation}\label{eq:estimate_sum_H}
    \sum_{\mu_1, \mu_2, l}H^{(i,j)}_{\mu_1, \mu_2; (1^l)}(n)_l = \delta_{i,j}
    4\, i^4 n^{i+j} + O(n^{i+j-1}).
\end{equation}
Comparing equations~\eqref{eq:estimate_sum_g2} and \eqref{eq:estimate_sum_H},
it gives
\begin{multline}
\left(\sum_{\mu_1, \mu_2, l}H^{(i,j)}_{\mu_1, \mu_2; (1^l)}(n)_l - (i+j)^2\sum_{l,k} g_{(i), (j); (1^l)}g_{(i), (j); (1^k)} (n)_l(n)_k \right)^{1/2}\\
 = O(n^{(i+j-1)/2}),
 \label{eq:TechnicalBound}
\end{multline}
which implies that
\[ A = \sum_{2 \leq i,j \leq d+1} \frac{1}{i\sqrt{ij} n^{(i+j)/2}}O(n^{(i+j-1)/2}) = O(n^{-1/2}),\]
which finishes the proof.
\end{proof}

\begin{lemma}
\label{lem:momentsDegree}
For any $k \geq 2$, and $\lambda \vdash n$ there exists $B_{\alpha,k} \in \RR$, which depends only on $k$ and $\alpha$ such that
\[ |M_k(\lambda)| \leq B_{\alpha,k} \max(\lambda_1, \lambda_1')^k .\]
\end{lemma}

\begin{proof}
    As explained in Section \ref{SectDefKerov}, $M^{(\alpha)}_k(\lambda)=M^{(1)}_k(A_\alpha(\lambda))$
    is the $k$-th moment of the transition measure of the anisotropic digram $A_\alpha(\lambda)$.
    But this measure is supported by the contents of inner corners of $A_\alpha(\lambda)$.
    All these contents are clearly bounded in absolute value 
    by $\max(\sqrt{\alpha} \, \lambda_1, \sqrt{\alpha}^{-1} \lambda'_1)$.
    Hence the $k$-th moment of the measure is bounded by the $k$-th power of this number,
    which proves the lemma.
\end{proof}

\begin{proposition}
\label{prop:ErrorTermB}
Let $d \in \N$ and let
\[ \beta := \sum_{2 \leq i,j,k \leq d+1} \frac{n}{i} |(W_i^* - W_i)(W_j^* - W_j)(W_k^* - W_k)|.\] 
Then $B := \esper_n^{(\alpha)}(\beta) = O(n^{-1/2})$.
\end{proposition}

\begin{proof}
Fix some integer $k \ge 2$. 
From the definition of $M^{(\alpha)}$, 
we know that $\lambda^*$ is obtained from $\lambda$ by removing a box from the diagram of $\lambda$ and reattaching it somewhere. It means that $\lambda = \tau^{(i_1)}$ and $\lambda^* = \tau^{(i_2)}$ for some $\tau \vdash n-1$. It implies that
\[ \left|W_k^* - W_k\right| \leq \sqrt{k}^{-1} n^{-k/2} \left| \Ch_{(k)}(\tau^{(i_1)}) - \Ch_{(k)}(\tau^{(i_2)}) \right| . \]
By equation \eqref{EqIng1Lassalle}, the right hand side of the above inequality is equal to
\[  \sqrt{k}^{-1} n^{-k/2}\left|\sum_\rho a^{(k)}_\rho \left( \sum_{\substack{g,h \geq 0, \\ \pi \vdash h}}
    b_{g,\pi}^\rho(\gamma)\ M_\pi(\tau)\ \left(z_{i_1}^g - z_{i_2}^g\right) \right)\right|,\]
    where $|\pi| \leq |\rho| - g - 2$.
By Proposition \ref{PropBound1} we know, that $a^{(k)}_\rho = 0$ for $|\rho| > k+1$, hence $a^{(k)}_\rho b_{g,\pi}^\rho(\gamma) = 0$ for $|\pi| > k-g-1$.
But $z_{i_1}^g$ and $z_{i_2}^g$ are bounded by $O_{\alpha,g}(\max(\lambda_1, \lambda_1')^g)$,
as $\alpha$-contents of some box or corner of $\lambda$.
Thanks to Lemma \ref{lem:momentsDegree}, it implies
that there exists some $C_{\alpha,k} \in \RR$ which depends only on $k$ and $\alpha$ such that
\[ \left|W_k^* - W_k\right| \leq n^{-k/2} C_{\alpha,k} \max(\lambda_1, \lambda_1')^{k-1}.\]
If $\lambda_1 \leq 2 e\sqrt{\frac{n}{\alpha}}$ and $\lambda_1' \leq 2 e\sqrt{n\alpha}$, then, for any integers $i,j,k \ge 2$, one has
\begin{multline}\label{eq:BoundTripleProductGoodCase}
\left|(W_i^* - W_i)(W_j^* - W_j)(W_k^* - W_k)\right| \leq  |(W_i^* - W_i)| |(W_j^* - W_j)| |(W_k^* - W_k)|\\
 = O(n^{-3/2}) .
\end{multline}
Summing over $i$, $j$ and $k$, we get $\beta = O(n^{-1/2})$.
Otherwise, we use the obvious bounds $\lambda_1,\lambda'_1 \le n$, and we get
\begin{multline}
\left|(W_i^* - W_i)(W_j^* - W_j)(W_k^* - W_k)\right| \leq  |(W_i^* - W_i)| |(W_j^* - W_j)| |(W_k^* - W_k)|\\
 = O(n^{(i+j+k)/2-3})
\end{multline}
that is $\beta = O(n^{3(d+1)/2-2})$.
By Lemma \ref{lem:FiniteSupport}, the probability that the second case occurs
is exponentially small.
Hence the bound of the first case holds in expectation. It finishes the proof.
\end{proof}

\subsection{Proof of the central limit theorem}
We are now ready to prove the main result of this section:

\begin{proof}[Proof of Theorem \ref{theo:FluctuationsJackCharacters}]
    Let $\tilde{\Xi}_d = (\Xi_2,\dots,\Xi_{d+1})$. In order to show that
    \[ \left( \frac{\Ch_{(k)}}{ \sqrt{k} n^{k/2}} \right)_{k=2,3,\dots} \xrightarrow{d} \left( \Xi_k \right)_{k=2,3,\dots} \]
as $n \to \infty$, it is enough to show, that for all $d \in \N$ and for any smooth function $h$ on $\RR^d$, with all derivatives bounded, one has, as $n \to \infty$:
\[ \left| \esper_n^{(\alpha)}h(\tilde{W}_d) - \esper h(\tilde{\Xi}_d) \right| \to 0.\]

Fix a positive integer $d$ and a function $h:\RR^d \to \RR$ as above.
Let $\varSigma = (\esper_n^{(\alpha)})(\tilde{W}_d\tilde{W}_d^t)$.
As $h$ has its first derivative bounded, one has
\[\left| \esper \left( h(\tilde{\Xi}_d) - 
h(\varSigma^{1/2}\tilde{\Xi}_d) \right) \right| \le
|h|_1 \cdot d\, \lVert \Id - \varSigma^{1/2} \rVert \cdot \esper (\lVert \tilde{\Xi}_d \rVert).\]
But $|h|_1$ and $\esper_n^{(\alpha)}(\lVert \tilde{\Xi}_d \rVert)$ are fixed
finite numbers, while $\lVert \Id - \varSigma^{1/2} \rVert$ is $O(n^{-1/2})$
by Proposition \ref{prop:SteinCondition2}.
Hence,
\begin{equation}
    \left| \esper \left( h(\tilde{\Xi}_d) - 
    h(\varSigma^{1/2}\tilde{\Xi}_d) \right) \right| = O(n^{-1/2}).
    \label{eq:forget_Sigma}
\end{equation}

By Corollary \ref{cor:SteinCondition1} and Proposition \ref{prop:SteinCondition2} we know, that the pair $(\tilde{W}_d, \tilde{W}_d^*)$ satisfies all hypotheses of Theorem \ref{theo:MultivariateStein}.
Using this theorem, we get
\[ \left| \esper_n^{(\alpha)}h(\tilde{W}_d) - \esper h(\varSigma \tilde{\Xi}_d)\right| \leq |h|_2\frac{A}{4} + |h|_3\frac{B}{12},\]
where
\[ A = \sum_{2 \leq i,j \leq d+1} \frac{n}{i} \sqrt{\Var_n^{(\alpha)}\left((\esper_n^{(\alpha)})^{\tilde{W}_d}(W_i - W_i^*)(W_j - W_j^*)\right)}\]
and
\[ B = \sum_{2 \leq i,j,k \leq d+1} \frac{n}{i} \esper_n^{(\alpha)} |(W_i^* - W_i)(W_j^* - W_j)(W_k^* - W_k)|.\] 
Propositions \ref{prop:ErrorTermA} and \ref{prop:ErrorTermB} imply that
\[ \left| \esper_n^{(\alpha)}h(\tilde{W}_d) - \esper h(\varSigma^{1/2} \tilde{\Xi}_d)\right| = O(n^{-1/2}).\]
Together with equation \eqref{eq:forget_Sigma},
it finishes the proof.
\end{proof}

\subsection{Speed in convergence}

We now use \cite[Corollary 3.1]{ReinertRollin2009}, which gives an estimate for
\[ \left| \esper_n^{(\alpha)}h(\tilde{W}_d) - \esper h(\tilde{\Xi}_d) \right|\]
for non-smooth test functions $h$.
In particular, we shall consider functions $h$
in the set $\HHH$ of indicator functions of convex sets.
We have the following result,
which is stronger than Theorem~\ref{theo:SpeedConvergence}.

\begin{theorem}
    For any integer $d \ge 2$, we have
    \[\sup_{h \in \HHH} |\esper_n^{(\alpha)}\, h(\tilde{W}_d) - \esper\, h(\tilde{\Xi}_d)| 
        =O\big(n^{-1/4} \big).\]
    \label{theo:SpeedConvergenceConvex}
\end{theorem}
\begin{proof}
    Fix an integer $d \ge 2$ and consider the exchangeable pair
    $(\tilde{W}_d,\tilde{W}^\star_d)$ defined as in the previous sections.

    As shown in Section \ref{subsec:CheckingHypotheses},
    this exchangeable pair fulfills the condition of \cite[Theorem 2.1]{ReinertRollin2009}.
    Besides, as mentioned in \cite[Section 3]{ReinertRollin2009}, 
    the set $\HHH$ of functions fulfills conditions $(C1)$, $(C2)$ and $(C3)$ 
    (with $a=2\sqrt{d}$) from this paper.
    Therefore, one can apply \cite[Corollary 3.1]{ReinertRollin2009}
    and we find that there exists a constant $\zeta=\zeta(d)$ such that
    \begin{equation}\label{eq:ReinertNonSmooth}
        \sup_{h \in \HHH} |\esper_n^{(\alpha)}\, h(\tilde{W}_d) - \esper\, h(\tilde{\Xi}_d)|
    \le \zeta^ 2 \left(  \frac{A' \log(1/T')}{2} +\frac{B'}{2\sqrt{T'}} +2\sqrt{dT'} \right),
    \end{equation}
    where we define
    \begin{align*}
        \hat{\lambda}^{(i)} &:= \sum_{m=1}^d |(\varSigma^{-1/2} \Lambda^{-1} \varSigma^{1/2})_{m,i}|,\\
        A'&:= \sum_{i,j} \hat{\lambda}^{(i)} \sqrt{\sum_{k,l} \varSigma^{-1/2}_{i,k}  \varSigma^{-1/2}_{j,\ell}
        \Var (\esper_n^{(\alpha)})^{\tilde{W}_d} (W'_k -W_k)(W'_\ell-W_\ell)},\\
        B'&:= \sum_{i,j,k} \hat{\lambda}^{(i)} \esper_n^{(\alpha)} \left|
        \sum_{r,s,t} \varSigma^{-1/2}_{i,r} \varSigma^{-1/2}_{j,s} \varSigma^{-1/2}_{k,t}
        (W'_r -W_r)(W'_s -W_s)(W'_t -W_t)\right|,\\
        T'&:= \frac{1}{4d}\left( \frac{A'}{2} + \sqrt{B'\sqrt{d} + \frac{(A')^2}{4}} \right)^2.
    \end{align*}
    Comparing this with the statement in \cite{ReinertRollin2009}, note that in our case,
    there is no remaining matrix $R$, and hence $C'=0$.
    Besides, for the set $\HHH$ considered here (indicator functions of convex subsets of $\RR^d$),
    one can choose $a=2\sqrt{d}$.

    We shall now describe the asymptotic behaviour of the quantities above.
    Note that all sums appearing above have fixed number of summands since
    the summation index set is always the set of integers less or equal to $d$.

    Recall that, in our setting, $\Lambda^{-1}$ is the diagonal matrix $(n/(i+1) \cdot  \delta_{i,j})$.
    Besides $\varSigma^{-1/2}$ (well-defined for $n$ big enough) is a bounded matrix 
    (Proposition~\ref{prop:SteinCondition2}).
    Hence for any $i \le d$,
    \[\hat{\lambda}^{(i)} = O(n).\]
    Consider now $A'$. It was proven in Section \ref{subsec:ErrorTerm}
    (Lemma~\ref{lem:fourth_moment} and equation~\eqref{eq:TechnicalBound})
    that, for any $k$ and $\ell$,
    \[|\Var (\esper_n^{(\alpha)})^{\tilde{W}_d} (W'_k -W_k)(W'_\ell-W_\ell) | = O(n^{-3}).\]
    Together with the bound on $\hat{\lambda}^{(i)}$ and the fact that
    $\varSigma^{-1/2}$ is bounded, this implies
    \[A'=O(n^{-1/2}).\]
    Consider now $B'$. We have proved that the bound \eqref{eq:BoundTripleProductGoodCase}
    holds in expectation, that is
    \[\esper_n^{(\alpha)} \left|(W'_r -W_r)(W'_s -W_s)(W'_t -W_t)\right| \leq O(n^{-3/2}). \]
    As before, the bound above on $\hat{\lambda}^{(i)}$ and the fact that
    $\varSigma^{-1/2}$ is bounded implies
    \[B'=O(n^{-1/2}).\]
    Combining the bounds for $A'$ and $B'$, we get $T' = O(n^{-1/2})$.
    Strictly from the definition of $T'$, we also get
    $T' \ge B'$, which implies that $\frac{B'}{2\sqrt{T'}} \le \sqrt{B'}=O(n^{-1/4})$.
    Plugging all these estimates in equation~\eqref{eq:ReinertNonSmooth},
    we get the desired result.
\end{proof}

\subsection{Fluctuations of other polynomial functions}\label{subsect:Fluct_NonHook}

As $\Ch_{(k)}$ is an algebraic basis of polynomial functions,
Theorem~\ref{theo:FluctuationsJackCharacters} implies that any polynomial function $F$
converge, after proper normalization, towards a multivariate polynomial evaluated in independent
Gaussian variables.

However, the proper order of normalization and the actual polynomial are not easy
to be expressed explicitly, as this relies on the expansion of $F$ in the $\Ch_{(k)}$ basis.
In particular, we are not able to do it for $\Ch_{\mu}$, when $\mu$ is not a hook and hence
we can not describe the fluctuations of these random variables
as it was done in the case $\alpha=1$;
see \cite{HoraFluctuationsCharacters} and \cite[Theorem 6.5]{IvanovOlshanski2002}.

\section{Jack measure: central limit theorems for Young diagrams and transition measures}
\label{sect:CLT2}

In this section, we state formally and prove our fluctuation results
for Young diagrams under Jack measure.
We will also present a fluctuation result for the transition measures of these diagrams.

Before we state our results, we need some preparations.
We follow here notations from \cite[Sections 7 and 8]{IvanovOlshanski2002}.
First, we define
\begin{align*}
u_k(x) &= U_k(x/2) = \sum_{0 \leq j \leq \lfloor k/2 \rfloor}(-1)^j\binom{k-j}{j}x^{k-2j},\\
 t_k(x) &= 2T_k(x/2) = \sum_{0 \leq j \leq \lfloor k/2 \rfloor}(-1)^j\frac{k}{k-j}\binom{k-j}{j}x^{k-2j},
\end{align*}
where $T_k$ and $U_k$ are respectively Chebyshev polynomials of the first and second kind.
They can be alternatively defined by the following equations:
\begin{align*}
 u_k(2 \cos(\theta)) &= \frac{\sin((k+1)\theta)}{\sin(\theta)};\\
 t_k(2 \cos(\theta)) &= 2 \cos(k\theta).
\end{align*}
It is known that
$\left(u_k(x)\right)_k$ and $\left(t_k(x)\right)_k$ form a family of orthonormal polynomials with respect to the measures $\frac{\sqrt{4-x^2}}{2\pi}dx$ and $\frac{1}{2\pi \sqrt{4-x^2}}dx$, respectively, {\it i.~e.}:
\begin{align*}
 \int^2_{-2} u_k(x)u_l(x)\frac{\sqrt{4-x^2}}{2\pi }dx &= \delta_{k,l};\\
 \int^2_{-2} t_k(x)t_l(x)\frac{1}{2\pi \sqrt{4-x^2}}dx &= \delta_{k,l}.
\end{align*}
The measure $\frac{\sqrt{4-x^2}}{2\pi}dx$ supported on the interval $[-2,2]$ is called the \emph{semi-circular distribution} and is denoted by the $\mu_{S-C}$ (see Subsection \ref{subsect:JackPlancherel}).

Recall from Theorem \ref{theo:RYD-limit} that the limit shape of scaled Young diagrams
$$\omega\bigg(D_{1/\sqrt{n}}\big(A_\alpha(\lambda_{(n)})\big)\bigg)$$ is given by $\Omega$,
where $\lambda_{(n)}$ is a random Young diagram with $n$ boxes distributed according to Jack measure.
Hence, in order to study fluctuations of the random Young diagrams around the limit shape,
we introduce the following application from the set of Young diagrams to the space of functions from $\RR$ to $\RR$:
\[ \Delta^{(\alpha)}(\lambda)(x) := \sqrt{n}\frac{\omega\bigg(D_{1/\sqrt{n}}\big(A_\alpha(\lambda)\big)\bigg)(x)
 - \Omega(x)}{2},
\]
where $n$ is the number of boxes of $\lambda$.

Besides, tt was shown by Kerov \cite{Kerov1993} that the transition measure (see Subsection \ref{subsect:TransMeasure}) of the continuous Young diagram $\Omega$ is the semi-circular distribution (see Subsection \ref{subsect:JackPlancherel}).
Thus we define the following function on the set of Young diagrams with $n$ boxes with values in the space of real signed measures:
\[ \widehat{\Delta}^{(\alpha)}(\lambda) := \sqrt{n}\left( \mu_{\bigg(D_{1/\sqrt{n}}\big(A_\alpha(\lambda)\big)\bigg)} - \mu_{S-C} \right) .\]
As above, $n$ is the number of boxes of $\lambda$.
This function describes the (scaled) difference between the transition measure of the scaled Young diagram and the limiting semi-circular measure.

Now, we are ready to formulate the central limit theorem for the Jack measure.
Here, we use the usual notation $[\text{condition}]$ for the indicator function of the corresponding condition.
\begin{theorem}
\label{theo:Fluctuations}
Choose a sequence $\left(\Xi_k \right)_{k=2,3,\dots}$ of independent standard Gaussian random variables
and let $\lambda_{(n)}$ be a random Young diagram of size $n$ distributed according to Jack measure.
As $n \to \infty$, we have
\begin{enumerate}
\item
\label{item:fluctuationsOfDiagrams}
a central limit theorem for Young diagrams: 
\[ \left( u^{(\alpha)}_{k}(\lambda_{(n)}) \right)_{k=1,2,\dots} \xrightarrow{d} 
\left(  \frac{\Xi_{k+1}}{\sqrt{k+1}} - \frac{\gamma}{k+1} \, [k\text{ is odd}]\right)_{k=1,2,\dots}, \]
where $u^{(\alpha)}_{k} (\lambda) = \int_\RR u_k(x) \Delta^{(\alpha)}(\lambda)(x)\ dx$;\medskip
\item
\label{item:fluctuationsOfTransition}
and a central limit theorem for transition measures:
\[ \left( t^{(\alpha)}_{k}(\lambda_{(n)}) \right)_{k=3,4,\dots} \xrightarrow{d} \left(  \sqrt{k-1}\, \Xi_{k-1} - \gamma \, [k\text{ is odd}]  \right)_{k=3,4,\dots}, \]
where $t^{(\alpha)}_{k} (\lambda) = \int_\RR t_k(x) \widehat{\Delta}^{(\alpha)}(\lambda)(dx)$.
\end{enumerate}
\end{theorem}

\begin{remark}
    Notice that, for $\gamma=0$ ({\it i.e.} $\alpha=1$),
    this theorem specializes to Kerov's central limit theorems for Plancherel measure \cite{Kerov1993, IvanovOlshanski2002}.
\end{remark}

\subsection{Extended algebra of polynomial functions and gradations }

The proof combines our fluctuation results for Jack characters
and arguments from the proof of Kerov, Ivanov and Olshanski for the case $\alpha=1$.
In particular we shall compare some $\alpha$-polynomial functions with their counterpart for $\alpha=1$.
Therefore, throughout this section, we will make the dependence in $(\alpha)$ explicit and 
we use the notations $\Ch_\mu^{(\alpha)}$, $M_k^{(\alpha)}$, and so on.
The only exception to this is $\Ch_{(1)}$ as, for any $\alpha$, the function $\Ch_{(1)}$ associates
to a Young diagram its number of boxes.

To prove Theorem \ref{theo:Fluctuations},
it is convenient to extend the algebra $\Pola$ in the same way as Ivanov and Olshanski \cite{IvanovOlshanski2002}: we adjoin to it the square root of
the element $\Ch_{(1)}$ and then localize over the multiplicative family generated by $\sqrt{\Ch_{(1)}}$. Let $\Polaext$ denotes the resulting algebra.

We also define, for a partition $\mu$ of length $\ell$,
\[\widetilde{\Ch}^{(\alpha)}_\mu := \prod_{i=1}^\ell \Ch^{(\alpha)}_{(\mu_i)}.\]
Then $\widetilde{\Ch}^{(\alpha)}_\mu$ is a multiplicative basis of $\Pola$,
while a multiplicative basis in $\Polaext$ is given by
\[ \widetilde{\Ch}^{(\alpha)}_\mu \left(\Ch_{(1)}\right)^{m/2},\text{ with } m_1(\mu)=0, \quad m \in \Z.\]
We equip $\Polaext$ with a gradation defined by
\[ \deg_4\bigg( \widetilde{\Ch}^{(\alpha)}_\mu \left(\Ch_{(1)}\right)^{m/2}\bigg) = |\mu| + m.\]
Note that some elements have negative degree.
Besides, for a general partition $\mu$, one has:
\[ \deg_4\big( \widetilde{\Ch}^{(\alpha)}_\mu \big) = |\mu| + m_1(\mu).\]

In \cite{IvanovOlshanski2002}, for $\alpha=1$, the authors consider a slightly different filtration
on $\Poluext$, namely they define $\deg_{\{1\}}$ as follows\footnote{In \cite{IvanovOlshanski2002}, $\deg_{\{1\}}$ is
abbreviated as $\deg_{1}$ but we shall not do that to avoid a conflict of notation with Section~\ref{SubsectBound1}.}:
for $\mu$ without part equal to $1$ and $m \in \Z$,
\begin{equation}
    \label{eq:def_grad_IO}
    \deg_{\{1\}}\bigg(\Ch^{(1)}_\mu \left(\Ch_{(1)}\right)^{m/2}\bigg) = |\mu|+m.
\end{equation}
Note that, for a general partition $\mu$, one has:
\[ \deg_{\{1\}}\big( \Ch^{(1)}_\mu \big) = \mu + m_1(\mu).\]
Let us compare $\deg_4$ and $\deg_{\{1\}}$.
For any integer $d$ (positive or not),
let $V^{\le d}$, ($V^{\le d}_{\{1\}}$, respectively) denotes the subspace of $\Poluext$ containing elements $x$
with $\deg_4(x)\le d$ ($\deg_{\{1\}}(x)\le d$, respectively).
\begin{lemma}
    For any integer $d$, one has
    \[V^{\le d}=V^{\le d}_{\{1\}}.\]
    \label{lem:ComparisonDegree}
\end{lemma}
\begin{proof}
    Let us first show that
    \begin{equation}
        \label{eq:SameDegree_in_Polun}
        V^{\le d} \cap \Polun =V^{\le d}_{\{1\}} \cap \Polun.
    \end{equation}
    By definition, the left-hand side has basis $(\widetilde{\Ch}^{(1)}_\mu)_{|\mu|+m_1(\mu) \le d}$,
    while the right-hand side has basis $(\Ch^{(1)}_\mu)_{|\mu|+m_1(\mu) \le d}$.
    But, if $|\mu|+m_1(\mu) \le d$,
    \[\deg_{\{1\}} \big(\widetilde{\Ch}^{(1)}_\mu \big)
    \leq \sum_{i=1}^\ell \deg_{\{1\}} \big(\widetilde{\Ch}^{(1)}_{(\mu_i)} \big) = |\mu|+m_1(\mu) \le d,\]
    which shows an inclusion between the two spaces.
    As they have the same dimension, \eqref{eq:SameDegree_in_Polun} holds.

    Observe now that, for both gradations,
    an element $F \in \Poluext$ has degree at most $d$
    if and only if it can be written as
    \[F = \Ch_{(1)}^m \, F_1 + \Ch_{(1)}^{m+1/2} \, F_2 \]
    for some integer $m$ and elements $F_1$, $F_2$ from $\Polun$
    of degree at most $d_1$ and $d_2$ with $2m+d_1 \le d$ and
    $2m+1+d_2 \le d$.
    Hence, the lemma follows from \eqref{eq:SameDegree_in_Polun}.
\end{proof}

\begin{remark}
    We are not able to show that
    \[\deg_{\{1\}}\bigg(\Ch^{(\alpha)}_\mu \left(\Ch_{(1)}\right)^{m/2}\bigg) = |\mu|+m\]
    defines a filtration of $\Polaext$, which would make a natural extension of \eqref{eq:def_grad_IO}. However, thanks to Lemma \ref{lem:ComparisonDegree}, we can use the multiplicative family $\widetilde{\Ch}^{(\alpha)}_\mu$ instead.
\end{remark}

\subsection{Proof of Theorem \ref{theo:Fluctuations}}

The main part of the proof of Kerov, Ivanov and Olshanski
is to prove that
$u_k^{(1)}$ and $t_k^{(1)}$ are in $\Poluext$ and fulfill
\begin{align}
    u_{k}^{(1)} &= \frac{\Ch^{(1)}_{(k+1)}}{(k+1) \Ch_{(1)}^{(k+1)/2}} +\text{ terms of negative degree for $\deg_4$;}
    \label{eq:FromIO_u}\\
    t_{k}^{(1)} &= \frac{\Ch^{(1)}_{(k-1)}}{\Ch_{(1)}^{(k-1)/2}} +\text{ terms of negative degree for $\deg_4$.}
    \label{eq:FromIO_t}
\end{align}

These equations are respectively the last equations of Sections 7 and 8 in paper \cite{IvanovOlshanski2002}.
As the notations are a little bit different here, let us give a few precisions.
\begin{itemize}
    \item The quantity $\eta_{k+1}$ in \cite{IvanovOlshanski2002} is
        defined by equation (6.5) and definition 3.1.
    \item In \cite{IvanovOlshanski2002}, it is shown that the reminder
        has negative degree in the filtration $\deg_{\{1\}}$, while here we use the gradation $\deg_4$.
        But we have proven in Lemma \ref{lem:ComparisonDegree} that both notions coincide on $\Polun$.
    \item The identities in \cite{IvanovOlshanski2002} are equalities of random
        variables, that is of functions on the set of Young diagrams
        of size $n$ (which is here the probability space) ;
        as they are valid for any $n$, we have in fact
        identities of functions on the set of Young diagrams, as claimed above.
\end{itemize}
Note that, as equalities of functions on the set of Young diagrams,
one can evaluate them on continuous diagrams, in particular on $A_\alpha(\lambda)$.

Our goal is to establish similar formulas in the general $\alpha$-case.
From the definition, it is straight-forward that
\begin{align}                                   
    u_{k}^{(\alpha)}(\lambda) &= u_{k}^{(1)}\big(A_\alpha(\lambda) \big); \\
    t_{k}^{(\alpha)}(\lambda) &= t_{k}^{(1)}\big(A_\alpha(\lambda) \big).
\end{align}
Therefore, applying equations~\eqref{eq:FromIO_u} and \eqref{eq:FromIO_t} on $A_\alpha(\lambda)$,
we get:
\begin{align}
    u_{k}^{(\alpha)}(\lambda) &= \frac{\Ch^{(1)}_{(k+1)}\big(A_\alpha(\lambda) \big)}{(k+1) \left(\Ch_{(1)}\right)^{(k+1)/2}} 
    +\text{ terms of negative degree for $\deg_4$;}
    \label{eq:FromIO_u_la^al}\\
    t_{k}^{(\alpha)}(\lambda) &= \frac{\Ch^{(1)}_{(k-1)}\big(A_\alpha(\lambda) \big)}{\left(\Ch_{(1)}\right)^{(k-1)/2}}
    +\text{ terms of negative degree for $\deg_4$.}
    \label{eq:FromIO_t_la^al}
\end{align}
Of course in general, although both quantities lie in $\Pola$,
\[\Ch^{(1)}_{(k)}\big(A_\alpha(\lambda) \big) \neq \Ch^{(\alpha)}_{(k)}(\lambda).\]
The following lemma compares the highest degree terms of quantities above:
\begin{lemma}
\label{lem:1toAlpha}
For any integer $m \geq 1$ we have that
\begin{multline} 
\Ch^{(1)}_{(k)}\big(A_\alpha(\lambda)\big) 
= \Ch^{(\alpha)}_{(k)}(\lambda) - \gamma\left(\Ch_{(1)}\right)^{(k/2)} \, [k\text{ is even}] \\
+ \text{ terms of degree less than }
    k\text{ with respect to }\deg_4.
\end{multline}
\end{lemma}

\begin{proof}
    We know, by Corollary~\ref{corol:dominant_deg1}, that for any $k \geq 2$,
    \[ \deg_1\left(\Ch^{(1)}_{(k)}\big(A_\alpha(\lambda)\big) - \Ch^{(\alpha)}_{(k)}(\lambda)\right) = k.\]
Let us consider its $\widetilde{\Ch}_\mu$ expansion:
\[ \Ch^{(1)}_{(k)}\big(A_\alpha(\lambda)\big) - \Ch^{(\alpha)}_{(k)}(\lambda) = \sum_\mu a^k_\mu \widetilde{\Ch}_\mu.\]
Since $\deg_4(\widetilde{\Ch}_\mu) \leq \deg_1(\widetilde{\Ch}_\mu)$ with equality only for $\mu = (1^m)$
for some non-negative integer $m$, one has that
\begin{align*}
    \Ch^{(1)}_{(2m+1)}\big(A_\alpha(\lambda)\big) &= \Ch^{(\alpha)}_{(2m+1)}(\lambda) 
    \\&\qquad+ \text{ terms of degree less than }
    2m+1\text{ with respect to }\deg_4;   \\
        \Ch^{(1)}_{(2m)}\big(A_\alpha(\lambda)\big)& 
        = \Ch^{(\alpha)}_{(2m)}(\lambda) - a^{2m}_{(1^m)}\left(\Ch_{(1)}\right)^m
        \\&\qquad + \text{ terms of degree less than }
    2m\text{ with respect to }\deg_4. \
\end{align*}
But 
\[ a^{2m}_{(1^m)} = [\left(R^{(\alpha)}_2(\lambda)\right)^m] \left(\Ch^{(\alpha)}_{(2m)}(\lambda)
- \Ch^{(1)}_{(2m)}\big(A_\alpha(\lambda)\big)\right) = [R_2^m]\Ch^{(\alpha)}_{(2m)}(\lambda) = \gamma,\]
by \cite[Theorem 10.3]{Lassalle2009} or Proposition \ref{prop:TopDegrees} here
(the coefficient of $R_2^m$ in $\Ch^{(1)}_{(2m)}$ is zero for parity reason).
This finishes our proof.
\end{proof}

\begin{corollary}
    \label{corol:uktk_chka}
One has the following equalities in $\Polaext$:
\begin{align}
    u_{k}^{(\alpha)} &= 
    \frac{\Ch^{(\alpha)}_{(k+1)}}{(k+1) \left(\Ch_{(1)}\right)^{(k+1)/2}}  - \frac{\gamma}{k+1} [k\text{ is odd}]
    +\text{ terms of negative degree;}
    \label{eq:uka_chka}\\
    t_{k}^{(\alpha)} &= \frac{\Ch^{(\alpha)}_{(k-1)} }{\left(\Ch_{(1)}\right)^{(k-1)/2}} - \gamma \, [k\text{ is odd}]
    +\text{ terms of negative degree.}
    \label{eq:tka_chka}
\end{align}
\end{corollary}
\begin{proof}
    For any Young diagram $\lambda$,
    equation \eqref{eq:uka_chka} evaluated at $\lambda$
    is obtained from equation~\eqref{eq:FromIO_u_la^al} and Lemma~\ref{lem:1toAlpha}.
    Similarly, equation \eqref{eq:tka_chka} is a consequence of equation~\eqref{eq:FromIO_t_la^al}
     and Lemma~\ref{lem:1toAlpha}.
\end{proof}

Now, we will prove that 
elements of negative degree are asymptotically negligible.

\begin{lemma}
\label{lem:NegDeg=>Zero}
Let $f \in \Polaext$ be a function of degree less than $0$. Then, as $n \to \infty$,
\[ f(\lambda_{(n)}) \xrightarrow{d} 0,\]
where the distribution of $\lambda_{(n)}$ is Jack measure
of size $n$.
\end{lemma}

\begin{proof}
It is enough to show that, as $n \to \infty$
\begin{equation}\label{eq:NegDeg=>Zero__OnBasis}
 \widetilde{\Ch}_\mu(\lambda_{(n)}) \left(\Ch_1(\lambda_{(n)})\right)^{-m/2} \xrightarrow{d} 0
\end{equation}
for $|\mu| < m$, where the distribution of $\lambda_{(n)}$ is Jack measure
of size $n$. But this is a consequence of Theorem~\ref{theo:FluctuationsJackCharacters}. Indeed, let $\left(\Xi_k \right)_{k=2,3,\dots}$ be a family 
of independent standard Gaussian random variables. Then Theorem~\ref{theo:FluctuationsJackCharacters} states that
\[ \widetilde{\Ch}_\mu(\lambda_{(n)}) \, n^{-|\mu|/2} \xrightarrow{d} \prod_{i}\sqrt{\mu_i} \, \Xi_{\mu_i}.\]
As $\Ch_1(\lambda_{(n)}) \equiv n$, this implies \eqref{eq:NegDeg=>Zero__OnBasis} and finishes the proof.
\end{proof}

Finally, Theorem \ref{theo:Fluctuations} follows from
Corollary~\ref{corol:uktk_chka}, Theorem~\ref{theo:FluctuationsJackCharacters} and Lemma~\ref{lem:NegDeg=>Zero}.

\subsection{Informal reformulation of Theorem \ref{theo:Fluctuations}}
\label{subsect:informal}
Choose, as above, a sequence of independent standard Gaussian random variables $\left(\Xi_k \right)_{k=2,3,\dots}$
and consider the random series
\begin{multline*}
\Delta_\infty^{(\alpha)} (2 \cos(\theta) ):=
\frac{1}{\pi} \sum_{k=2}^\infty \left(\frac{\Xi_k}{\sqrt{k}}-\frac{\gamma}{k} [k\text{ is even}]\right) \sin(k \theta) \\
= \frac{1}{\pi} \sum_{k=2}^\infty \frac{\Xi_k}{\sqrt{k}} \sin(k \theta) - \gamma/4 +\gamma \theta/2\pi.
\end{multline*}
This series is nowhere convergent (almost surely),
but it makes sense as a generalized Gaussian process with values in
the space of generalized functions $(\mathcal{C}^\infty(\RR))'$,
that is the dual of the space of infinitely differentiable functions;
see \cite[Section 9]{IvanovOlshanski2002} for details.

For polynomials $u_\ell(x)$, one has
\begin{multline*}
    \langle u_\ell,\Delta_\infty^{(\alpha)} \rangle = \int_{-2}^2 u_\ell(x) \Delta_\infty^{(\alpha)}(x) dx
= 2 \int_0^\pi u_\ell(2 \cos(\theta)  \Delta_\infty^{(\alpha)}(2 \cos(\theta)) \sin(\theta) d\theta\\
=\frac{2}{\pi} \sum_{k=2}^\infty \left(\frac{\Xi_k}{\sqrt{k}}-\frac{\gamma}{k} [k\text{ is even}]\right)
\int_0^\pi \sin\big( (\ell+1)\theta \big) \sin(k \theta) d\theta.
\end{multline*}
Only the integral for $k=\ell+1$ is non-zero (it is equal to $2/\pi$). Thus we get
\[\langle u_\ell,\Delta_\infty^{(\alpha)} \rangle =
\frac{\Xi_{\ell+1}}{\sqrt{\ell+1}}-\frac{\gamma}{\ell+1} [\ell\text{ is odd}],\]
which, from Theorem~\ref{theo:Fluctuations}, is exactly the limit in distribution of 
\[ \langle u_\ell,\Delta^{(\alpha)}(\lambda_{(n)}) \rangle.\]
As this limit in distribution holds jointly for different values of $\ell$,
by linearity, one can replace $u_\ell$ by any polynomial $P$.
Hence, $\Delta_\infty^{(\alpha)}$ can be informally seen as the limit
of the random functions $\Delta^{(\alpha)}(\lambda_{(n)})$,
which justifies equation~\eqref{eq:informal}.

\appendix

\section{Kerov polynomials}
\label{app:Kerov_polynomials}

In this Section, we answer some questions of Lassalle 
concerning Kerov polynomials \cite{Lassalle2009}.
The results are consequences of the methods or results from Section \ref{SectPolynomial}
and thus fit in the scope of this paper.
However, as Kerov polynomials are 
used in this paper only as a tool, 
we decided to present them in Appendix.

\subsection{Comparison with Lassalle's normalizations}
\label{SubsectPolAvecLassalleConvention}
Recall that our normalization is different than the one used by Lassalle.
As in Section \ref{SubsectAlgo}, we use boldface font for quantities defined
in Lassalle's paper \cite{Lassalle2009}.
Our second bound on the degree of coefficients of $K_\mu$,
implies the following result.
\begin{proposition}
    The coefficient $\bm{c^\mu_\rho}$ of $\bm{R_\rho}$ in $\bm{K_\mu}$
    with Lassalle's normalization is a polynomial in $\alpha$
    divisible by
    $\alpha^{|\rho|-\ell(\rho)}$.
\end{proposition}
\begin{proof}
Let us start by a comparison of Lassalle's conventions with ours.
If $\mu$ does not contain a part equal to $1$ then
\begin{align*}
    \bm{\vartheta_\mu^\lambda(\alpha)} &=
    z_\mu \theta^{(\alpha)}_{\mu,1^{|\lambda|-|\mu|}}(\lambda),\\
    \intertext{so that } \Ch_\mu(\lambda)&=\alpha^{-\frac{|\mu|-\ell(\mu)}{2}}
     \bm{\vartheta_\mu^\lambda(\alpha)}.\\
    \intertext{Besides, } \bm{R_k(\lambda)} &= \alpha^{-k/2} R_k(\lambda) \\
    \intertext{and } \bm{\vartheta_\mu^\lambda(\alpha)} &= \bm{K_\mu} (\bm{R_2},
    \bm{R_3},\cdots).
\end{align*}
Finally, the coefficient $\bm{c^\mu_\rho}$ of $\bm{R_\rho}$ in $\bm{K_\mu}$
with Lassalle's normalization is related to the coefficient $d^\mu_\rho$ of $R_\rho$
in $K_\mu$ with our conventions by:
\[ \bm{c^\mu_\rho} = \alpha^{\frac{|\mu|-\ell(\mu)}{2}+\frac{|\rho|}{2}} d^\mu_\rho.\]
But we have shown that $d^\mu_\rho$ is a polynomial in $\gamma=\frac{1-\alpha}{\sqrt{\alpha}}$
of degree less than $|\mu|-\ell(\mu)-(|\rho|-2\ell(\rho))$.
Thus $\bm{c^\mu_\rho}$ is a polynomial in $\sqrt{\alpha}$ divisible by
$\alpha^{|\rho|-\ell(\rho)}$.
Our parity result for  $d^\mu_\rho$ (second part of Proposition~\ref{PropBound1})
implies that $\bm{c^\mu_\rho}$ is in fact a polynomial
in $\alpha$.
\end{proof}

Lassalle had only proved in his article that these quantities were rational functions
in $\alpha$.
He conjectured that they are polynomials with integer coefficients 
\cite[Conjecture 1.1]{Lassalle2009}.
Our result is weaker than this conjecture as we are not able to prove the integrity
of the coefficients.
However, we also proved that the polynomials are divisible by $\alpha^{|\rho|-\ell(\rho)}$,
which fits with
Lassalle's data \cite[Section 1]{Lassalle2009}, but was not mentioned by him.

\subsection{Linear terms in Kerov polynomials}
\label{SectLinearTerms}

In this short section, we compute the top degree part of the coefficients of linear
terms in Kerov polynomials. 
This proves a conjecture of Lassalle
\cite[page 31]{Lassalle2009}.

\begin{proposition}
For any integers $k > 0$ and $k-1 \geq i \geq 0$, we have
$$[R_{k+1-i}]K_{(k)} = \left[ \begin{matrix} k\\ k-i\\ \end{matrix} \right] \gamma^i
+ \text{lower degree terms},$$
where $\left[ \begin{matrix} k\\ k-i\\ \end{matrix} \right]$ denotes the
positive Stirling number of the first kind.
\end{proposition}

\begin{proof}
We recall that
\[ \Ch_\mu(\lambda) = \sum_\rho a_\rho^\mu M_\rho(\lambda).\]
Thanks to the relation between moments and free cumulants -- see Section \ref{subsect:TransMeasure} -- it is enough to prove that 
$$a^{(k)}_{(k+1-i)} = \left[ \begin{matrix} k\\ k-i\\
\end{matrix} \right] \gamma^i
+ \text{lower degree terms}$$
for any positive integers $k > 0$
and $k-1 \geq i \geq 0$. We will prove it by induction over $k$. For $k=1$ we have
that $K_{(1)} = M_2 = R_2$ and the inductive assertion holds in this case. 

Putting $\mu = (k)$ in Equation \eqref{EqRec2} we have that
$$\sum_\rho a^{(k)}_\rho \left( \sum_{\substack{g,h \geq 0, \\ \pi \vdash h}}
b_{g,\pi}^\rho(\gamma) M_{\pi \cup (g+1)}  \right) = k L_{k-1},$$
hence
$$\sum_\rho a^{(k)}_\rho b_{k-1-i,0}^\rho(\gamma) = k a^{(k-1)}_{(k-i)}$$
for any integer $0 \leq i \leq k-1$.
From Lemma
\ref{LemDeg1B}, $b_{k-1-i,0}^\rho$ vanishes for $|\rho| < k+1-i$.
Moreover, by Proposition \ref{PropBound1} and by Lemma
\ref{LemDeg2B}, we have that
$$\deg_\gamma(a^{(k)}_\rho b_{k-1-i,0}^\rho(\gamma)) \leq k + 1 - |\rho| + |\rho| -
2\ell(\rho) - (k-3-i)= i - 2(\ell(\rho)-1),$$
By inductive hypothesis one has
\begin{equation}
\label{eq:ForLinearTerms}
\sum_{k+1 \geq r \geq k+1-i} a^{(k)}_{(r)} b_{k-1-i,0}^{(r)}(\gamma) = k \left[
\begin{matrix} k-1\\ k-1-i\\
\end{matrix} \right] \gamma^i
+ \text{lower degree terms}
\end{equation}
for any integer $0 \leq i \leq k-1$.
From Proposition \ref{PropMkLo} we know that
$$b_{k-1-i,0}^{(r)}(\gamma) = \binom{r-1}{r-(k-i)}(-\gamma)^r-(k-i+1).$$
Putting it into Equation \eqref{eq:ForLinearTerms} we obtain that in order to
finish the proof it is enough to prove the following identity
(set $r=k+1-i+j$ in the summation index)
$$\sum_{0 \leq j \leq i}\binom{k-i+j}{j+1}(-1)^j \left[ \begin{matrix} k\\
k-i+j\\
\end{matrix} \right] = k\left[ \begin{matrix} k-1\\ k-1-i\\
\end{matrix} \right]$$
for any integer $0 \leq i \leq k-1$.

The following proof has been communicated to us by Goulden.
It uses the fact (see, e.g.,~\cite{GJbook}, Ex. 3.3.17) 
that Stirling numbers of the first kind are defined by 
$$\sum_{j \ge 0} \left[  \begin{matrix} k \\ j\\ \end{matrix} \right]
    x^j = (x)^{(k)},\qquad k\ge 0,$$
using the notation for rising factorials $(a)^{(m)} =a(a+1)\cdots (a+m-1)$
for positive integer $m$, and $(a)^{(0)}=1$. Thus we have
    \begin{align*}
        &\sum_{0 \leq j \leq i} \binom{k-i+j}{j+1}(-1)^j \left[ \begin{matrix}
k\\ k-i+j\\ \end{matrix} \right]  \\
        &= - \sum_{-1 \leq j \leq i} \binom{k-i+j}{k-i-1}(-1)^{j+1} \left[
\begin{matrix} k\\ k-i+j\\ \end{matrix} \right] + \left[ \begin{matrix} k\\
k-i-1\\ \end{matrix} \right]  \\
        &= - \sum_{-1 \leq j \leq i} [x^{k-i-1}] (x-1)^{k-i+j} \left[
\begin{matrix} k\\
k-i+j\\ \end{matrix} \right] + [x^{k-i-1}] (x)^{(k)} \\
        &= - [x^{k-i-1}]\left( \sum_{j \geq 0} \left[
\begin{matrix} k\\
j\\ \end{matrix} \right](x-1)^{j} - \sum_{0 \leq j \leq k-i-2}\left[
\begin{matrix} k\\
j\\ \end{matrix} \right](x-1)^{j} \right) + [x^{k-i-1}] (x)^{(k)} \\
	&= -[x^{k-i-1}](x-1)^{(k)} + [x^{k-i-1}] (x)^{(k)} \\
        &=[x^{k-i-1}] (x)^{(k-1)} \Bigl\{ -(x-1)+(x+k-1) \Bigr\} \\
        &= k \left[ \begin{matrix} k-1 \\ k-i-1 \\ \end{matrix} \right],
    \end{align*}
for all $0 \le i \le k-1$, establishing the required identity.
\end{proof}

\subsection{High degree terms of Kerov polynomials for $\deg_1$}
\label{SectRandomYD}

In Corollary \ref{corol:dominant_deg1},
we have given the highest degree term of $K_\mu$ for $\deg_1$.
We shall now describe the next two terms for a one-part partition $\mu=(k)$.

Let $\m_\pi(\mu)$ denotes the monomial symmetric function indexed by $\pi$
evaluated in variables $\mu_1$,$\mu_2$, \dots. For example,
\[\m_{1^2}(\mu)= \sum_{i<j}\mu_i \mu_j.\]
We also introduce the notation
$\tilde{R}_i =
(i-1)R_i$ and $\tilde{R}_{\mu} =
\prod_i\frac{\tilde{R}_i^{m_i(\mu)}}{m_i(\mu)!}$.

\begin{proposition}
\label{prop:TopDegrees}
    For $k \geq 1$, one has
\begin{multline}
K_{(k)} = R_{k+1} + \gamma \frac{k}{2}\sum_{|\mu| = k}(\ell(\mu)-1)!\tilde{R}_\mu +
\\
\sum_{|\mu| = k-1} \left(\frac{1}{4}\binom{k+1}{3} + \gamma^2 k
\frac{3\m_2(\mu) +
4\m_{1^2}(\mu) + 2\m_1(\mu)}{24}\right)\ell(\mu)!\tilde{R}_\mu + \\
\text{terms of
degree less than $k-1$ with respect to $\deg_1$.}
\end{multline}
\end{proposition}

\begin{proof}
Let us write:
$$ K_{(k)} = \sum_{\mu}c_\mu R_\mu.$$
By Proposition \ref{PropBound1}, $c_\mu$ is a polynomial in
$\gamma$ of degree at most $k+1-|\mu|$, hence $c_\mu$ is a polynomial in $\gamma$ of degree at
most $2$ for $|\mu| \geq k-1$.
Moreover, we know explicitly how to express $K_{(k)}$ in terms
of free cumulants for $\alpha \in \{ \frac{1}{2}, 1, 2\}$ (which corresponds to $\gamma \in \{
-\frac{1}{\sqrt{2}}, 0, \frac{1}{\sqrt{2}}\}$).
The case $\alpha=1$ has been solved separately in papers
\cite{GouldenRattan2007,SniadyGenusExpansion},
while the cases $\alpha=1/2$ and $2$ follows from the combinatorial
interpretation given in \cite{NousZonal} and the explicit computation
done in \cite{AgnieskaZonalGenus1}.
\end{proof}

\begin{remark}
One can notice, that the explicit formulas for $c_\mu$ with $|\mu| \geq k$ were
also proved by Lassalle \cite[Theorems 10.2 and 10.3]{Lassalle2009}.
Moreover, our calculations
for $c_\mu$ with $|\mu| = k-1$ are consistent with Lassalle's computer
experiments \cite[p. 2257]{Lassalle2009}, which provide a new evidence to
Conjecture 11.2 of Lassalle \cite{Lassalle2009}.
\end{remark}

\section{Other consequences of the second main result}
\label{app:Matching-Jack_And_Matsumoto}
We present here three consequences of our polynomiality result
for structure constants for Jack characters (see Theorem \ref{theo:struct-const}).
These results were mentioned in the introduction 
(Section \ref{subsect:other_applications}),
but, as they are quite independent of the rest of the paper,
we present them in Appendix.

\subsection{Recovering a recent result of Vassilieva}
\label{sect:Recovering_Katya}
Corollary \ref{CorolConnectionSeries},
which gives a bound on the degree in $\alpha$ of $c_{\mu,\nu;\pi}$,
can be used to give a short proof of a recent result 
of Vassilieva.
In the paper \cite{VassilievaJack}, she considered the following quantity:
for $\mu$ a partition of $n$, let $r=|\mu|-\ell(\mu)$ and
\begin{equation}
    \label{eq:def_Katya}
\bm{a^r_\mu(\alpha)} := \sum_{\lambda \vdash n} \frac{1}{j_\lambda^{(\alpha)}} 
\theta_\mu(\lambda) \left(\theta_{(2,1^{n-2})}(\lambda)\right)^r.
\end{equation}
Using structure constants, we can write: for any partition $\lambda$ of $n$,
\[\left(\theta_{(2,1^{r-2})}(\lambda)\right)^r
=\sum_{\mu^1,\mu^2,\ldots,\mu^r \vdash n \atop \mu^1=(2,1^{r-2})}
\left(\prod_{i=1}^{r-1} c_{\mu^i,(2,1^{r-2});\mu^{i+1}}\right) \theta_{\mu^r}(\lambda).\]
Plugging this into Equation \eqref{eq:def_Katya} and using the orthogonality relation presented in Lemma \ref{lem:identities} \eqref{eq:orthogonality}:
\[\sum_{\lambda \vdash n} \frac{1}{j_\lambda^{(\alpha)}} 
\theta_\mu(\lambda) \theta_\nu(\lambda)
=\frac{\delta_{\mu,\nu}}{z_\mu \alpha^{\ell(\mu)}}\]
(see \cite[Section 3.3]{VassilievaJack}),
we get
\begin{equation}
    \label{eq:Katya_translated}
    \bm{a^r_\mu(\alpha)} = \frac{1}{z_\mu \alpha^{\ell(\mu)}}
\sum_{\mu^1,\mu^2,\ldots,\mu^r \vdash n \atop \mu^1=(2,1^{r-2}), \ \mu^r=\mu}
\prod_{i=1}^{r-1} c_{\mu^i,(2,1^{r-2});\mu^{i+1}}.
\end{equation}
From Corollary \ref{CorolConnectionSeries},
the coefficient $c_{\mu^i,(2,1^{r-2});\mu^{i+1}}$ vanishes unless
\begin{equation}\label{eq:l_mu_increase_by_1}
    |\mu^{i+1}| -\ell(\mu^{i+1}) \le |\mu^i| -\ell(\mu^i) + 1.
\end{equation}
As $|\mu^1|-\ell(\mu^1)=1$ and $|\mu^r|-\ell(\mu^r)=|\mu| - \ell(\mu) = r$,
for any non-zero summand in~\eqref{eq:Katya_translated},
one has equality in~\eqref{eq:l_mu_increase_by_1} for all integers $i$.
But, again from Corollary \ref{CorolConnectionSeries},
equality in~\eqref{eq:l_mu_increase_by_1} implies that the
coefficient $c_{\mu^i,(2,1^{r-2});\mu^{i+1}}$ is independent of $\alpha$.
Hence, the quantity $\alpha^{\ell(\mu)} z_\mu \bm{a^r_\mu(\alpha)}$ is independent on $\alpha$.

In the case $\alpha=1$, it can be interpreted as some number of {\em minimal} factorizations
in the symmetric group (see \cite[Lemma 1]{VassilievaJack} or 
\cite[Proposition 3.1]{GouldenJacksonMatchingJack}),
which has been computed by Dénes in \cite{Denes1959}:
\[z_\mu \bm{a^r_\mu(1)} = \binom{r}{\mu_1-1,\cdots,\mu_{\ell(\mu)-1}} \prod_{i=1}^{\ell(\mu)-1} \mu_i^{\mu_i-2}.\]
Dénes in fact considered only the case $\mu=(n)$, that is minimal factorizations of a cycle,
but it can be easily proved that minimal factorizations of a product of disjoint cycles
are obtained by shuffling factors of minimal factorizations of its cycles.

From the case $\alpha=1$ and the independence on $\alpha$,
we conclude immediately that
\[\bm{a^r_\mu(\alpha)}
= \frac{1}{\alpha^{\ell(\mu)} z_\mu} \binom{r}{\mu_1-1,\cdots,\mu_{\ell(\mu)-1}}
\prod_{i=1}^{\ell(\mu)-1} \mu_i^{\mu_i-2},\]
which is the main result in \cite{VassilievaJack}.

\subsection{Goulden's and Jackson's $b$-conjecture}
\label{sect:b-conj}
In this Section, we explain that our quantities 
$c_{\mu,\nu;\pi}$ (for a general value of the parameter $\alpha$)
are the same as quantities $\bm{c_{\mu,\nu}^\pi(b)}$
considered by Goulden and Jackson in \cite{GouldenJacksonMatchingJack}.
As a consequence, we give a partial answer to a question
raised by these authors.
We use the convention that the boldface quantities
refer to the notations of Goulden and Jackson \cite{GouldenJacksonMatchingJack}.

To establish this connection we will need to use the $\alpha$-scalar
product on the symmetric functions, for which Jack polynomials and
power-sum symmetric functions are orthogonal basis \cite[(VI,10)]{Macdonald1995}.
The following formula is a natural extension of Frobenius counting formula,
see {\em e.g.} \cite[Appendix, Theorem 2]{LandoZvonkin2004}.
\begin{proposition}\label{PropTripleProduit}
    Let $\mu$, $\nu$ and $\pi$ be three partitions of the same integer $n$.
    Then 
    \[c_{\mu,\nu;\pi} = z_\pi \alpha^{\ell(\pi)} \sum_{\lambda \vdash n}
    \frac{\theta_{\pi}(\lambda)\ \theta_{\mu}(\lambda)\ \theta_{\nu}(\lambda)}{
    \langle J_\lambda, J_\lambda \rangle}. \]
\end{proposition}
\begin{proof}
    Let partitions $\mu\vdash n$ and $\nu\vdash n$ be fixed.
    We consider the following symmetric function:
    \[F:=\sum_{\lambda \vdash n} \frac{\theta_\mu(\lambda)\ \theta_\nu(\lambda)}{
    \langle J_\lambda, J_\lambda \rangle} J_\lambda.\]
    By definition of $c_{\mu,\nu;\pi}$, one has:
    \begin{equation}\label{EqTec1}
        F=\sum_{\lambda \vdash n} \sum_{\pi \vdash n} c_{\mu,\nu;\pi}
    \left( \frac{\theta_\pi(\lambda)}{\langle J_\lambda, J_\lambda \rangle} J_\lambda \right).
    \end{equation}
    But $\theta_\pi(\lambda)$ is defined by
    \[J_\lambda = \sum_{\pi \vdash n} \theta_\pi(\lambda)\ p_\pi.\]
    As $p_\pi$ is an orthogonal basis, this implies
    \[ \theta_\pi(\lambda) = \frac{\langle J_\lambda, p_\pi \rangle}{\langle p_\pi, p_\pi \rangle}.\]
    But $J_\lambda$ is also an orthogonal basis, hence:
    \begin{equation}\label{EqJinP}
        p_\pi = \sum_\lambda \frac{\langle J_\lambda, p_\pi \rangle}{
    \langle J_\lambda, J_\lambda \rangle} J_\lambda =
    \langle p_\pi, p_\pi \rangle \sum_\lambda 
    \frac{\theta_\pi(\lambda)}{\langle J_\lambda, J_\lambda \rangle} J_\lambda.
\end{equation}
        Plugging this into \eqref{EqTec1}, one has:
    \[F=\sum_{\pi \vdash n} c_{\mu,\nu;\pi} \frac{p_\pi}{\langle p_\pi, p_\pi \rangle}\]
    and thus,
    \begin{multline*}
        c_{\mu,\nu;\pi} = \langle F,p_\pi \rangle =
        \sum_{\lambda \vdash n} \frac{\theta_\mu(\lambda)\ \theta_\nu(\lambda)}{        
            \langle J_\lambda, J_\lambda \rangle} \langle J_\lambda,p_\pi \rangle \\
           = \sum_{\lambda \vdash n} \frac{\theta_\mu(\lambda)\ \theta_\nu(\lambda)}{        
            \langle J_\lambda, J_\lambda \rangle}
            \langle p_\pi, p_\pi \rangle \theta_\pi(\lambda).
        \end{multline*}
        As $\langle p_\pi, p_\pi \rangle= z_\pi \cdot \alpha^{\ell(\pi)}$,
        we obtain the claimed formula.
\end{proof}

Comparing the proposition with the definition of the connection series 
$\bm{c_{\mu,\nu}^\pi(b)}$ \cite[equations (1) and (5)]{GouldenJacksonMatchingJack},
we get that
\begin{equation}
    c_{\mu,\nu;\pi} = \bm{c_{\mu,\nu}^\pi(b)}.
    \label{EqIdentificationWithGJ}
\end{equation}
Goulden and Jackson had conjectured that they were polynomials
with non-negative integer coefficients in $b=\alpha-1$
(which have
conjecturally a combinatorial meaning in terms of matchings ;
see \cite[Section 4]{GouldenJacksonMatchingJack}).
Corollary~\ref{CorolConnectionSeries} implies the following weaker statement,
which was not known yet.

\begin{proposition}
    The connection series $\bm{c_{\mu,\nu}^\pi(b)}$ introduced in \cite{GouldenJacksonMatchingJack}
    is a polynomial in $b$ with rational coefficients of degree at most $d(\mu,\nu;\pi)$.
\end{proposition}

\subsection{Symmetric functions of contents}\label{SubsectMatsumoto}
In this section we consider a closely related problem
considered by Matsumoto in \cite[Section 8]{MatsumotoOddJM}
in connection with matrix integrals.
Our results allow us to prove two conjectures stated in his paper.

For a box $\Box =(i,j)$ of a Young diagram $\lambda$ ($i$ is the row-index,
$j$ is the column index and $j\le \lambda_i$), we define its
\emph{($\alpha$-)content}
as $c(\Box) = \sqrt{\alpha} (j-1) - \sqrt{\alpha}^{-1} (i-1)$.
The \emph{alphabet of the contents} of $\lambda$ is the multiset 
$\CCC_\lambda=\{c(\Box) : \Box \in \lambda\}$.

Matsumoto \cite[Equation (8.9)]{MatsumotoOddJM} (beware that in his paper the
normalization is different than ours) showed the following remarkable result: for any partition $\lambda$
\begin{equation}
    e_k(\CCC_\lambda) = \sum_{\substack{\mu: \\ |\mu|-\ell(\mu)=k,\\ m_1(\mu)=0}}
    \frac{\Ch_\mu(\lambda)}{z_\mu}.
    \label{EqElemJM}
\end{equation}
In particular, $\lambda \mapsto e_k(\CCC_\lambda)$ is a shifted symmetric function.
Therefore, for any symmetric function $F$, the map $\lambda \mapsto F(\CCC_\lambda)$
is also a shifted symmetric function and one may wonder how it can be expressed in the $\Ch$ basis.
Explicitly, we are interested in the coefficients $a_\mu(F)$ defined by:
\begin{equation}\label{DefAMuF}
    F(\CCC_\lambda) = \sum_{\mu \text{ partition}}
a_\mu(F) \Ch_\mu(\lambda).
\end{equation}

Using the results of Section \ref{SubsectStructPol}, one has the following result:
\begin{proposition}\label{PropDegJM}
    Let $F$ be a symmetric function of degree $d$ and let $\mu$ be a partition.
    The coefficient $a_\mu(F)$ is a polynomial in $\gamma$ of degree at most
    $$d - (|\mu|-\ell(\mu)+m_1(\mu)).$$
\end{proposition}
\begin{proof}
    From~\eqref{EqElemJM}, the proposition is true for $F=e_k$ for any $k\geq 1$.
    Besides, if it is true for two symmetric functions $F_1$ and $F_2$, it is
clearly true for any linear combination
    of them.
    Using Theorem \ref{theo:struct-const}, it is also true for $F_1 \cdot F_2$.
    Since the elementary symmetric functions form a basis of symmetric functions, 
    it follows that the proposition is true for any symmetric function $F$.
\end{proof}

From now on, we use the convention that
the boldface quantities refer to the notation of Matsumoto.
The coefficients $a_\mu(F)$ are closely related to the quantities
$\bm{\mathcal{A}_\mu^{(\alpha)}(F,n)}$ introduced by S.~Matsumoto \cite{MatsumotoOddJM}.
Namely, one has the following lemma (which extends \cite[Lemma 8.5]{MatsumotoOddJM}):
\begin{lemma}
    Let $\mu$ be a partition.
    For $n \geq |\mu|+\ell(\mu)$, let $\pi := \mu + (1^{n-|\mu|})$ be the partition obtained from $\mu$ by adding $1$ to every part and adding new parts equal to $1$.
    Then, for any homogeneous symmetric function $F$ of degree $d$, one has:
    \[\bm{\mathcal{A}_\mu^{(\alpha)}(F,n)} =
    \alpha^{\frac{d-(|\pi|-\ell(\pi))}{2}}
    \left[ \sum_{i \leq m_1(\pi)}\,
        a_{\tilde{\pi} 1^i}(F)\, z_{\tilde{\pi}} \, i!  \, 
            \binom{n-|\tilde{\pi}|}{i} \right],\]
    \label{LemLinkMatsumoto}
    where $\bm{\mathcal{A}_\mu^{(\alpha)}(F,n)}$ is the quantity defined in 
    \cite[Section 8.3]{MatsumotoOddJM}.
\end{lemma}
\begin{proof}
    If we fix the integer $n$, one may rewrite
    Equation~\eqref{DefAMuF} using the definition of $\Ch$:
    \begin{multline*}
    F(\CCC_\lambda) = \sum_{\substack{\nu,\\ |\nu| \leq n}}
    a_\nu(F) \alpha^{\frac{|\nu|-\ell(\nu)}{2}} z_\nu \binom{n-|\nu|+m_1(\nu)}{m_1(\nu)}
    \theta_{\nu 1^{n-|\nu|}}(\lambda) \\
    = \sum_{\pi \vdash n}  \theta_\pi(\lambda)
    \left[\alpha^{\frac{|\pi|-\ell(\pi)}{2}} \sum_{i \leq m_1(\pi)}
    a_{\tilde{\pi} 1^i}(F) z_{\tilde{\pi}} i! 
    \binom{n-|\tilde{\pi}|}{i} \right].
    \end{multline*}
    The notations are the same as in Section~\ref{SubsectProjN}.
    The second equality comes from the fact that each partition $\nu$ of size
    at most $n$ writes uniquely as $\tilde{\pi} 1^i$ where $\pi$ is a partition
    of $n$ and $i$ a non-negative integer smaller or equal to $m_1(\pi)$.
    Let $A_\pi$ denotes the expression in the bracket in the equation above.

    As in the proof of Proposition~\ref{PropTripleProduit} we shall use
    the Jack deformation of Hall scalar product on the space of symmetric functions.
    \[ \sum_{\lambda \vdash n} F(\CCC_\lambda) 
    \frac{J_\lambda}{\langle J_\lambda, J_\lambda \rangle} 
    = \sum_{\lambda,\pi \vdash n} A_\pi \theta_\pi(\lambda) 
    \frac{J_\lambda}{\langle J_\lambda, J_\lambda \rangle} 
    = \sum_{\pi \vdash n} A_\pi \frac{p_\pi}{\langle p_\pi,p_\pi \rangle}.\]
    The last equality corresponds to \eqref{EqJinP}.
    We deduce that
    \[A_\pi = \left\langle \sum_{\lambda \vdash n} F(\CCC_\lambda)             
        \frac{J_\lambda}{\langle J_\lambda, J_\lambda \rangle},
        p_\pi \right\rangle =
        \sum_{\lambda \vdash n} F(\CCC_\lambda) \frac{\theta_\pi(\lambda)
        \cdot \langle p_\pi,p_\pi \rangle
        }{\langle J_\lambda, J_\lambda \rangle}.\]
    This formula coincides with the definition of $\bm{\mathcal{A}_\mu^{(\alpha)}(F,n)}$
    in \cite[paragraph 8.3]{MatsumotoOddJM} up to a scalar multiplication, namely,
    \[ \bm{\mathcal{A}_\mu^{(\alpha)}(F,n)} = \alpha^{d/2} A_\pi.\]
    The only difficulty is the difference of notations.
    To help the reader, we provide the following dictionary.
    First recall that $\langle p_\pi,p_\pi \rangle = z_\pi\ \alpha^{\ell(\pi)}$.
    Then our partition $\pi$ corresponds to $\bm{\mu + (1^{n-|\mu|})}$.
    In particular, one has  $\bm{|\mu|}=|\pi|-\ell(\pi)$ and $\bm{z_{\mu + (1^{n-|\mu|})}}=z_\pi$.
    Besides, $F(\CCC_\lambda)$ in our paper corresponds to
    $\bm{\alpha^{d/2}\ F(A_\lambda^{\alpha})}$ in \cite{MatsumotoOddJM}.
    Finally, the probability $\bm{\PP_n^{(\alpha)}(\lambda)}$ is simply given
    by $\frac{n! \alpha^n}{\langle J_\lambda, J_\lambda \rangle}$.
\end{proof}

Proposition~\ref{PropDegJM}, when translated into Matsumoto's notation by
Lemma~\ref{LemLinkMatsumoto}, has several interesting consequences.
As above, we consider an homogeneous symmetric function $F$ of degree $d$.
\begin{itemize}
    \item If $d=|\mu|$, the only term of the sum which can be non-zero
        corresponds to $i=0$ (by Proposition~\ref{PropDegJM}).
        Moreover, it does not depend on $\alpha$.
        Besides, the exponent of $\alpha$ in the formula is equal to zero.
        Finally, $\bm{\mathcal{A}_\mu^{(\alpha)}(F,n)}$ does not depend
        neither on $\alpha$ nor on $n$, which proves \cite[Conjecture 9.2]{MatsumotoOddJM}.
    \item If $d=|\mu|+1$, there are only two terms
        (corresponding to $i=0,1$) which can be
        non-zero in the sum.
        Besides, the coefficient $a_{\tilde{\pi} 1}$ does not depend on $\alpha$
        because of Proposition~\ref{PropDegJM}.
        But it is easy to prove that it is equal to $0$ in the case $\alpha=1$
        (it comes from the combinatorial interpretation of $\bm{\mathcal{A}_\mu^{(1)}(F,n)}$,
        see \cite[Example 9.2]{MatsumotoOddJM}).
        Hence, $a_{\tilde{\pi} 1}=0$ and only the term corresponding to $i=0$
        is non-zero.
        In particular, one can see that $\bm{\mathcal{A}_\mu^{(\alpha)}(F,n)}$
        does not depend on $n$, which proves \cite[Conjecture 9.3]{MatsumotoOddJM}.
    \item In the general case, non-zero terms of the sum are indexed by 
        values of $i$ smaller or equal to $d - |\mu|$ (by Proposition~\ref{PropDegJM}).
        Hence $\bm{\mathcal{A}_\mu^{(\alpha)}(F,n)}$ is a polynomial in $n$ of degree
        at most $d - |\mu|$.
        This result is not stronger that the bound of Matsumoto on the
        degree of $\bm{\mathcal{A}_\mu^{(\alpha)}(F,n)}$ \cite[Theorem 8.8]{MatsumotoOddJM}.
        Nevertheless, it is better in some cases 
        and we also have some control on the dependence on $\alpha$
        (as illustrated by the proofs of the
        conjectures above).
\end{itemize}
\section*{Acknowledgments}

We thank I.~Goulden and P.~\'Sniady for enlightening discussions on the subject
and relevant pieces of advice for the presentation of the paper.

\bibliographystyle{alpha}

\bibliography{biblio2011}
\end{document}